\documentclass[12pt]{article}
\usepackage{amsmath,amsthm,amssymb}
\usepackage[mathscr]{euscript}
\usepackage{verbatim}
\usepackage{dsfont,bm}
\usepackage{stmaryrd} %llbracket.
\usepackage{color}
\usepackage{mathtools}
\usepackage{graphicx}
\usepackage{enumerate}
\usepackage{esint} 
\usepackage{mathabx}
\allowdisplaybreaks[4]

\usepackage{xcolor}
\definecolor{darkgreen}{rgb}{0,0.75,0}
\definecolor{darkred}{rgb}{0.75,0,0}
\definecolor{darkmagenta}{rgb}{0.5,0,0.5}
\usepackage[colorlinks,citecolor=darkgreen,linkcolor=darkred,urlcolor=darkmagenta]{hyperref}
\newtheorem{theorem}{Theorem}[section]
\newtheorem{cor}[theorem]{Corollary}

\newtheorem{lemma}[theorem]{Lemma}
\newtheorem{prop}[theorem]{Proposition}

\theoremstyle{definition}
\newtheorem{definition}[theorem]{Definition}
\newtheorem{question}[theorem]{Question}

\newtheorem{remark}[theorem]{Remark}
\newtheorem{example}[theorem]{Example}

\newtheorem{notation}[theorem]{Notation}
\newtheorem{conjecture}[theorem]{Conjecture}

\numberwithin{equation}{section}

\def\be{\begin{equation}}
	\def\ee{\end{equation}}
\def\bes{\begin{equation*}}
	\def\ees{\end{equation*}}

\newcommand{\set}[1]{\left\{ #1 \right\}}

\newcommand{\abs}[1]{{\left\lvert #1\right\rvert}}
\newcommand\norm[1]{\left\lVert #1\right\rVert} %norm
\newcommand{\one}{\mathds{1}} %indicator

\DeclareMathOperator*{\esssup}{ess\,sup}
\DeclareMathOperator*{\essinf}{ess\,inf}

\def\to {\rightarrow}

\def\q{\quad} 
\def\dint{\int\kern-.6em\int}
\def\grad{\nabla}

\newcommand\restr[2]{{% we make the whole thing an ordinary symbol
		\left.\kern-\nulldelimiterspace % automatically resize the bar with \right
		#1 % the function
		\vphantom{\big|} % pretend it's a little taller at normal size
		\right|_{#2} % this is the delimiter
}} %restriction of a function
\def\diam{{\mathop{{\rm diam }}}}
\def\dist{{\mathop {{\rm dist}}}}

\def\supp{\mathop{{\rm supp}}\nolimits}

\newcommand{\on}[1]{\operatorname{ #1}}

\def\wt{\widetilde}

\def\be{\begin{equation}}
	\def\ee{\end{equation}}
\def\bes{\begin{equation*}}
	\def\ees{\end{equation*}}
\def\ba{\begin{align}}
	\def\ea{\end{align}}
\def\xxea{\end{align}}
\def\bas{\begin{align*}}
\def\eas{\end{align*}}

\definecolor{dgreen}{rgb}{0, 0.6, 0.1}
\definecolor{dblue}{rgb}{0, 0.0, 0.6}
\definecolor{vdblue}{rgb}{0,.08, 0.45}
\definecolor{dred}{rgb}{0.7, 0.0, 0.0}
\definecolor{vdblue}{rgb}{0,.08, 0.45}

\definecolor{purple}{rgb}{0.6, 0.0, 0.6}
\definecolor{mytext}{rgb}{0.1, 0.1, 0.1}

\newcommand{\gener}{\mathcal{A}} %generator of a DF
\newcommand{\generlaak}{\mathcal{A}^{\laak}} 
\newcommand{\branch}{\mathbf{b}}
\newcommand{\glue}{\mathbf{g}}
\newcommand{\tree}{\mathcal{T}} %tree
\newcommand{\metrictree}{\mathsf{d}_{\tree(\branch)}}
\newcommand{\meastree}{\mathsf{m}_{\tree(\branch)}}
\newcommand{\lentree}{\lambda_{\tree(\branch)}}
\newcommand{\codetree}{\chi} %coding map for tree
\newcommand{\ultra}{\mathcal{U}} %ultrametric space
\newcommand{\metricultra}{\mathsf{d}_{\ultra(\glue)}}
\newcommand{\measultra}{\mathsf{m}_{\ultra(\glue)}}
\newcommand{\laak}{\mathcal{L}} %laakso type space
\newcommand{\metlaak}{\mathsf{d}_{\laak(\glue,\branch)}}
\newcommand{\measlaak}{\mathsf{m}_{\laak(\glue,\branch)}}
\newcommand{\lenlaak}{\lambda_{\laak(\glue,\branch)}}
\newcommand{\prspace}{\mathcal{P}(\glue,\branch)} %laakso type space
\newcommand{\metprod}{\mathsf{d}_{\prspace}}
\newcommand{\measprod}{\mathsf{m}_{\prspace}}
\newcommand{\worm}{\mathcal{W}} %wormholes
\newcommand{\quolaak}{\mathcal{Q}}
\newcommand{\projtree}{\pi^{\tree}}
\newcommand{\contfunc}{C} %continuous function
\newcommand{\dftree}{\mathcal{E}^{\tree}}
\newcommand{\domtree}{\mathcal{F}^{\tree}}
\newcommand{\emtree}{\Gamma^{\tree}}
\newcommand{\dflaak}{\mathcal{E}^{\laak}}
\newcommand{\domlaak}{\mathcal{F}^{\laak}}
\newcommand{\emlaak}{\Gamma^{\laak}}
\newcommand{\sttree}{\Psi_\branch}

\def\Cap{\operatorname{Cap}} %Capacity
\newcommand{\projultra}{\pi^\ultra}

\newcommand{\projlaak}{\pi^\laak}

\begin{document}
	
	\font\titlefont=cmbx14 scaled\magstep1
	\title{\titlefont Diffusions and random walks with prescribed sub-Gaussian heat kernel estimates} %Other possibilities: Spaces and diffusions with prescribed volume growth and space-time scaling    
	\author{Mathav Murugan} 
	\renewcommand{\thefootnote}{}
	\footnotetext{The author is partially supported by NSERC and the Canada research chairs program.}
	\renewcommand{\thefootnote}{\arabic{footnote}}
	\setcounter{footnote}{0}
 
	\maketitle
	\vspace{-0.5cm}
%\ver
\begin{abstract}
	Given suitable functions $V, \Psi:[0,\infty) \to [0,\infty)$, we obtain necessary and sufficient conditions on $V,\Psi$ for the existence of a metric measure space and a symmetric diffusion process that satisfies sub-Gaussian heat kernel estimates with volume growth profile $V$ and escape time profile $\Psi$. We prove sufficiency by constructing a new family of diffusions. Special cases of this construction also leads to a new family of infinite graphs whose simple random walks satisfy sub-Gaussian heat kernel estimates with prescribed volume growth and escape time profiles. In particular, these random walks on graphs generalizes earlier results of Barlow who considered the case $V(r)=r^\alpha$ and $\Psi(r)=r^\beta$ (Rev Mat Iberoam 2004). The family of diffusions we construct have martingale dimension one but can have arbitrarily high spectral dimension. Therefore our construction shows the impossibility of obtaining non-trivial lower bounds on martingale dimension in terms of spectral dimension which is in contrast with upper bounds on martingale dimension in terms of spectral dimension, as obtained by Hino (Probab Theory Relat Fields 2013).
\end{abstract}

\section{Introduction}

Following the seminal work of Barlow and Perkins,  sub-Gaussian estimates on heat kernel have been established for diffusions on many fractals  \cite{BP}.
It is now well-established that on a large class of spaces there is a  nice diffusion process $\{X_{t}\}_{t\geq 0}$  symmetric with respect to a suitable reference measure $m$ and exhibits   sub-diffusive
behavior in the sense that its transition density (heat kernel) $p_{t}(x,y)$ satisfies
the following \textbf{sub-Gaussian estimates}:
\begin{align}\label{e:HKEbeta}
	p_{t}(x,y) &\le \frac{c_{1}}{m(B(x,t^{1/\beta}))} \exp\biggl( - c_{2} \Bigl( \frac{d(x,y)^{\beta}}{t} \Bigr)^{\frac{1}{\beta-1}}\biggr) \nonumber
	\\
	p_{t}(x,y)	&\ge \frac{c_{3}}{m(B(x,t^{1/\beta}))} \exp\biggl( - c_{4} \Bigl( \frac{d(x,y)^{\beta}}{t} \Bigr)^{\frac{1}{\beta-1}}\biggr)
\end{align}
for all points $x,y$ and all $t>0$, where $c_{1},c_{2},c_{3},c_{4}>0$ are some constants,
$d$ is a natural metric on the space, $B(x,r)$ denotes the open ball of radius $r>0$ centered at $x$, and $\beta\geq 2$ is a characteristic of the diffusion called the
\textbf{walk dimension} or \textbf{escape time exponent}. 
This result was obtained first for the Sierpi\'{n}ski gasket in
\cite{BP}, then for nested fractals in \cite{Kum93}, for affine nested fractals in \cite{FHK}
and for Sierpi\'{n}ski carpets in \cite{BB92,BB99}. We refer to \cite{CQ,Kig23} for some very recent examples that have fewer symmetries but are low dimensional in a certain sense. We refer to \cite{Bar98,Kig01} for an  introduction to diffusion on fractals.

In all of the examples mentioned above, there exists $C \in (1,\infty), \alpha \in [1,\infty)$ such that 
\begin{equation}
	\label{e:alphavol}
	C^{-1}r^\alpha \le	m(B(x,r)) \le C r^\alpha
\end{equation}
for all points $x$ and for all radii $r$ less than the diameter of the space.
Here  the parameter $\alpha \in [1,\infty)$ is called the \textbf{volume growth exponent} since it governs the volume of balls with respect to the reference measure $m$. The term \emph{escape time exponent} for $\beta$ in \eqref{e:HKEbeta} arises from the fact that  \eqref{e:HKEbeta} implies that the expected exit time  $\mathbb{E}_x[\tau_{B(x,r)}]$ from a ball $B(x,r)$ for the diffusion started at the center $x$ is comparable to $r^\beta$ for all balls $B(x,r)$ whose radii is less than $c$ times the diameter of the space, where $c \in (0,1)$ (see Definition \ref{d:conseqhke}(d) and the discussion above it).

More generally, the  sub-Gaussian heat kernel estimates also occur when the volume growth and expected exit times are not described by power functions. For example, this is the case for scale-irregular Sierpi\'nski gaskets \cite{Ham92, BH, Ham00}. In this case, it is natural to replace volume growth and escape time \emph{exponents} with \emph{profiles} (or functions). These profiles are \emph{doubling functions} which can be viewed as a generalization of power functions mentioned above. We say that a function $V:[0,\infty) \to [0,\infty)$ is \emph{doubling} if $V$ is a homeomorphism (hence increasing) and there exists $C>1$ such that 
\begin{equation} \label{e:doublef}
	V(2r) \le C V(r) \quad \mbox{for all $r>0$.}
\end{equation}
We say that a metric space $(X,d)$ equipped with a measure $m$ has \emph{volume growth profile} if there exists $C>1, c \in (0,1)$ such that 
\[
C^{-1} V(r) \le m(B(x,r)) \le C V(r),\quad \mbox{for all $x \in X, 0<r<\diam(X,d)$,}
\]
where $\diam(X,d):= \sup_{x,y \in X} d(x,y)$ denotes the diameter of the metric space $(X,d)$. 
Similarly, the \emph{escape time profile} for a symmetric diffusion process is a doubling function $\Psi: [0,\infty) \to [0,\infty)$ such that 
\[
C^{-1} \Psi(r) \le \mathbb{E}_x[\tau_{B(x,r)}] \le C \Psi(r),\quad \mbox{for all $x \in X, 0<r<\diam(X,d)$.}
\]
The version of sub-Gaussian heat kernel estimates \eqref{e:HKEbeta} corresponds to the escape time profile $\Psi(r)=r^\beta$. The generalization of \eqref{e:HKEbeta} for a more general escape time profile $\Psi:[0,\infty) \to [0,\infty)$ is given in Definition \ref{d:HKE}. We refer to Remark \ref{r:profile} for some examples that are not power functions.

%There are several characterizations of sub-Gaussian heat kernel estimates in terms of functional inequalities like Poincar\'e and Faber-Krahn inequalities, Harnack inequalities, estimates on expected exit times and capacity estimates \cite{AB,BBK,BB04,GHL, GT}. We review some of these characterizations later in this work.
%Much of the  literature on sub-Gaussian heat kernel estimates concerns either obtaining them for specific examples of interest or to obtain general characterizations in terms of other properties. 

Our work is closely related to  
earlier results of M.~Barlow who considered an \emph{inverse problem} that we describe below \cite{Bar04}.
The setting of \cite{Bar04} is random walks on graphs where variants of \eqref{e:HKEbeta} in a discrete-time setting are known to hold  for random walks on graphs (see Definition \ref{d:dHKE}). Well-known estimates on random walks imply that if sub-Gaussian heat kernel estimates with volume growth exponent $\alpha$ and escape time exponent $\beta$ hold, then we necessarily have  (see \cite[Theorem 1]{Bar04})
\begin{equation} \label{e:condrw}
	2 \le \beta \le \alpha+1. 
\end{equation}
% We note that in the discrete setting the estimates on volumes and expected escape time are  imposed only for balls of radii at least one. 
Under condition \eqref{e:condrw}, Barlow constructs an infinite graph with    volume growth exponent $\alpha$, escape time exponent $\beta$  satisfying the discrete-time sub-Gaussian heat  kernel estimates for the simple random walk on graph \cite[Theorem 2 and Lemma 1.3(a)]{Bar04}. This result can be viewed as solving an \emph{inverse problem}: the construction of a graph with prescribed sub-Gaussian heat kernel estimates.

Prior to our work, the solution to an analogous inverse problem was not known if one replaces random walks with diffusions, or if volume growth and escape time exponents are replaced with the corresponding profiles, as previously discussed. This motivates the following question, which the author learned from Barlow \cite{Bar22}.
\begin{question} \label{q:main}
	For which  doubling functions $V:[0,\infty) \to [0,\infty)$ and $\Psi:[0,\infty) \to [0,\infty)$ there exist a metric measure space $(X,d,m)$  and a $m$-symmetric diffusion
	process on  $X$ that satisfies sub-Gaussian heat kernel estimates with volume growth profile
	$V$ and escape time profile $\Psi$? If such a diffusion exists, how to construct a corresponding diffusion and the underlying metric
	measure space?
\end{question}

Our main results answer Question \ref{q:main} (see Theorems \ref{t:mainnec} and \ref{t:mainsuf}).  For a  diffusion  that satisfies sub-Gaussian heat kernel estimates with volume growth profile
$V$ and escape time profile $\Psi$ to exist on a metric space $(X,d)$, we show that the following condition is both necessary and sufficient: there exists $C>1$ such that 
\begin{equation} \label{e:conddiff}
	C^{-1} \frac{R^2}{r^2} \le \frac{\Psi(R)}{\Psi(r)} \le C \frac{RV(R)}{rV(r)}, \quad \mbox{for all $0<r<R< \diam(X,d)$.}
\end{equation}
We note that if $V(r)=r^\alpha, \Psi(r)=r^\beta$, then \eqref{e:conddiff} is equivalent to \eqref{e:condrw} as stated in \cite[Theorem 1]{Bar04}. 
Hence our work verifies  that the condition \eqref{e:condrw} plays the same role in the continuous time setting as well. This was long believed to be case as the authors of \cite[p. 2069]{GHL03} write 
``\emph{There	is no doubt that the same is true for continuous time heat kernels}". 

Our construction produces fractal-like spaces and graphs in a  unified framework and hence we also answer the discrete version of Question \ref{q:main} for random walks on graphs (see Theorem \ref{t:sufgraph}(a)). As explained above, since \eqref{e:conddiff} is a generalization of \eqref{e:condrw}, our results can be viewed as a generalization of \cite{Bar04}. The necessity of first estimate in \eqref{e:conddiff} was shown by the author in \cite[Corollary 1.10]{Mur20} but we provide a new shorter proof using capacity bounds (cf.  Lemma \ref{l:dwge2}). The necessity of the second estimate in \eqref{e:conddiff} follows from an argument using Poincar\'e inequality, capacity bounds along with the chain condition for the metric obtained in \cite[Theorem 2.11]{Mur20} (cf. Lemma \ref{l:dwdf}).

Before discussing our construction proving the sufficiency of \eqref{e:conddiff},
let us briefly mention further motivations behind Question \ref{q:main} and its variants.  One of the most important motivations is to construct a rich class of examples, as they provide testing grounds for various questions and methods (for example, see Question \ref{q:dcw}). By the necessity of \eqref{e:conddiff}, our construction encompasses all possible volume growth and escape time profiles.
Certain predictions in the physics literature (for example, \cite{WF}) rely on treating dimension as a continuous parameter. We hope that the spaces we construct will provide a rigorous setting for verifying some of these predictions. Additionally, our construction yields growing sequences of finite graphs with uniform sub-Gaussian heat kernel estimates (see Theorem \ref{t:sufgraph}(b)). Such examples are relevant for the study of hitting times, mixing times and cutoff phenomenon on finite Markov chains \cite{DKN, SW25}. We anticipate that the new heat kernel estimates developed here will lead to new phenomena for the corresponding random walks and diffusions.

Similar to Barlow's work \cite{Bar04}, our construction of the metric measure space proving the sufficiency of \eqref{e:conddiff} can be viewed as a variant of Laakso space \cite{Laa}.
We  note that variants of Laakso's construction are of recent interest in various contexts such as the study of conformal dimension, Loewner spaces, Poincar\'e inequalities and metric embeddings \cite{CE,CK,AE, OO}. Similar to \cite{Bar04}, the construction consists of two steps that we outline below.
\begin{enumerate}[(1)]
	\item  In the first step, we construct an $\mathbb{R}$-tree equipped with a measure whose volume growth profile is comparable to the function $r \mapsto \Psi(r)/r$. This tree admits a natural symmetric diffusion process \cite{Kig95, AEW}. Using standard results on resistance forms, we know that this diffusion satisfies sub-Gaussian heat kernel estimates whose escape time profile is $\Psi$ \cite{Kum04,Kig12}. The first step only depends on the exit time profile $\Psi$. 
	\item If $V$ happens to be comparable to the function $r \mapsto \frac{\Psi(r)}{r}$, then we stop after the first step. In general,  the tree would not have the desired volume growth profile.
	In the second step, we take (typically infinitely) many disjoint copies of the tree constructed in the first step and identify (or glue) carefully chosen points in different copies of the tree so as to increase the volume growth profile without affecting the escape time profile. 	The identification is done in such a way that the  preimage of the corresponding quotient map of any point on the quotient space  is finite. This leads to a quotient space which we call the \textbf{Laakso-type space}. We define a natural diffusion on the Laakso-type space such that it behaves like the diffusion on the tree and is equally likely to proceed in one of the identified copies of the tree in case the diffusion hits a point on the quotient space that corresponds to multiple copies of the tree.
\end{enumerate}

We use the theory of Dirichlet forms to construct the diffusion process in both    steps mentioned above. This theory also plays an important role in the analysis of the corresponding diffusion processes. In particular, our proof of sub-Gaussian heat kernel estimates for the Laakso-type space relies on a general characterization in terms of Poincar\'e inequality and cutoff energy inequality \cite[Theorem 1.2]{GHL}. Our proof of Poincar\'e inequality is an adaptation of the \emph{pencil of curves} approach in Laakso's work \cite[Definition 2.3 and Proof of Theorem 2.6]{Laa}. 
For the proof of the cutoff energy inequality, we use cutoff functions on tree to construct cutoff functions on the Laakso-type space.
Our proof of cutoff energy inequality relies on a recent simplification of the inequality in \cite[Lemma 6.2]{Mur24} following \cite{AB,BM}.
Although cutoff energy inequality was introduced two decades ago in \cite{BB04}, it has almost exclusively been used to obtain abstract characterizations and stability results for Harnack inequalities and heat kernel estimates \cite{BB04, BBK,GHL,AB, BM}. This work along with another very recent work \cite{Mur24} shows that cutoff energy inequalities can also be useful to obtain heat kernel estimates for concrete diffusions of interest.

There are challenges associated with constructing a diffusion process (or Dirichlet form) on Laakso-type spaces that do not arise in the discrete setting of \cite{Bar04}. A key insight from Barlow and Evans \cite{BE} is that Laakso  spaces can be viewed as projective limits of simpler spaces, where defining corresponding diffusions is more straightforward. These diffusions on approximating spaces were then used to construct a diffusion on the projective limit, as demonstrated in \cite[Theorem 4.3]{BE}. While it might be possible to adapt their arguments to construct a diffusion, we prefer the Dirichlet form approach as it provides additional tools to analyze the diffusion process such as the functional inequalities mentioned in the previous paragraph.

Our construction of the Dirichlet form is an analytic counterpart to that of \cite{BE}. Specifically, we approximate the Laakso-type space by a sequence of simpler Laakso-type spaces, in which at most countably many trees are glued together. In these approximating spaces, the points where non-trivial identifications occur form a separated subset of the tree.  For these simpler  Laakso-type   spaces, the definition of Dirichlet forms is easier since it can be described directly using the Dirichlet form for the diffusion on the tree (Proposition \ref{p:dflaakprelim}). 
We then define the Dirichlet form on the projective limit by constructing the limiting generator, using the Friedrichs extension theorem (see Proposition \ref{p:friedrich}).

Although the overall approach behind our work is similar to that in \cite{Bar04}, every step of our construction and the proof of sub-Gaussian bounds differs from that in \cite{Bar04}. We highlight some of the key differences below.
\begin{remark} \label{r:diff}
	\begin{enumerate}[(1)]
		\item The construction of  trees in \cite[\textsection 4]{Bar04} is a variant of the \emph{Vicsek set} while our tree is a variant of the \emph{continuum self-similar tree} recently studied by Bonk and Tran \cite{BT21} (see Example \ref{x:csst}).  The latter tree seems slightly simpler to construct and analyze compared to the Vicsek tree based approach. Following the terminologies of \cite{BH, BT21} our tree can be viewed as a \emph{scale-irregular continuum self-similar tree}.
		\item For the construction of the Laakso-type space from the tree, Barlow uses certain sets with good separation properties \cite[Proposition 2.1]{Bar04} on the tree constructed in the first step. Such sets are called \emph{wormholes} in \cite{Laa} and are used to glue points in different copies of the tree. The  existence of such sets in \cite[proof of Proposition 2.1]{Bar04} relies implicitly on the axiom of choice (in its equivalent form of Zorn's lemma).  In contrast, our construction of wormholes has the advantage of being \emph{constructive}; that is, explicit and independent of the axiom of choice (see \eqref{e:defwormhole} and the definition of Laakso-type space in \textsection \ref{ss:laak}). 
		\item In both steps of our construction, the volume growth profile need not be a power function, unlike the corresponding construction in \cite{Bar04}. See Remark \ref{r:profile} for a family of such examples.
		\item The proof of sub-Gaussian heat kernel estimates in \cite{Bar04} relies on obtaining bounds on expected exit times from balls and the elliptic Harnack inequality. Our approach differs,  as we rely on deriving  Poincar\'e  and cutoff energy inequalities as mentioned above (see Propositions \ref{p:pi} and \ref{p:cs}). In particular, we provide a new proof of  \cite[Theorem 2]{Bar04}.  One difficulty in extending the method of \cite{Bar04} is that the proof of elliptic Harnack inequality relies on the fact that the space is a graph which is no longer true in our setting.  Furthermore, our method has the advantage of being easily extended to non-linear setting of $p$-energies \cite{Yan}.
	\end{enumerate}
\end{remark}

Our work sheds new light on the relationship between spectral and martingale dimensions for diffusions. Hino \cite{Hin13} shows that the martingale dimension is bounded from above by the spectral dimension for diffusions on self-similar sets. This bound  is sharp in certain cases such as the   Brownian motion on $\mathbb{R}^n$. We conjecture that this upper bound of Hino is a general phenomenon that should not require the space 
to be self-similar. On the other hand, our construction shows that the spectral dimension could be arbitrarily high for spaces with martingale dimension one. This shows the impossibility of obtaining non-trivial lower bound on martingale dimension that only depends on the spectral dimension.
In particular, our work shows that sub-Gaussian heat kernel bounds are not enough to obtain non-trivial lower bounds on martingale dimension. Obtaining non-trivial lower bounds or determining the martingale dimension for concrete diffusions on fractals such as Brownian motion on the generalized (high dimensional) Sierpinski carpets remains a challenging open problem \cite{Hin08, Hin10, Hin13}.
\subsection{Related results: flexibility versus rigidity}
A key takeaway from this work is that sub-Gaussian heat kernel estimates are flexible as the only essential constraint is \eqref{e:conddiff}. However as we illustrate below there are several closely related situations that exhibit a contrasting rigidity.

The first rigidity phenomenon concerns random walks on graphs.
The volume growth  and expected exit time behaviours we consider are independent of the center of the ball and depends only on the radius. Therefore it is tempting to ask the following variant of Question \ref{q:main} by imposing an additional requirement that the graph be transitive.
\begin{question} \label{q:trans}
	For which  doubling functions $V:[1,\infty) \to [1,\infty)$ and $\Psi:[1,\infty) \to [1,\infty)$ there exist an infinite \emph{transitive}  graph such that the corresponding simple random walk
	satisfies sub-Gaussian heat kernel estimates with volume growth profile $V$ and escape time profile $\Psi$? If so, how to construct such graphs?
\end{question}
The answer to this question is known.
Despite the flexibility in the sub-Gaussian heat kernel  behavior of random walks on graphs, the transitivity condition imposes severe restrictions on $V$ and $\Psi$. It turns out that the answer to Question \ref{q:trans} is given by (cf. \eqref{e:conddiff})
\begin{equation*} \label{e:trans}
	V(r)= r^n \quad \mbox{where $n \in \mathbb{N}$, and} \quad \Psi(r)=r^2.
\end{equation*}
All such examples are quasi-isometric to groups of polynomial volume growth. 
These statements follow from deep results on the structure of groups of polynomial growth \cite{Gro} and transitive graphs of polynomial growth \cite{Tro} along with known heat kernel estimates for such graphs \cite{HS}.

Next, we discuss a rigidity result for heat kernel in the continuous time setting.
The most classical example of heat kernel is that of the Brownian motion in $\mathbb{R}^n$ given by
\[
p_t(x,y)= \frac{1}{(2\pi t)^{n/2}} \exp\left(-\frac{d(x,y)^2}{2t}\right) \quad \mbox{for all $x,y \in \mathbb{R}^n$ and $t>0$.}
\]
This is an example of sub-Gaussian heat kernel estimates with volume profile $V(r)=r^n$ and $\Psi(r)=r^2$. A remarkable   rigidity result of Carron and Tewodrose shows that if a $m$-symmetric diffusion process on a metric measure space $(X,d,m)$ has an \emph{Euclidean-like} heat kernel 
\begin{equation*} \label{e:hkealpha}
	p_t(x,y)= \frac{1}{(2\pi t)^{\alpha/2}} \exp\left(-\frac{d(x,y)^2}{2t}\right) \quad \mbox{for all $x,y \in X$ and $t>0$,}
\end{equation*}
where $\alpha \in [1,\infty)$, then $\alpha=n$ for some $n \in \mathbb{N}$, $X=\mathbb{R}^n$ equipped with Lebesgue measure $m$, Euclidean metric $d$ and the diffusion process is the Brownian motion on $\mathbb{R}^n$ \cite[Theorem 1.1]{CT}.  This rigidity result is in sharp contrast to our results that imply the existence of diffusions with (sub-)Gaussian heat kernel bounds with volume growth profile $V(r)=r^\alpha$ and escape time profile $\Psi(r)=r^2$ for any $\alpha \in [1,\infty)$. The authors also obtain a quantitative version of this rigidity result \cite[Theorem 1.2]{CT} and an analogous result for the Brownian motion on the $n$-sphere $\mathbb{S}^n$ \cite[Theorem 7.1]{CT}.

\subsection{Outline of the work}
In \textsection \ref{s:framework}, we recall the   setting of metric measure space with an associated Dirichlet form and state the main result. In \textsection \ref{s:mms}, we construct   trees and Laakso-type spaces as metric measures spaces and prove some of their   properties. In \textsection \ref{s:df}, we define Dirichlet forms on the metric measure spaces constructed in \textsection \ref{s:mms}. In \textsection \ref{s:hke}, we obtain heat kernel estimates for the Laakso-type spaces and show that they always have martingale dimension  one. These heat kernel estimates are enough to show the sufficiency of \eqref{e:conddiff}. We   show how these constructions of Laakso-type spaces lead to constructions of graphs and growing family of graphs with prescribed heat kernel estimates. 
Finally, we propose several questions that arise from this work.

%The tree   $\tree(\branch)$ is determined by a \emph{branching function} $\mathbf{b}:\mathbb{Z} \to \mathbb{Z}$ such that $2 \le \inf_{k \in \mathbb{Z}} \branch(k) \le \sup_{k \in \mathbb{Z}} \branch(k)< \infty$. 
%Let us informally describe this construction. 
%For $n \in \mathbb{N}$ we start with an isometric copy of $[0,2^n]$ which we think of a metric graph $T_{n,n}$  with two vertices and an edge of length $2^n$. This edge of length $2^n$ is replaced by a metric graph of corresponding to the complete bipartite graph $K_{1,\branch(n-1)}$ 
%where each of the $\branch(-1)$-edges is isometric to an interval of length $2^{n-1}$ and the the two original vertices are isometrically embedded to two vertices of the partition containing $\branch(-1)$ vertices. This yields the $\mathbb{R}$-tree $T_{n-1,n}$. Now each edge of length $2^{n-1}$ in $T_{n-1,n}$ is replaced by a a metric graph of corresponding to the complete bipartite graph $K_{1,\branch(n-2)}$ 
%where each edge is isometric to an interval of length $2^{n-2}$ to obtained $T_{n-2,0}$ as before so that $T_{n-1,n}$ embeds isometrically in $T_{n-2,n}$.
%It turns out that the spaces $T_{n-m,n}$ converge in the Gromov-Hausdorff sense as $m \to \infty$. 
%The limiting space $T_{-\infty,n}$ is a compact $\mathbb{R}$-tree of diameter $2^n$. By a suitable choice of base points on $T_{-\infty,n}$, the  (***add figure)

\begin{notation}
	Throughout this paper, we use the following notation and conventions.
	\begin{enumerate}[(i)]
		\item For $a<b$, we write $\llbracket a,b \rrbracket = [a,b] \cap \mathbb{Z}$.
		\item The cardinality (the number of elements) of a set $A$ is denoted by $\#A$.
		\item  We write
		$a\vee b:=\max\{a,b\}$, $a\wedge b:=\min\{a,b\}$.
		\item Let $X$ be a non-empty set. We define $\one_{A}=\one_{A}^{X}\in\mathbb{R}^{X}$ for $A\subset X$ by
		\[\one_{A}(x):=\one_{A}^{X}(x):= \begin{cases}
			1 & \mbox{if $x \in A$,}\\
			0 & \mbox{if $x \notin A$.}
		\end{cases} \]
		
		\item We use the notation  $A \lesssim B$ for quantities $A$ and $B$ to indicate the existence of an
		implicit constant $C \ge 1$ depending on some inessential parameters such that $A \le CB$. We write $A \asymp B$, if $A \lesssim B$ and $B \lesssim A$.
		
		\item  For a function $f : X \to \mathbb{R}$ and $A \subset X$, we denote the restriction of $f$ to $A$ by $\restr{f}{A}$.
	\end{enumerate}
\end{notation}

\section{Framework and main results} \label{s:framework}

We recall basic notions in the theory of Dirichlet forms in \textsection \ref{ss:df} followed by definitions of sub-Gaussian heat kernel estimates and   related properties in \textsection\ref{ss:hke}.  In \textsection \ref{ss:maindiff}, we state the main results that answer Question \ref{q:main}.
Finally, we prove the necessity of \eqref{e:conddiff} in \textsection \ref{ss:necproof}.
\subsection{Metric measure Dirichlet space and energy measure} \label{ss:df}
Let $(\mathcal{E},\mathcal{F})$ be a \emph{symmetric Dirichlet form} on $L^{2}(X,m)$;
that is, $\mathcal{F}$ is a dense linear subspace of $L^{2}(X,m)$, and
$\mathcal{E}:\mathcal{F}\times\mathcal{F}\to\mathbb{R}$
is a non-negative definite symmetric bilinear form which is \emph{closed}
($\mathcal{F}$ is a Hilbert space under the inner product $\mathcal{E}_{1}:= \mathcal{E}+ \langle \cdot,\cdot \rangle_{L^{2}(X,m)}$)
and \emph{Markovian} ($f^{+}\wedge 1\in\mathcal{F}$ and $\mathcal{E}(f^{+}\wedge 1,f^{+}\wedge 1)\leq \mathcal{E}(f,f)$ for any $f\in\mathcal{F}$).
Recall that $(\mathcal{E},\mathcal{F})$ is called \emph{regular} if
$\mathcal{F}\cap \contfunc_{\mathrm{c}}(X)$ is dense both in $(\mathcal{F},\mathcal{E}_{1})$
and in $(\contfunc_{\mathrm{c}}(X),\|\cdot\|_{\mathrm{sup}})$, and that
$(\mathcal{E},\mathcal{F})$ is called \emph{strongly local} if $\mathcal{E}(f,g)=0$
for any $f,g\in\mathcal{F}$ with $\supp_{m}[f]$, $\supp_{m}[g]$ compact and
$\supp_{m}[f-a\one_{X}]\cap\supp_{m}[g]=\emptyset$ for some $a\in\mathbb{R}$. Here
$\contfunc_{\mathrm{c}}(X)$ denotes the space of $\mathbb{R}$-valued continuous functions on $X$ with compact support, and
for a Borel measurable function $f:X\to[-\infty,\infty]$ or an
$m$-equivalence class $f$ of such functions, $\supp_{m}[f]$ denotes the support of the measure $|f|\,dm$,
i.e., the smallest closed subset $F$ of $X$ with $\int_{X\setminus F}|f|\,dm=0$,
which exists since $X$ has a countable open base for its topology; note that
$\supp_{m}[f]$ coincides with the closure of $X\setminus f^{-1}(0)$ in $X$ if $f$ is continuous.
The pair $(X,d,m,\mathcal{E},\mathcal{F})$ of a metric measure space $(X,d,m)$ and a strongly local,
regular symmetric Dirichlet form $(\mathcal{E},\mathcal{F})$ on $L^{2}(X,m)$ is termed
a \emph{metric measure Dirichlet space}, or an \emph{MMD space} in abbreviation. By Fukushima's theorem about regular Dirichlet forms, the MMD space corresponds to a symmetric Markov processes on $X$ with continuous sample paths \cite[Theorem 7.2.1 and 7.2.2]{FOT}.
We refer to \cite{FOT,CF} for details of the theory of symmetric Dirichlet forms.

We recall the definition of energy measure.
Note that $fg\in\mathcal{F}$
for any $f,g\in\mathcal{F}\cap L^{\infty}(X,m)$ by \cite[Theorem 1.4.2-(ii)]{FOT}
and that $\{(-n)\vee(f\wedge n)\}_{n=1}^{\infty}\subset\mathcal{F}$ and
$\lim_{n\to\infty}(-n)\vee(f\wedge n)=f$ in norm in $(\mathcal{F},\mathcal{E}_{1})$
by \cite[Theorem 1.4.2-(iii)]{FOT}.

\begin{definition}\label{d:EnergyMeas}
	Let $(X,d,m,\mathcal{E},\mathcal{F})$ be an MMD space.
	The \emph{energy measure} $\Gamma(f,f)$ of $f\in\mathcal{F}$
	associated with $(X,d,m,\mathcal{E},\mathcal{F})$ is defined,
	first for $f\in\mathcal{F}\cap L^{\infty}(X,m)$ as the unique ($[0,\infty]$-valued)
	Borel measure on $X$ such that
	\begin{equation}\label{e:EnergyMeas}
		\int_{X} g \, d\Gamma(f,f)= \mathcal{E}(f,fg)-\frac{1}{2}\mathcal{E}(f^{2},g) \qquad \textrm{for all $g \in \mathcal{F}\cap \contfunc_{\mathrm{c}}(X)$,}
	\end{equation}
	and then by
	$\Gamma(f,f)(A):=\lim_{n\to\infty}\Gamma\bigl((-n)\vee(f\wedge n),(-n)\vee(f\wedge n)\bigr)(A)$
	for each Borel subset $A$ of $X$ for general $f\in\mathcal{F}$. The signed measure $\Gamma(f,g)$ for $f,g \in \mathcal{F}$ is defined by polarization.
\end{definition}
Associated with a Dirichlet form is a \textbf{strongly continuous contraction semigroup} $(P_t)_{t > 0}$; that is, a family of symmetric bounded linear operators $P_t:L^2(X,m) \to L^2(X,m)$ such that
\[
P_{t+s}f=P_t(P_sf), \q \norm{P_tf}_2 \le \norm{f}_2, \q \lim_{t \downarrow 0} \norm{P_tf-f}_2 =0, 
\]
for all $t,s>0, f \in L^2(X,m)$. In this case, we can express $(\mathcal{E},\mathcal{F})$ in terms of the semigroup as
\begin{equation} \label{e:semigroup}
	\mathcal{F}=\{f \in L^2(X,m): \lim_{t \downarrow 0}  \frac{1}{t}\langle f - P_t f, f \rangle < \infty \}, \q \mathcal{E}(f,f)= \lim_{t \downarrow 0}  \frac{1}{t}\langle f - P_t f, f \rangle, 
\end{equation}
for all $f \in \mathcal{F}$, where $\langle \cdot, \cdot \rangle$ denotes the inner product in $L^2(X,m)$ \cite[Theorem 1.3.1 and Lemmas 1.3.3 and 1.3.4]{FOT}. It is known that  $P_t$ restricted to  $L^2(X,m) \cap L^\infty(X,m)$ extends to a linear contraction on $L^\infty(X,m)$ \cite[pp. 5 and 6]{CF}. %If $P_t 1 = 1$ ($m$ a.e.) for all $t >0$, we say that the corresponding Dirichlet form $(\mathcal{E},\mathcal{F})$ is \emph{conservative}.

The (non-negative) \emph{generator} $\gener$ with domain $D(\gener)$ of the Dirichlet form $(\mathcal{E},\mathcal{F})$ on $L^2(X,m)$ is defined by
\begin{align} \label{e:defgener}
	D(\gener)  &:= \Biggl\{f \in L^2(X,m): \lim_{t \downarrow 0} \frac{f-P_t f}{t} \mbox{ exists as a strong limit in $L^2(X,m)$}\Biggr\}, \nonumber \\
	\gener( f) &:=\lim_{t \downarrow 0} \frac{f-P_t f}{t}, \quad \mbox{ for all $f \in D(\gener)$.}
\end{align}
By \cite[Lemma 1.3.1]{FOT} the operator $(\gener,D(\gener))$ is a non-negative definite, self-adjoint operator.
The domain of the generator $D(\gener)$  is dense not only in $L^2(X,m)$ by also in $\mathcal{F}$ with respect ot the $\mathcal{E}_1$-inner product due to \cite[Lemma 1.3.3-(iii) and (1.3.3)]{FOT}. By \cite[Corollary 1.3.1]{FOT}, we have 
\begin{equation} \label{e:generate}
	\mathcal{E}(f,g)= \langle \gener(f), g \rangle_{L^2(X,m)}, \quad \mbox{for all $f \in D(\gener), g \in \mathcal{F}$.}
\end{equation}
We recall the definition of capacity between sets.
\begin{definition}[Capacity between sets]{
		For subsets $A,B \subset X$, we define
		\[
		\mathcal{F}(A,B):= \set{f \in \mathcal{F}: f \equiv 1 \mbox{ on a neighborhood of $A$ and } f \equiv 0 \mbox{ on a neighborhood of $B$}},
		\]
		and the capacity $\Cap(A,B)$ as 
		\[
		\Cap(A,B)= \inf \set{\mathcal{E}(f,f): f \in \mathcal{F}(A,B)}.
		\]
}\end{definition}
We recall the notion of a quasi-continuous modification as we will need it for the statement of cutoff energy inequality below. We recall the notion of $1$-capacity of a set. Given an MMD space $(X,d,m,\mathcal{E},\mathcal{F})$ and a Borel set $A$, we define its \emph{$1$-capacity} as
\[
\Cap_1(A)= \inf \set{\mathcal{E}(f,f) + \norm{f}_2^2 : \mbox{$f \in \mathcal{F} \cap \contfunc(X)$, $f \equiv 1$ on a neighborhood of $A$}},
\]
where $\norm{f}_2$ denotes the $L^2(X,m)$ norm.  
We say that an increasing sequence of closed subsets $\{F_k\}$ of $X$ is said to be \emph{nest} for $(X,d,m,\mathcal{E},\mathcal{F})$ if $\lim_{k \to \infty} \Cap_1(X \setminus F_k)=0$.
We say that a function $\wt{u}:X \to \mathbb{R}$ is a quasi-continuous version of $u \in \mathcal{F}$  if $u=\wt{u}$ $m$-almost everywhere and if for any $\epsilon>0$, there exists an open subset $G \subset X$ such that $\Cap_1(G)<\epsilon$ and the restriction $\restr{\wt{u}}{X \setminus G}$ is finite and continuous on $X \setminus G$. Every function in the domain $\mathcal{F}$ admits a quasi-continuous version by \cite[Theorem 2.1.3]{FOT}.
\subsection{Sub-Gaussian heat kernel estimates} \label{ss:hke}
Let $\Psi:[0,\infty) \to [0,\infty)$ be a homeomorphism, such that for all $0 < r \le R$,
\begin{equation}  \label{e:reg}
	C^{-1} \left( \frac R r \right)^{\beta_1} \le \frac{\Psi(R)}{\Psi(r)} \le C \left( \frac R r \right)^{\beta_2}, 
\end{equation}
for some constants $1 < \beta_1 < \beta_2$ and $C>1$. 
Such a function $\Psi$ is said to be a \textbf{scale function}.
For $\Psi$ satisfying \eqref{e:reg}, we define
\begin{equation} \label{e:defPhi}
	\Phi(s)= \sup_{r>0} \left(\frac{s}{r}-\frac{1}{\Psi(r)}\right).
\end{equation}

\begin{definition}[\hypertarget{hke}{$\on{HKE(\Psi)}$}]\label{d:HKE}
	Let $(X,d,m,\mathcal{E},\mathcal{F})$ be an MMD space, and let $\set{P_t}_{t>0}$
	denote its associated Markov semigroup. A family $\set{p_t}_{t>0}$ of non-negative
	Borel measurable functions on $X \times X$ is called the
	\emph{heat kernel} of $(X,d,m,\mathcal{E},\mathcal{F})$, if $p_t$ is the integral kernel
	of the operator $P_t$ for any $t>0$, that is, for any $t > 0$ and for any $f \in L^2(X,m)$,
	\[
	P_t f(x) = \int_X p_t (x, y) f (y)\, dm (y) \qquad \mbox{for $m$-almost all $x \in X$.}
	\]
	We say that $(X,d,m,\mathcal{E},\mathcal{F})$ satisfies the sub-Gaussian \textbf{heat kernel estimates}
	\hyperlink{hke}{$\on{HKE(\Psi)}$}, if there exist $C_{1},c_{1},c_{2},c_{3},\delta\in(0,\infty)$
	and a heat kernel $\set{p_t}_{t>0}$ such that for any $t>0$,
	\begin{align}\label{e:uhke}
		p_t(x,y) &\le \frac{C_{1}}{m\bigl(B(x,\Psi^{-1}(t))\bigr)} \exp \left( -c_{1} t \Phi\left( c_{2}\frac{d(x,y)} {t} \right) \right)
		\qquad \mbox{for $m$-a.e.~$x,y \in X$,}\\
		p_t(x,y) &\ge \frac{c_{3}}{m\bigl(B(x,\Psi^{-1}(t))\bigr)}
		\qquad \mbox{for $m$-a.e.~$x,y\in X$ with $d(x,y) \le \delta\Psi^{-1}(t)$,}
		\label{e:nlhke}
	\end{align}
	where $\Phi$ is as defined in \eqref{e:defPhi} and $\Psi^{-1}:[0,\infty)\to [0,\infty)$ denotes the inverse of the homeomorphism $\Psi$. 
	If the underlying metric is close to a geodesic metric (see Definition \ref{d:chain}), then it is  obtain a lower bound that matches the upper bound in \eqref{e:uhke}. We say that $(X,d,m,\mathcal{E},\mathcal{F})$ satisfies the \emph{full} sub-Gaussian \textbf{heat kernel estimates}
	\hypertarget{hkef}{$\on{HKE_f(\Psi)}$}, if there exist $C_{1},c_{1},c_{2},c_{3},C_4, C_5>0$
	and a heat kernel $\set{p_t}_{t>0}$ such that for any $t>0$, we have \eqref{e:uhke} along with the  lower bound
	\begin{align}\label{e:lhke}
		p_t(x,y) &\ge \frac{c_{3}}{m\bigl(B(x,\Psi^{-1}(t))\bigr)} \exp \left( -C_{4} t \Phi\left( C_{5}\frac{d(x,y)} {t} \right) \right)
		\qquad \mbox{for $m$-a.e.~$x,y \in X$.}
	\end{align}
	%	\begin{equation} \label{e:defPhiRt
			%		\Phi(R,t) := \Phi_{\Psi}(R,t) := \sup_{r>0} \biggl(\frac{R}{r}-\frac{t}{\Psi(r)}\biggr),
			%		\qquad (R,t)\in[0,\infty)\times(0,\infty).
			%	\end{equation}
	\end{definition}
	The function $\Psi$ above governs the space-time scaling of the associated diffusion process as the expected distance traveled by the diffusion process by time $t$ is comparable to $\Psi^{-1}(t)$ for all $t$ such that $\Psi(t)<\diam(X,d)$. 
	%	***Add some examples of $\Psi, \Phi, \Psi^{-1}$.
	
	The simplest scale functions $\Psi$ are power function   with logarithmic or iterated logarithmic corrections. We describe the corresponding inverse   $\Psi^{-1}$ and the function $\Phi$ below in this case.
	\begin{remark} \label{r:profile}
		Suppose there exists $\alpha_1, \beta_1 \in (1,\infty)$ and $\alpha_2, \alpha_3, \beta_2,\beta_3 \in \mathbb{R}$ such that 
		\[
		\Psi(r) \asymp r^{\alpha_1} (\log r)^\alpha_2 (\log \log r)^\alpha_3, \quad \mbox{for all $r>e^e$,}
		\]
		and 
		\[
		\Psi(r) \asymp r^{\beta_1} (\log (1/r))^\beta_2 (\log \log (1/r))^\beta_3, \quad \mbox{for all $r<e^{-e}$.}
		\]
		Then we have 
		\[
		\Psi^{-1}(r) \asymp r^{1/\alpha_1} (\log r)^{-\alpha_2/\alpha_1} (\log \log r)^{-\alpha_3/\alpha_1}, \quad \mbox{for all $r>e^e$,}
		\]
		and 
		\[
		\Psi^{-1}(r) \asymp r^{1/\beta_1} (\log (1/r))^{-\beta_2/\beta_1} (\log \log (1/r))^{-\beta_3/\beta_1}, \quad \mbox{for all $r<e^{-e}$.}
		\]
		The above estimates follow standard facts about regularly varying functions and their asymptotic inverses. Similarly using \cite[Lemma 6.6]{GT}, we have the following estimates on $t \Phi\left( c\frac{d(x,y)}{t}\right)$. It is divided into three cases:
		\begin{itemize}
			\item If   $t<\Psi(d(x,y)), t \ge d(x,y)$, we have 
			\[
			t \Phi\left( c\frac{d(x,y)}{t}\right) \lesssim 1.
			\]
			\item If $t<\Psi(d(x,y))$ and $t \ge d(x,y)$, we have that $t \Phi\left( c\frac{d(x,y)}{t}\right)$ is comparable to 
			\[
			\left(\frac{d(x,y)^{\alpha_1}}{t}\right)^{1/(\alpha_1-1)}  \left(\log \left( \frac{et}{d(x,y)}\right)\right)^{-\alpha_2/(\alpha_1-1)}\left( \log \log \left( \frac{e^2t}{d(x,y)}\right)\right)^{-\alpha_3/(\alpha_1-1)}.
			\]
			\item If $t<\Psi(d(x,y))$ and $t \le d(x,y)$, we have that $t \Phi\left( c\frac{d(x,y)}{t}\right)$ is comparable to 
			\[
			\left(\frac{d(x,y)^{\beta_1}}{t}\right)^{1/(\beta_1-1)} \left(\log \left( \frac{ed(x,y)}{t}\right)\right)^{-\beta_2/(\beta_1-1)} \left(\log \log \left( \frac{e^2d(x,y)}{t}\right)\right)^{-\beta_3/(\beta_1-1)}.
			\]
		\end{itemize}
		Similar estimates also hold if there are higher iterations of logarithm.
	\end{remark}
	
	The sub-Gaussian heat kernel estimates are known to imply Poincar\'e inequality, capacity estimates, cutoff energy inequality and exit time bounds (see for example, \cite[Theorem 1.2]{GHL} and results of \cite{AB,BBK,BB04,GT}). We recall some of the relevant properties below. 
	\begin{definition} \label{d:conseqhke}
		Let $(X,d,m,\mathcal{E},\mathcal{F})$ be an MMD space, $\Psi:[0,\infty) \to [0,\infty)$ be a scale function, and let $\Gamma(\cdot,\cdot)$ denote the corresponding energy measure. Here the function $\Psi$ governs to the space-time scaling of the corresponding diffusion as the properties below can be used to characterize sub-Gaussian heat kernel bounds \hyperlink{hke}{$\on{HKE(\Psi)}$} \cite[Theorem 1.2]{GHL}.
		\begin{enumerate}[(a)]
			\item 
			We say that $(X,d,m,\mathcal{E},\mathcal{F})$ satisfies the \textbf{Poincar\'e inequality} \hypertarget{pi}{$\operatorname{PI}(\Psi)$},
			if there exist constants $C_{P},A_P\ge 1$ such that 
			for all $(x,r)\in X\times(0,\infty)$ and all $f \in \mathcal{F}$,
			\begin{equation} \tag*{$\operatorname{PI}(\Psi)$}
				\int_{B(x,r)} (f -   f_{B(x,r)})^2 \,dm  \le C_{P} \Psi(r) \int_{B(x,A_P r)}d\Gamma(f,f),
			\end{equation}
			where $f_{B(x,r)}:= m(B(x,r))^{-1} \int_{B(x,r)} f\, dm$.
			
			\item 
			For open subsets $U,V$ of $X$ with $\overline{U} \subset V$, we say that
			a function $\phi \in \mathcal{F}$ is a \emph{cutoff function} for $U \subset V$
			if $0 \le \phi \le 1$, $\phi=1$ on a neighbourhood of $\overline{U}$ and $\supp_m[\phi] \subset V$.
			Then we say that $(X,d,m,\mathcal{E},\mathcal{F})$ satisfies the \textbf{cutoff energy inequality} \hypertarget{cs}{$\operatorname{CS}(\Psi)$},
			if there exists $C_S>0, A \in (1,\infty)$ such that the following holds: for all $x \in X$ and $R>0$,
			there exists a cutoff function $\phi \in \mathcal{F}$ for $B(x,R) \subset B(x,R+r)$ such that for all $f \in \mathcal{F}$,
			\begin{equation}\tag*{$\operatorname{CS}(\Psi)$}
				\begin{split}
					&\int_{B(x,R+r) \setminus B(x,R)} \wt{f}^2\, d\Gamma(\phi,\phi)\\
					&\mspace{40mu}\le C_S \int_{B(x,R+r) \setminus B(x,R)} \wt{\phi}^2 \, d\Gamma(f,f)
					+ \frac{C_S}{\Psi(r)} \int_{B(x,R+r) \setminus B(x,R)} f^2\,dm;
				\end{split}
			\end{equation}
			where $\wt{f},\wt{\phi}$ are the quasi-continuous versions of $f,\phi \in \mathcal{F}$ so that $\wt{\phi}$ is uniquely determined $\Gamma(f,f)$-a.e.\ for
			any $f\in\mathcal{F}$; see \cite[Theorem 2.1.3, Lemmas 2.1.4 and 3.2.4]{FOT}.
			
			\item We say that an MMD space $(X,d,m,\mathcal{E},\mathcal{F})$ satisfies the \emph{capacity bound} \hypertarget{cap}{$\operatorname{cap}(\Psi)$} 	
			if there exist $C_1,A_1,A_2>1$ such that for all $R\in (0,\diam(X,d)/A_2)$, $x \in X$, we have
			\begin{equation}\tag*{$\operatorname{cap}(\Psi)$}
				C_1^{-1} \frac{m(B(x,R))}{\Psi(R)} \le	\Cap(B(x,R),B(x,A_1R)^c) \le C_1 \frac{m(B(x,R))}{\Psi(R)}.
			\end{equation}
			The upper and lower bounds on capacity above will be denoted by $\operatorname{cap}(\Psi)_{\le}$ and $\operatorname{cap}(\Psi)_{\ge}$ respectively.
			\item   We say that   the  an MMD space $(X,d,m,\mathcal{E},\mathcal{F})$ satisfies the   \emph{exit time   bound} \hypertarget{exit}{$\on{E}(\Psi)$}, if there exist $C,A \in (1,\infty), \delta \in (0,1)$ such that for all $x \in X, 0<r<\diam(X,d)/A$ the corresponding Hunt process satisfies 
			\begin{equation}\tag*{$\operatorname{E}(\Psi)$}
				C^{-1} \Psi(r) \le \essinf_{y \in B(x, \delta r)} \mathbb{E}_y[\tau_{B(x,r)}]	,\quad \esssup_{y \in B(x,r)} \mathbb{E}_y[\tau_{B(x,r)}] \le C \Psi(r).
			\end{equation}
		\end{enumerate}
	\end{definition}
	
	\subsection{Main results: diffusion case} \label{ss:maindiff}
	We are now ready to state the main results of this work that answer Question \ref{q:main}. Our first main result addresses the necessity of \eqref{e:conddiff}.
	\begin{theorem} \label{t:mainnec}
		Let $V, \Psi:[0,\infty) \to [0,\infty)$ be doubling functions such that $\Psi$ is a scale function. Let $(X,d,m,\mathcal{E},\mathcal{F})$ be an MMD space that satisfies the full sub-Gaussian kernel estimate 	\hyperlink{hkef}{$\on{HKE_f(\Psi)}$}. Furthermore, suppose that exists $C_1>1$ satisfying
		\begin{equation} \label{e:volest}
			C_1^{-1} V(r) \le m(B(x,r)) \le C_1 V(r), \quad \mbox{for all $x \in X, 0<r< \diam(X,d)$.}
		\end{equation}
		Then there exists $C\in (1,\infty)$ such that 
		\begin{equation} \label{e:necdiff}
			C^{-1} \frac{R^2}{r^2} \le \frac{\Psi(R)}{\Psi(r)} \le C \frac{RV(R)}{rV(r)}, \quad \mbox{for all $0<r \le R< \diam(X,d)$.}
		\end{equation}
	\end{theorem}
	Our next result addresses the sufficiency of \eqref{e:conddiff}.  
	\begin{theorem}\label{t:mainsuf}
		Let $V, \Psi:[0,\infty) \to [0,\infty)$ be doubling functions and let $C_0 \in (1,\infty)$ be such that 
		\[
		C_0^{-1} \frac{R^2}{r^2} \le \frac{\Psi(R)}{\Psi(r)} \le C_0 \frac{RV(R)}{rV(r)}, \quad \mbox{for all $0<r \le R$.}
		\]
		Then there exists  an unbounded (that is, infinite diameter) MMD space  $(X,d,m,\mathcal{E},\mathcal{F})$ that satisfies the full sub-Gaussian kernel estimate 	\hyperlink{hkef}{$\on{HKE_f(\Psi)}$} and there exists $C_1 \in (1,\infty)$ such that the volume of balls satisfy the estimate \eqref{e:volest}. 
	\end{theorem}
	We   obtain similar results for simple random walks on graphs and also for a sequence of growing finite graphs in Theorem \ref{t:sufgraph}.
	\subsection{Necessary conditions on volume and space-time scaling} \label{ss:necproof}
	
	The rest of the section is devoted to the proof of Theorem \ref{t:mainnec}.
	So let us assume that  $V, \Psi:[0,\infty) \to [0,\infty)$ be doubling functions such that $\Psi$ is a scale function and let $(X,d,m,\mathcal{E},\mathcal{F})$ be an MMD space satisfying the estimate  \eqref{e:volest}.

	The following lemma   demonstrates how  obtaining bounds on capacity at a larger scale using bounds on capacity at smaller scale leads to the first estimate in \eqref{e:necdiff}.
	
	\begin{lemma} \label{l:dwge2} Let $(X,d,m,\mathcal{E},\mathcal{F})$ be a MMD space that satisfies   the volume estimate \eqref{e:volest} and the capacity estimate \hyperlink{cap}{$\operatorname{cap}(\Psi)$}. Then there exists $C \in (1,\infty)$ such that 
		\[
		\frac{\Psi(R)}{\Psi(r)} \ge  C^{-1}\frac{R^2}{r^2}, \quad \mbox{for all $\diam(X,d)>R>r>0$.}
		\]
	\end{lemma}
	
	\begin{proof}
		We make a simplifying assumption that $\diam(X,d)=\infty$ for ease of notation. The general case follows from the same argument. 
		
		Let $R>r>0$.
		Let $n$ be  the largest positive integer such that $B(x,A_1 R) \setminus B(x,R) \supseteq \cup_{j=0}^{n-1} B(x,R+(j+1)r) \setminus B(x,R+jr)$, where $A_1 \in (1,\infty)$ is the constant in \hyperlink{cap}{$\operatorname{cap}(\Psi)$}. 
		By the triangle inequality
		\begin{equation} \label{e:dw1}
			n \gtrsim \frac{R}{r}.
		\end{equation}
		By a   covering argument  (by covering the annulus $ B(x,R+(j+1)r) \setminus B(x,R+jr)$ with balls of radii comparable to $r$)as described in the argument for \cite[(6.1) in Proof of Lemma 6.2]{Mur24}, we obtain
		\[
		\Cap(B(x,R+jr),B(x,R+(j+1)r)^c) \lesssim \frac{m(B(x,R+(j+1)r)) \setminus B(x,R+jr)}{\Psi(r)}.
		\]
		Let $\psi_j \in \mathcal{F} \cap C_c(X)$ be a cutoff function such that $\restr{\psi_j}{B(x,R+jr)} \equiv 1, \restr{\psi_j}{B(x,R+(j+1)r)^c} \equiv 0$ and 
		\begin{equation} \label{e:dw2}
			\mathcal{E}(\psi_j,\psi_j) \lesssim   \frac{m(B(x,R+(j+1)r)) \setminus B(x,R+jr)}{\Psi(r)}.
		\end{equation} By considering the function $\psi= \frac{1}{n}\sum_{j=0}^{n-1} \psi_j$, by strong locality and \hyperlink{cap}{$\operatorname{cap}(\Psi)$} 	
		\begin{align*}
			\frac{V(R)}{\Psi(R)} &\stackrel{\mbox{\hyperlink{cap}{$\operatorname{cap}(\Psi)_{\ge}$,}\eqref{e:volest} 	}}{\lesssim}\Cap(B(x,R),B(x,A_1 R)^c) \\
			&\le \mathcal{E}(\psi,\psi) = \frac{1}{n^2} \sum_{j=0}^{n-1} \mathcal{E}(\psi_j,\psi_j) \stackrel{\eqref{e:dw1},\eqref{e:dw2}}{\lesssim} \frac{r^2}{R^2} \frac{\mu(B(x,AR))\setminus B(x,R)}{\Psi(r)} \stackrel{\eqref{e:volest}}{\lesssim} \frac{r^2}{R^2} \frac{V(R)}{\Psi(r)}
		\end{align*}
		which implies the desired bound.
	\end{proof}
	
	%Next, we would like to generalize the condition $d_w \le d_f+1$. Consider the Brownian motion on $\mathbb{R}$ with respect to the (snowflaked) metric $d^\alpha$, where $d$ is the Euclidean metric and $\alpha \in (0,1]$. In this case, the Lebesgue (symmetric) measure  satisfies $d_f= \alpha^{-1}$ and the exit time exponent is $d_w=2 \alpha^{-1}$. Therefore $d_w \le d_f + 1$ implies $\alpha =1$. In order to generalize  $d_w \le d_f+1$, we will assume that the metric satisfies the chain condition.
	
	We recall the definition of the chain condition which is a consequence of \hyperlink{hkef}{$\on{HKE_f(\Psi)}$} as shown in \cite[Theorem 2.11]{Mur20}.
	\begin{definition} \label{d:chain}{
			Let $(X,d)$ be a metric space.
			We say that a sequence  $\set{x_i}_{i=0}^N$  of points in $X$ is an \emph{$\varepsilon$-chain} between points $x,y \in X$ if 
			\[
			x_0=x, \q x_N=y, \q \mbox{ and } \q d(x_i,x_{i+1}) < \varepsilon \q \mbox{ for all $i=0,1,\ldots,N-1$.}
			\]
			For any $\varepsilon>0$ and $x,y \in X$, define
			\[
			d_\varepsilon(x,y) = \inf_{\set{x_i} \mbox{ is $\varepsilon$-chain}} \sum_{i=0}^{N-1} d(x_i,x_{i+1}),
			\]
			where the infimum is taken over all $\varepsilon$-chains  $\set{x_i}_{i=0}^N$ between $x,y$ with arbitrary $N$. \\
			We say that $(X,d)$ satisfies the \emph{chain condition} if there exists $K \in [1,\infty)$ such that 
			\[
			d_\varepsilon(x,y) \le K d(x,y) \quad \mbox{for all $\varepsilon>0, x,y \in X$.}
			\]
	}\end{definition}
	The following lemma is needed for the proof of the second estimate in \eqref{e:necdiff}. The reason for chain condition  in the assumption of Lemma \ref{l:dwdf} is due to the fact the conclusion fails to hold without this assumption. This can be seen using a snowflake transformation of the metric ($d \mapsto d^\gamma$ for some $\gamma \in (0,1)$).
	\begin{lemma} \label{l:dwdf} Let $(X,d,m,\mathcal{E},\mathcal{F})$ be a MMD space that satisfies   the volume estimate  \eqref{e:volest} and the capacity upper estimate \hyperlink{cap}{$\operatorname{cap}(\Psi)_\le$}, and the Poincar\'e inequality \hyperlink{pi}{$\operatorname{PI}(\Psi)$}. Assume that the metric space $(X,d)$ satsifies the chain condition. Then  there exists $C \in (1,\infty)$ such that 
		\[
		\frac{\Psi(R)}{\Psi(r)} \le C \frac{R V(R)}{r V(r)}, \quad \mbox{for all $\diam(X,d)> R>r>0$.}
		\]
	\end{lemma}
	
	\begin{proof}
		We make a simplifying assumption that $\diam(X,d)=\infty$ for ease of notation. The general case follows from the same argument. 
		
		Let $R>r>0$ and let $\phi$ be a function for $ \mathcal{E}(\phi,\phi) \le 2 \Cap(B(x,R), B(x,A_1 R)^c)$ such that $\phi \equiv 1$ on $B(x,R)$ and $\phi \equiv 0$ on $B(x,A_1 R)^c$ almost everywhere, where $A_1$ is the constant in \hyperlink{cap}{$\operatorname{cap}(\Psi)_\le$}. 
		By the chain condition, we can find a sequence of balls $B_0,\ldots,B_{n}$ where $B_i=B(x_i,r), i=0,1,\ldots,n$ of radii $r$ and $K\in (1,\infty)$  such that $K^{-1}B_i= B(x_i,K^{-1}r)$ are pairwise disjoint, $B_{i+1} \subset KB_i=B(x_i,Kr)$ for all $i=0,\ldots,n-1$, $B_0 \subset B(x,A_1 R)^c, B_{n} \subset B(x,R)$, and $n \lesssim \frac{R}{r}$. Let $$\phi_{B_i}:= \frac{1}{m(B_i)} \int_{B_i} \phi \,dm. $$
		By the above properties of $B_i$, we have $\phi_{B_0}=0, \phi_{B_n}=1$.
		Let $A_P \in [1,\infty)$ denote the constant in \hyperlink{pi}{$\operatorname{PI}(\Psi)$}.
		We obtain the desired bound as follows:
		\begin{align*}
			1 &= \left( \sum_{i=0}^{n-1} \phi_{B_{i+1}}- \phi_{B_i} \right)^2 \le n \sum_{i=0}^{n-1}  \left( \phi_{B_{i+1}}- \phi_{B_i} \right)^2  \q \mbox{(by Cauchy-Schwarz)}\\
			&\lesssim \frac{R}{r} \sum_{i=0}^{n-1}  \left( \phi_{B_{i+1}}- \phi_{B_i} \right)^2 \quad \mbox{(since $n \lesssim R/r$)}\\
			&\lesssim \frac{R}{r} \sum_{i=0}^{n-1}  \frac{1}{m(B_i) m(B_{i+1})} \int_{B_{i+1}}\int_{B_i} (\phi(x)-\phi(y))^2 \,m(dy)\,m(dx)    \q \mbox{(by Jensen's inequality)}\\
			& \lesssim \frac{R}{r}\frac{1}{V(r)} \sum_{i=0}^{n-1}\int_{K B_i} \abs{\phi-\phi_{K B_i}}^2\,dm \q \mbox{(by volume doubling, \eqref{e:volest}, $B_{i+1} \subset K B_i$)}\\
			&\lesssim \frac{R}{r}  \frac{\Psi(r)}{V(r)}  \sum_{i=0}^{n-1} \int_{A_P KB_i} d\Gamma(\phi,\phi) \q \mbox{(by \hyperlink{pi}{$\operatorname{PI}(\Psi)$} and bounded overlap $(A_PKB_i)_{0 \le i \le n}$)}\\
			& \lesssim \frac{R}{r} \frac{\Psi(r)}{V(r)}  \mathcal{E}(\phi,\phi) \lesssim \frac{R}{r} \frac{\Psi(r)}{V(r)}  \Cap(B(x,R),B(x,AR)^c) \\
			& \lesssim \frac{R}{r} \frac{\Psi(r)}{V(r)}  \frac{V(R)}{\Psi(R)} \q \mbox{(by \hyperlink{cap}{$\operatorname{cap}(\Psi)_{\le}$} and \eqref{e:volest})}.
		\end{align*}
	\end{proof} 
	We conclude the proof below.
	\begin{proof}[Proof of Theorem \ref{t:mainnec}]
		We note that capacity bounds \hyperlink{cap}{$\operatorname{cap}(\Psi)$}, Poincar\'e inequality \hyperlink{pi}{$\operatorname{PI}(\Psi)$} and the chain condition follow from \cite[Theorem 1.2]{GHL} and \cite[Theorem 2.11]{Mur20}. Hence the desired conclusion follows from Lemmas \ref{l:dwge2} and \ref{l:dwdf}.
	\end{proof}
	
	\section{Construction of the trees and Laakso-type spaces} \label{s:mms}

	The goal of this section is to construct Laakso-type spaces which in turn requires us to define a suitable family of $\mathbb{R}$-trees. 
	In \textsection \ref{ss:ultra}, we define the ultrametric space that is used to index different copies of trees. This ultrametric space and its corresponding measure will also be used to construct a measure on the $\mathbb{R}$-tree. In \textsection \ref{ss:finitetrees}, we define a family of finite trees and study their properties.  In \textsection \ref{ss:rtree}, we define the $\mathbb{R}$-tree as a suitable limit of finite trees introduced in \textsection\ref{ss:finitetrees}. In \textsection\ref{ss:laak}, we define the Laakso-type space by taking many disjoint copies of $\mathbb{R}$-trees constructed in \textsection \ref{ss:rtree} and identifying (or gluing) carefully chosen \emph{wormhole} points in the different copies of tree. We   define a suitable metric and measure on the Laakso-type space and establish some of the basic properties of the resulting metric measure space. We show that certain balls in trees and Laakso-type spaces are uniform domains in \textsection \ref{ss:uniform}.

	The $\mathbb{R}$-tree we construct is completely determined by \emph{branching function} $\branch: \mathbb{Z} \to \mathbb{N}$ whose definition we introduce below. 
	\begin{definition}
		We say that a function  $\branch: \mathbb{Z} \to \mathbb{N}$ is a \emph{branching function}, if 
		\begin{equation} \label{e:defbranch}
			2 \le			\inf_{k \in \mathbb{Z}} \branch(k) \le \sup_{k \in \mathbb{Z}} \branch(k) < \infty.
		\end{equation}
	\end{definition}
	
	The $\mathbb{R}$-tree we construct in \textsection\ref{ss:rtree} is  completely determined by the branching function $\branch$.
	Different copies of the $\mathbb{R}$-tree will be indexed by an ultrametric space that depends on   a function $\glue: \mathbb{Z} \to \mathbb{N}$ such that 
	\begin{equation}\label{e:defglue}
		1 \le \inf_{\mathbb{Z}} \glue \le \sup_{\mathbb{Z}} \glue < \infty.
	\end{equation}
	This function determines how different points in copies of the $\mathbb{R}$-tree are glued together and hence we could call $\glue$ as the gluing function.
	For the remainder of this section, we choose and fix a  branching function $\branch:\mathbb{Z} \to \mathbb{N}$ satisfying \eqref{e:defbranch} and a gluing function $\glue: \mathbb{Z} \to \mathbb{N}$ that satisfies \eqref{e:defglue}. 
	Roughly speaking, the values of $\branch(k)$ and $\glue(k)$ determine the construction at scale $2^k$.
	As we will see, the family of spaces obtained by varying $\branch$ and $\glue$ satisfying the above conditions is large enough to prove Theorem \ref{t:mainsuf}.

	\subsection{An ultrametric  space with a measure} \label{ss:ultra}
	We start with a description of the ultrametric space determined by $\glue:\mathbb{Z} \to \mathbb{Z}$ that satisfies \eqref{e:defglue}.
	This sequence defines an ultrametric space on  the set
	\begin{equation} \label{e:defultra}
		\ultra(\glue):= \set{ \mathbf{s}:\mathbb{Z} \to \mathbb{Z} \mid \mathbf{s}(k) \in \llbracket 0, \glue(k)-1 \rrbracket \mbox{ for all $k \in \mathbb{Z}$ and } \lim_{k \to \infty} \mathbf{s}(k)=0 },
	\end{equation}
	equipped with the metric 
	\begin{equation} \label{e:defultramet}
		\metricultra(\mathbf{s},\mathbf{t})= 2^{\inf \{k \in \mathbb{Z} \mid \mathbf{s}(l) = \mathbf{t}(l), \mbox{ for all $l \ge k$} \} }, \quad \mbox{for all $\mathbf{s},\mathbf{t} \in \ultra(\glue)$.}
	\end{equation}
	We denote the open and closed balls centered at $\mathbf{s} \in \ultra(\glue)$ with radius $r>0$ as $B_{\ultra(\glue)}(\mathbf{s},r)$ and $\overline{B}_{\ultra(\glue)}(\mathbf{s},r)$.
	
	Next, we describe a Borel measure $\mathsf{m}_{\ultra(\glue)}$ on $(\ultra(\glue), \metricultra)$. To define this measure, it would be convenient to view $\ultra(\glue)$ as a product of two sets $\ultra(\glue,-\infty,0) \times\ultra(\glue,1,\infty)$, where 
	\[
	\ultra(\glue,-\infty,0):= \left\{ \restr{\mathbf{s}}{\llbracket-\infty,0 \rrbracket}  \mid \mathbf{s} \in \ultra(\glue) \right\}, \quad \ultra(\glue,1,\infty):= \left\{ \restr{\mathbf{s}}{\llbracket 1, \infty \rrbracket} \mid \mathbf{s} \in \ultra(\glue) \right\}.
	\]
	We note that $\ultra(\glue,-\infty,0)$ can be naturally identified with the product space $\prod_{\stackrel{j \in \mathbb{Z},}{j \le 0}} \llbracket 0, \glue(k)-1\rrbracket$. Using this we define the measure  $\mathsf{m}_{\ultra(\glue,-\infty,0)}$ on  $\ultra(\glue,-\infty,0)$  as the product measure $\prod_{\stackrel{j \in \mathbb{Z},}{j \le 0}} \mathsf{m}_{\glue(k)}$, where $\mathsf{m}_{\glue(k)}$ denote the uniform probability measure on the finite set $\llbracket 0, \glue(k)-1\rrbracket$. We define $\mathsf{m}_{\ultra(\glue,1,\infty)}$ as the counting measure on  $\ultra(\glue,1,\infty)$ (note that $\ultra(\glue,1,\infty)$ is at most countable).
	
	We identify   $\ultra(\glue)$ with $\ultra(\glue,-\infty,0) \times\ultra(\glue,1,\infty)$ using the obvious bijection and using this bijection we define the measure $\mathsf{m}_{\ultra(\glue)}$ on $\ultra(\glue)$ as the product measure 
	\begin{equation} \label{e:defultrameas}
		\mathsf{m}_{\ultra(\glue)}:=	\mathsf{m}_{\ultra(\glue,-\infty,0)} \times \mathsf{m}_{\ultra(\glue,1,\infty)}.
	\end{equation}
	This defines a metric measure space $(\ultra(\glue), \metricultra,\mathsf{m}_{\ultra(\glue)})$. 
	It is easy to verify that the measure of balls is given by
	\begin{equation} \label{e:measultraformula}
		\mathsf{m}_{\ultra(\glue)}\left(B_{{\ultra(\glue)}}(\mathbf{s},r)\right) = \begin{cases}
			\left(\prod_{\stackrel{0 \le k \le n}{k \in \mathbb{Z}}} \glue(k)\right)^{-1}& \mbox{if $2^{n-1}<r \le 2^n, n \in \mathbb{Z}, n \le 0$,}\\
			1& \mbox{if $1<r \le 2$,}\\
			\prod_{k=1}^{n-1} \glue(k)& \mbox{if $2^{n-1}<r \le 2^n, n \in \mathbb{Z}, n \ge 2$,}
		\end{cases}
	\end{equation}
	for any $\mathbf{s} \in \ultra(\glue)$ and $r>0$. Since $\sup_{\mathbb{Z}} \glue<\infty$, we note that $\mathsf{m}_{\ultra(\glue)}$ is a doubling measure on $(\ultra(\glue),\metricultra)$. Since the volume growth does not depend  on the center, we introduce the abbreviated notation
	\begin{equation} \label{e:defvol}
		V_\glue(r):= 
		\mathsf{m}_{\ultra(\glue)}\left(B_{{\ultra(\glue)}}(\mathbf{s},r)\right) \quad \mbox{for all $r>0$.}
	\end{equation}
	We also use the construction of ultrametric space above when the gluing function $\glue:\mathbb{Z} \to \mathbb{Z}$ is replaced with the branching function $\branch:\mathbb{Z} \to \mathbb{Z}$. In that case we denote the corresponding metric measure space by $(\ultra(\branch),\mathsf{d}_{\ultra(\branch)},\mathsf{m}_{\ultra(\branch)})$ and its volume function by $V_\branch$ (see   Proposition \ref{p:dsr}, \eqref{e:defmeastree}, and Corollary \ref{c:meastree}).
	
	We use the notion of David-Semmes regular map   to obtain estimates on measures   on trees and Laasko-type spaces. We recall its definition below \cite[Definition 12.1]{DS}.
	\begin{definition}
		A map $F:(X,\mathsf{d}_X) \to (Y,\mathsf{d}_Y)$ between metric spaces is  \textbf{David-Semmes regular} if there exists $L,M,N \in (0,\infty)$ such that $F$ is $L$-Lipschitz and for every $y \in Y, r>0$, there exists $x_1,\ldots,x_M \in X$ such that 
		\begin{equation} \label{e:defdsr}
			F^{-1}(B_{Y}(y,r)) \subset \cup_{i=1}^M B_{X}(x_i,Nr),
		\end{equation}
		where $B_X(\cdot,\cdot), B_Y(\cdot,\cdot)$ denote open balls in the corresponding metric spaces.
	\end{definition}
	
	The following elementary lemma is inspired by \cite[Proof of Theorem 2.6]{Laa}. It shows that uniform estimates on doubling measures is preserved by push-forward measures under a surjective, David-Semmes regular map.
	\begin{lemma} \label{l:volume}
		Let $C_1 \in (1,\infty)$ and let $V:(0,\infty) \to (0,\infty)$ be a non-decreasing function such that 	
		$V(2r) \le C_1 V(r)$. Let $F:(X,\mathsf{d}_X) \to (Y,\mathsf{d}_Y)$ be a surjective, David-Semmes regular map. Let $C_2 \in (1,\infty)$ and let $\mathsf{m}_X$ be a Borel measure on $(X,\mathsf{d}_X)$ such that 
		\begin{equation} \label{e:dsl1}
			C_2^{-1} V(r) \le \mathsf{m}_X(B_{X}(x,r)) \le C_2 V(r), \quad \mbox{for all $x \in X, r>0$.}
		\end{equation}
		Then, there exists $C_3>1$ such that the push forward measure $\mathsf{m}_Y:= F_* \mathsf{m}_X =\mathsf{m}_X \circ F^{-1}$ on $Y$ satisfies
		\begin{equation} \label{e:dsl2}
			C_3^{-1} V(r) \le \mathsf{m}_Y(B_{Y}(y,r)) \le C_3 V(r), \quad \mbox{for all $y \in Y, r>0$.}
		\end{equation}
	\end{lemma}
	\begin{proof}
		The upper bound in \eqref{e:dsl2} is obtained by using the covering property \eqref{e:defdsr} and using \eqref{e:dsl1} for each ball in the covering.
		
		For the lower bound, we use the fact that $F$ is Lipschitz. Since $F$ is surjective and $L$-Lipschitz, for every $y \in Y, r>0$, there exists $x \in X$ such that $F(x)=y$ and $B_{X}(x,r/L) \subset F^{-1}(B_{Y}(y,r))$ and hence $\mathsf{m}_Y(B_{Y}(y,r)) \ge \mathsf{m}_X(B_{X}(x,r/L) )$. The desired conclusion follows from the doubling property of $\mathsf{m}_X$ and \eqref{e:dsl1}.
	\end{proof}

	\subsection{An approximating family of finite trees} \label{ss:finitetrees}
	We introduce  a family of finite graphs $T_{m,n}$ that approximate the desired  $\mathbb{R}$-tree. We outline the construction before providing the technical details. Given a branching function $\mathbf{b}: \mathbb{Z} \to \llbracket 2 ,\infty \rrbracket$, $T_{n,n}$ is the complete bipartite graph $K_{1,\mathbf{n}}$. We recursively construct $T_{m-1,n}$ from $T_{m,n}$ for each $m \le n$ by replacing each edge in $T_{m,n}$ with a copy of the complete bipartite graph $K_{1,\mathbf{b}(m-1)}$, where the end points of each edge in  $T_{m,n}$ are mapped to leaves of the tree $K_{1,\mathbf{b}(m-1)}$. These finite trees $T_{m,n}$ converge to a $\mathbb{R}$-tree as $m \to -\infty$ and $n \to \infty$ and this limiting $\mathbb{R}$-tree is used as a building block for our construction of Laakso-type space.

	\begin{definition}	
		For $m,n \in \mathbb{Z}$ with $m\le n$ we define a graph $T_{m,n}$ as follows.
		The \textbf{vertex set} $V(T_{m,n})$ is defined to be the collection of equivalence classes on set the \emph{labels} $	S(T_{m,n})$ with respect to an equivalence relation $R_{m,n}$ on $	S(T_{m,n})$; that is $V(T_{m,n})= S(T_{m,n})/R_{m,n}$. The set of labels 	$S(T_{m,n})$ is defined as 
		\begin{equation} \label{e:deflabel}
			S(T_{m,n}):= \{\mathbf{s}:  \llbracket m,n \rrbracket \to \mathbb{Z}   \mid \mbox{ $\mathbf{s}(m) \in \{0,1\}$, }   \mathbf{s}(k) \in \llbracket 0, \branch(k)-1 \rrbracket \mbox{ for all $k \in \llbracket m+1,n \rrbracket$} \}
		\end{equation}
		An ordered pair of labels $(\mathbf{s},\mathbf{t}) \in 	S(T_{m,n})\times 	S(T_{m,n})$ belong to  $R_{m,n}$ if and only if either $\mathbf{s}=\mathbf{t}$ or if there exists $l \in \llbracket m+1,n \rrbracket$ such that  
		\[
		\restr{\mathbf{s}}{\llbracket m ,n \rrbracket \setminus \{l\}}=  \restr{\mathbf{t}}{\llbracket m ,n \rrbracket \setminus \{l\}}, \quad 	\mathbf{s}(l-1)=\mathbf{t}(l-1)=1, \quad \mbox{and } 	\mathbf{s}(j)=\mathbf{t}(j)= 0 \quad \mbox{for all $j < l-1$.}
		\]
		It is evident that  $R_{m,n}$ defines an equivalence relation on $S(T_{m,n})$. For $\mathbf{s} \in S(T_{m,n})$, we denote the equivalence class with respect to $R_{m,n}$ as  $[\mathbf{s}]_{m,n} \in V(T_{m,n})$. A pair of distinct vertices $\{[\mathbf{s}]_{m,n},[\mathbf{t}]_{m,n}\}$ belong to the set of \textbf{edges} $E(T_{m,n})$ if and only if
		\begin{equation}\label{e:defedge}
			\{\mathbf{s}(m), \mathbf{t}(m)\} = \{0,1\}, \quad \mbox{and} \quad \restr{\mathbf{s}} {\llbracket m+1,n \rrbracket}=\restr{\mathbf{t}} {\llbracket m+1,n \rrbracket}.
		\end{equation}
		We define two distinguished vertices $r(T_{m,n})= [\mathbf{r}]_{m,n}, p(T_{m,n})= [\mathbf{p}]_{m,n} \in V(T_{m,n})$ as 
		\[
		\mathbf{r}(k)=\mathbf{p}(k)=0 \quad \mbox{for all $k \in \llbracket m ,n-1 \rrbracket$, $\mathbf{r}(n)=0$, and \quad} \mathbf{p}(n)=1.
		\]
		Let $\mathsf{d}_{m,n}: V(T_{m,n}) \times V(T_{m,n}) \to [0,\infty)$ denote the (combinatorial) graph distance corresponding to the graph $T_{m,n}$. 
	\end{definition}
	\begin{figure}\centering
		\includegraphics[height=110pt]{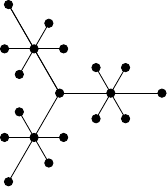}\hspace*{35pt}
		\includegraphics[height=110pt]{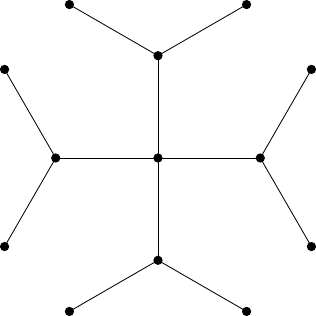}\hspace*{35pt}
		\includegraphics[height=110pt]{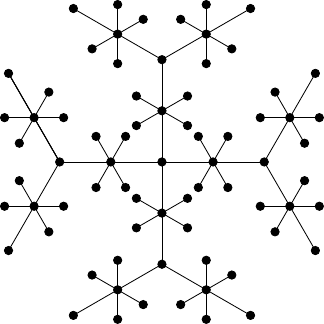}
		\caption{The graphs $T_{-2,0}, T_{-1,1}$ and $T_{-2,1}$ (from left to right) in the case $\branch(1)=4,\branch(0)=3,\branch(-1)=6$. Note that   $T_{-2,1}$ can be constructed by gluing $\branch(1)=4$ copies of $T_{-2,0}$ or alternately by replacing every edge of $T_{-1,1}$ with the complete bipartite graph $K_{1,\branch(-1)}=K_{1,6}$\label{fig.rec}}
	\end{figure} 
	We state some basic properties of the graph $T_{m,n}$. The following notation would be convenient to describe the relationship between $T_{m,n}$ and $T_{m,n+1}$.  If $\mathbf{s} \in S(T_{m,n})$ and $j \in \llbracket 0 , b(n+1)-1 \rrbracket$, we define the label $ (\mathbf{s}j): \llbracket m , n+1 \rrbracket \to \mathbb{Z} \in S(T_{m-1,n})$    as
	\begin{equation} \label{e:defsj}
		(\mathbf{s}j)(n+1)=j, \quad (\mathbf{s}j)(k)=\mathbf{s}(k), \quad \mbox{for all $k \in \llbracket m ,n \rrbracket$.}
	\end{equation}
	It is easy to see that the map $\mathbf{s} \mapsto \mathbf{s}j$ respects the corresponding equivalence relations; that is, $(\mathbf{s},\mathbf{t}) \in R_{m,n}$ implies $(\mathbf{s}j,\mathbf{t}j) \in R_{m,n+1}$ for any $j \in \llbracket 0, b(n+1)-1 \rrbracket$. Hence this map $\mathbf{s} \mapsto \mathbf{s}j$ induces a well-defined function from $V(T_{m,n})$ to $V(T_{m,n+1})$.

	The graph $T_{m,n+1}$ can be viewed as taking $b(n+1)$-copies of $T_{m,n}$ with the vertex $p(T_{m,n})$ in each of the  $b(n+1)$-copies identified (glued) as a single vertex as described in the following lemma (see Figure \ref{fig.rec}).
	\begin{lemma} \label{l:incntree}
		Let $m \le n$ with $m,n \in \mathbb{Z}$ and $j \in \llbracket 0, b(n+1)-1 \rrbracket$, 
		Define $I_{m,n}^{m,n+1,j}: V(T_{m,n}) \to V(T_{m,n+1})$ as 
		\[
		I_{m,n}^{m,n+1,j}([\mathbf{s}]_{m,n})= [\mathbf{s}j]_{m,n+1}.
		\]
		\begin{enumerate}[(a)]
			\item If $k,l \in  \llbracket 0, b(n+1)-1 \rrbracket$ are distinct, then
			\[
			I_{m,n}^{m,n+1,k}(V(T_{m,n})) \cap 	I_{m,n}^{m,n+1,l}(V(T_{m,n})) = \{I_{m,n}^{m,n+1,k}(p(T_{m,n})) \}
			\]
			\item The map $I_{m,n}^{m,n+1,j}: V(T_{m,n}) \to V(T_{m,n+1})$ is an isometry, that is 
			\[
			\mathsf{d}_{m,n+1}([\mathbf{s}j]_{m,n+1},[\mathbf{t}j]_{m,n+1})= 	\mathsf{d}_{m,n}([\mathbf{s}]_{m,n},[\mathbf{t}]_{m,n}), \quad \mbox{for all $[\mathbf{s}]_{m,n},[\mathbf{t}]_{m,n}\in V(T_{m,n})$.}
			\]
			\item The graph $T_{m,n}$ is a tree and its diameter is given by $\diam(V(T_{m,n}), \mathsf{d}_{m,n})= 2^{n-m}= \mathsf{d}_{m,n}(p(T_{m,n}),r(T_{m,n}))$.
		\end{enumerate}
		
	\end{lemma}
	
	\begin{proof}
		\begin{enumerate}[(a)]
			\item This follows from the definition of   equivalence relation $R_{m,n}$.
			\item We note that $\{[\mathbf{s}j]_{m,n+1},[\mathbf{t}j]_{m,n+1}\} \in E(T_{m,n+1})$ if and only if $\{[\mathbf{s}]_{m,n},[\mathbf{t}]_{m,n}\} \in E(T_{m,n})$. This implies that $I_{m,n}^{m,n+1,j}$ is an isometry.
			\item The claim is obvious when $m=n$ as $T_{m,m}$ consists of two vertices $r(T_{m,m})$ and $p(T_{m,m})$ joined by an edge. If $m<n$, the result follows from (a), (b) and induction on $n-m$.
		\end{enumerate}
	\end{proof}

	Next, we describe a similar recursive construction of 
	$T_{m-1,n}$ from $T_{m,n}$.
	Similar  to \eqref{e:defsj}, if $i \in \{0,1\}$ and $\mathbf{s} \in S(T_{m,n})$, we define the label $ (i\mathbf{s}): \llbracket m-1 , n\rrbracket$    as
	\begin{equation} \label{e:defjs}
		(i\mathbf{s})(m-1)=i, \quad  (i\mathbf{s})(k)=\mathbf{s}(k), \quad \mbox{for all $k \in \llbracket m ,n \rrbracket$.}
	\end{equation}
	We define 	$I^{m-1,n}_{m,n}: V(T_{m,n}) \to V(T_{m-1,n})$ as
	\begin{equation} \label{e:defdecm}
		I^{m-1,n}_{m,n}([\textbf{s}]_{m,n})=[0\textbf{s}]_{m-1,n}, \quad \mbox{for all $[\textbf{s}]_{m,n} \in V(T_{m,n})$.}
	\end{equation}
	It is evident that the map $I^{m-1,n}_{m,n}$ is well-defined.
	If $\mathbf{s} \in S(T_{m+1,n})$ and if $[0\mathbf{s}]_{m,n}, [1\mathbf{s}]_{m,n}$ denote adjacent vertices in $V(T_{m,n})$ then the distance between the corresponding images $I^{m-1,n}_{m,n}([0\mathbf{s}]_{m,n})=[00\mathbf{s}]_{m-1,n}, I^{m-1,n}_{m,n}([1\mathbf{s}]_{m,n})=[01\mathbf{s}]_{m-1,n}$ is two as they are both adjacent to $[10\mathbf{s}]_{m-1,n}=[11\mathbf{s}]_{m-1,n}$. We summarize this in the following lemma.
	\begin{lemma} \label{l:decmtree}
		Let $m,n \in \mathbb{Z}$ be such that $m \le n$. Then the map $I^{m-1,n}_{m,n}: V(T_{m,n}) \to V(T_{m-1,n})$ defined in \eqref{e:defdecm} is well-defined and satisfies
		\[
		\mathsf{d}_{m-1,n}\left(I^{m-1,n}_{m,n}([0\mathbf{s}]_{m,n}),I^{m-1,n}_{m,n}([0\mathbf{t}]_{m,n})\right) = 2 \mathsf{d}_{m,n} \left(I^{m-1,n}_{m,n}([0\mathbf{s}]_{m,n}), I^{m-1,n}_{m,n}([0\mathbf{t}]_{m,n})\right)
		\]
	\end{lemma} 
	
	The construction of $F_{m-1,n}$ from $F_{m,n}$ is done by replacing each edge $\{[0\mathbf{s}]_{m,n}, [1\mathbf{s}]_{m,n}\}$ with a copy of the complete bipartite graph $K_{1,\branch(m)}$, where the partitions are given by $\{[10\mathbf{s}]_{m-1,n}\}$ and $\{[0j\mathbf{s}]_{m-1,n}: j \in \llbracket 0, b(m)-1 \rrbracket \}$ respectively and the end points of the edge are mapped to the latter partition via the map $I_{m,n}^{m-1,n}$ (see Figure \ref{fig.rec}). This can be used to provide another proof for Lemma \ref{l:incntree}(c).
	
	We describe the continuum self-similar tree of \cite{BT21} that motivated our construction.
	\begin{example} \label{x:csst}
		If $\branch(k)=3$ for all $k \in \mathbb{Z}$, then the metric spaces $(V(T_{m,0}),2^md_{m,0})$ converge in the Gromov-Hausdorff topology to a compact metric space as $m \to -\infty$. This limit is called the continuum self-similar tree (see Figure \ref{fig.csst}). One reason for its importance is that Aldous' continuum random tree is almost surely homeomorphic to the continuum self-similar tree \cite{BT21}.
		\begin{figure}\centering
			\includegraphics[height=110pt]{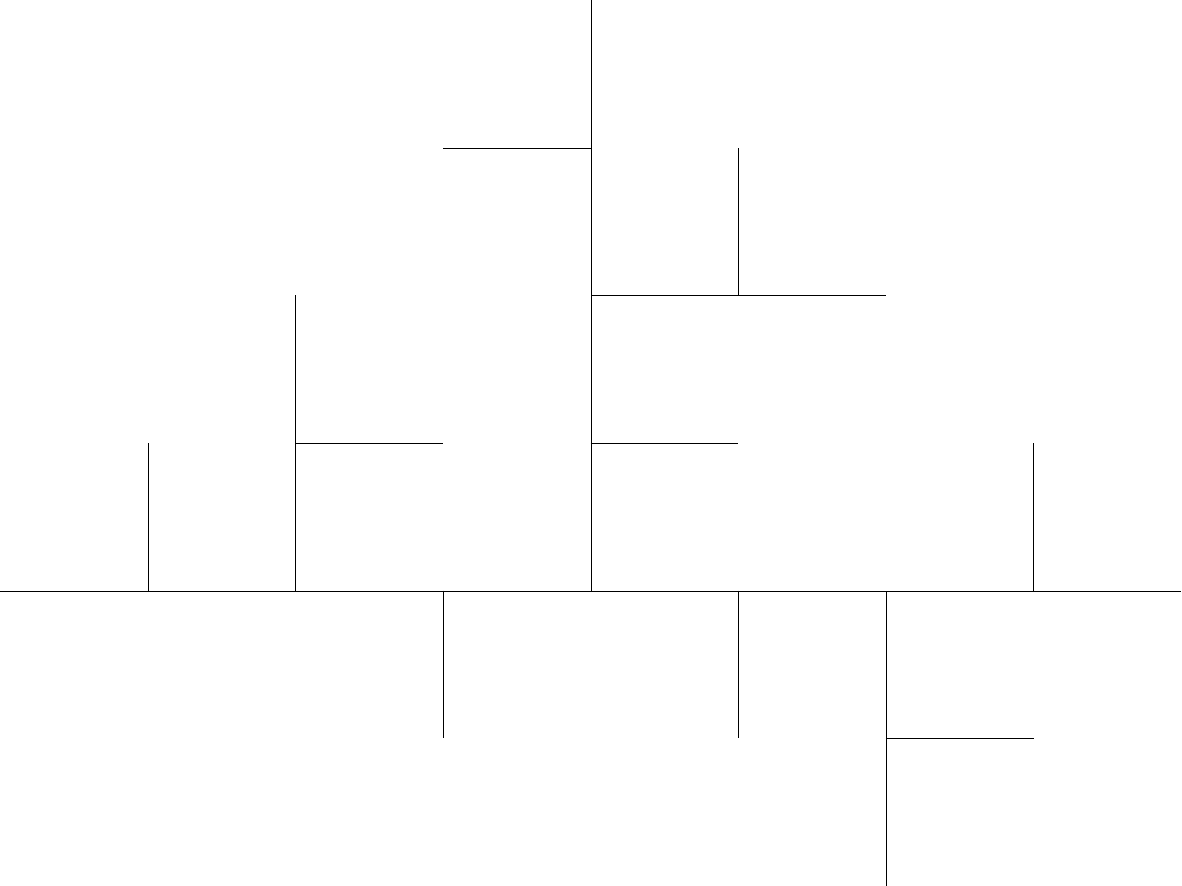}\hspace*{35pt}
			\includegraphics[height=110pt]{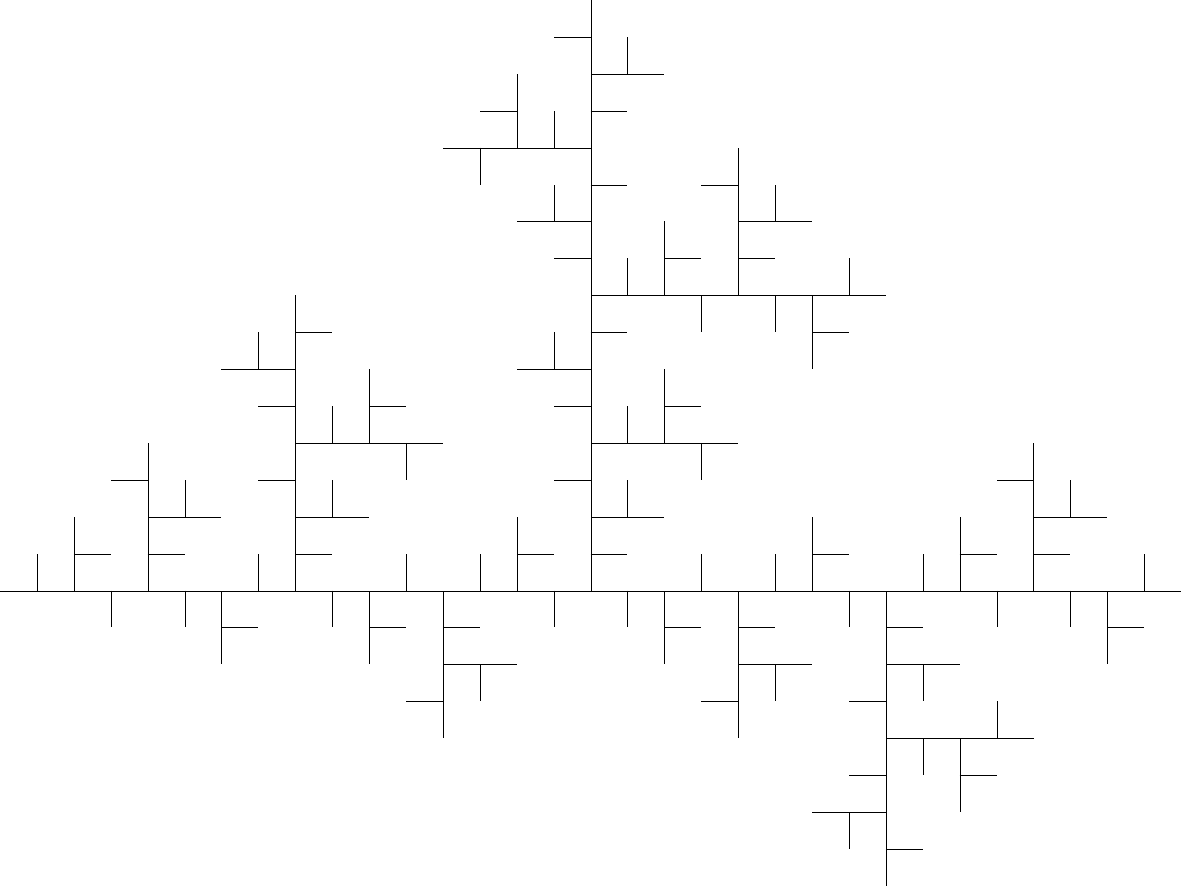}\hspace*{35pt} \\
			\includegraphics[height=110pt]{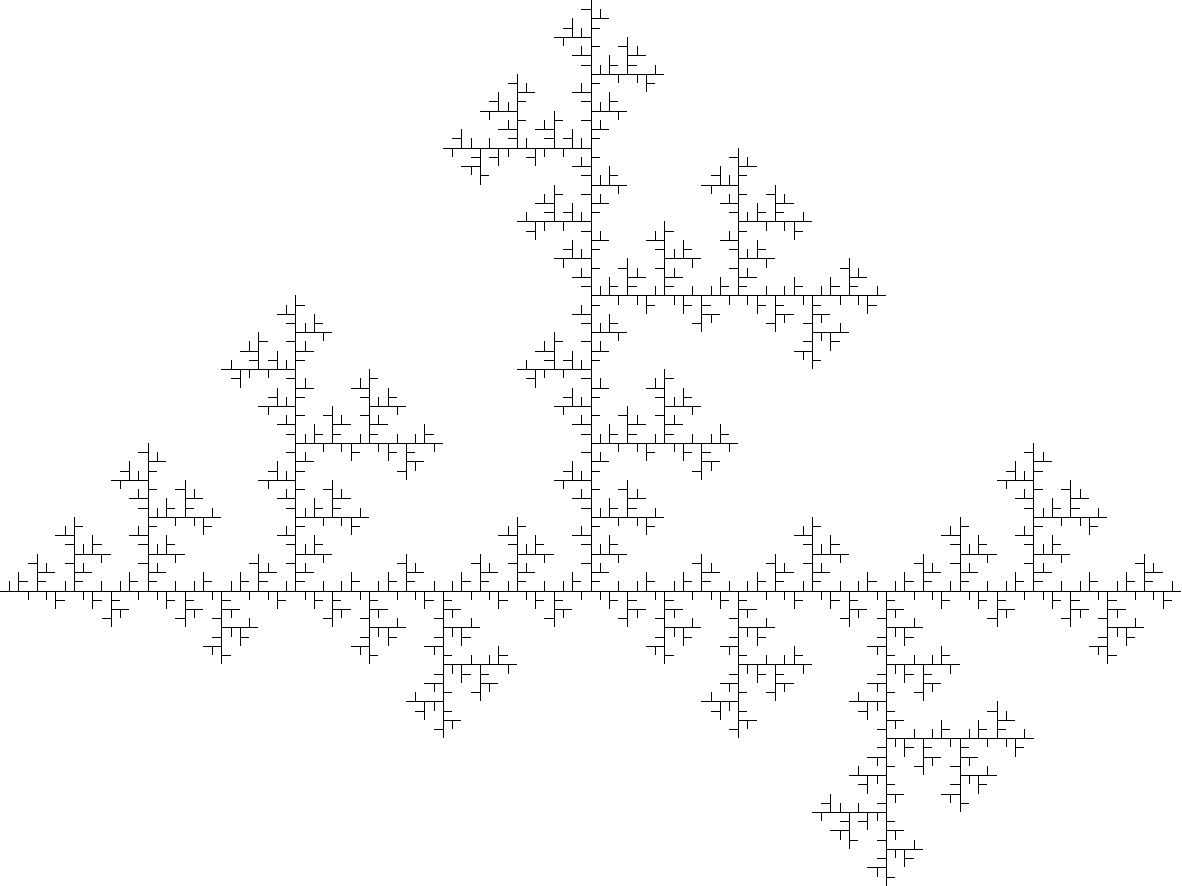}\hspace*{35pt}
			\includegraphics[height=110pt]{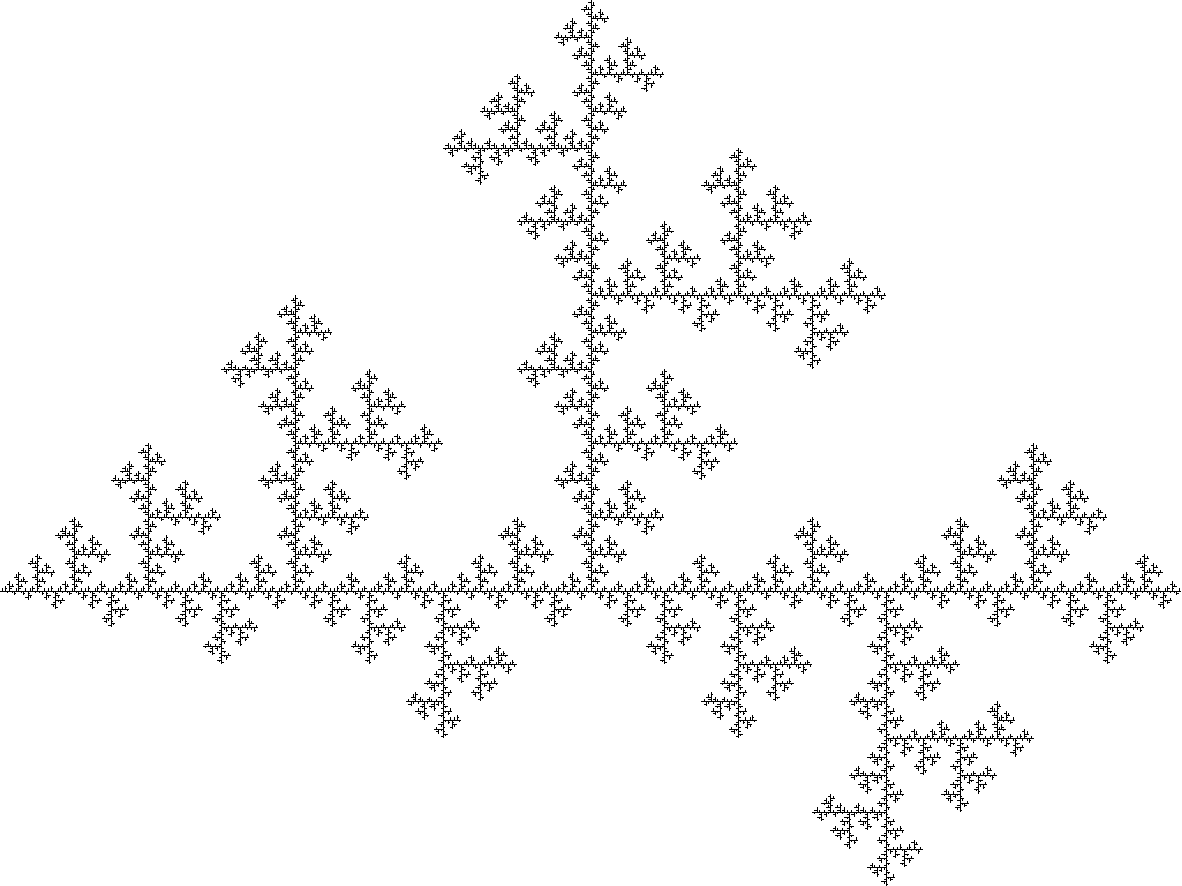}
			\caption{The graphs $T_{-3,0}, T_{-5,0}$, $T_{-7,0}$  and $T_{-9,0}$ if $\branch(k)=3$ for all $k \in \mathbb{Z}$.}\label{fig.csst} 
		\end{figure} 
	\end{example}
	
	\subsection{$\mathbb{R}$-tree as a limit of finite graphs} \label{ss:rtree}
	Lemmas \ref{l:decmtree} and \ref{l:incntree} imply that the re-scaled metric space $(V(T_{m,n}),2^m \mathsf{d}_{m,n})$  naturally embeds isometrically in the  spaces $(V(T_{m,n+1}),2^m \mathsf{d}_{m,n+1})$ and $(V(T_{m-1,n}),2^{m-1}\mathsf{d}_{m-1,n})$. 
	Our goal is to take limits as $m \to -\infty$ and $n \to \infty$ (see Figure \ref{fig.csst}). To this end, we define the set of `infinite' labels 
	\[
	S_0(T_{-\infty,\infty}):= \{\mathbf{s}: \mathbb{Z} \to \mathbb{Z} \mid \mathbf{s}(k) \in \llbracket 0, \branch(k)-1 \rrbracket \mbox{ for all $k \in \mathbb{Z}$, and  $\supp(\mathbf{s})$ is finite}\}.
	\]
	The equivalence relations $R_{m,n}$ on $S(T_{m,n})$ induce an equivalence relation $R_\infty$ on $S_0(T_{-\infty,\infty})$ as follows. We say that $(\mathbf{s},\mathbf{t}) \in R_\infty$ if there exists $N \in \mathbb{N}$ such that 
	\[
	\supp(\mathbf{s}) \cup 	 \supp(\mathbf{t}) \subset \llbracket -(N-1),N \rrbracket, \quad \mbox{and } \restr{\mathbf{s}}{\llbracket -N,N \rrbracket}R_{-N,N}\restr{\mathbf{t}}{\llbracket -N,N \rrbracket}.
	\]
	It is easy to check that  $R_\infty$ is an equivalence relation on $S_0(T_{-\infty,\infty})$.
	For $\mathbf{s} \in S_0(T_{-\infty,\infty})$, we denote the equivalence class of $\mathbf{s}$ under the equivalence relation $R_\infty$ as $[\mathbf{s}]_\infty$. The set of equivalence classes of $S_0(T_{-\infty,\infty})$ with respect to  $R_\infty$ as $T_\infty:=  S_0(T_{-\infty,\infty})/R_\infty$.
	We define a suitable metric $\mathsf{d}_{T_\infty}: T_\infty \times T_\infty \to [0,\infty)$ as 
	\begin{equation} \label{e:defdinf}
		\mathsf{d}_{T_\infty}([\mathbf{s}]_\infty,[\mathbf{t}]_\infty)= 2^{-N} \mathsf{d}_{-N,N}\left(\left[\restr{\mathbf{s}}{\llbracket -N, N \rrbracket}\right]_{-N,N}, \left[\restr{\mathbf{t}}{\llbracket -N, N \rrbracket}\right]_{-N,N} \right),
	\end{equation}
	where $N \in \mathbb{N}$ is chosen large enough so that $ \supp(\mathbf{s}) \cup 	 \supp(\mathbf{t}) \subset \llbracket -(N-1),N \rrbracket$.
	The following lemma is an immediate consequence of Lemmas \ref{l:incntree}(b) and \ref{l:decmtree}.
	For $\mathbf{u}: \llbracket n+1,\infty \rrbracket$ such that $\mathbf{u}(k) \in \llbracket 0, \branch(k)-1 \rrbracket$ for all $k \in \llbracket n+1,\infty \rrbracket$, we set 
	\begin{equation} \label{e:defTinfu}
		T_\infty^{\mathbf{u}}:= \left\{ [\mathbf{s}]_\infty: \mathbf{s} \in S_0(T_{-\infty,\infty}), \restr{\mathbf{s}}{\llbracket n+1 ,\infty \rrbracket}= \mathbf{u}\right\}.
	\end{equation}
	The following lemma shows an embedding of $T_{m,n}$ in $T_\infty$ and basic properties of the sets $T_\infty^{\mathbf{u}}$ defined above.
	\begin{lemma} \label{l:inftree}
		\begin{enumerate}[(a)]
			\item 	Let $m \le n$ with $m,n \in \mathbb{Z}$.
			The equation \eqref{e:defdinf} gives a well-defined metric on $T_\infty$. The map $I_{m,n}: (V(T_{m,n}),2^m \mathsf{d}_{m,n} ) \to (T_\infty,\mathsf{d}_{T_\infty})$ defined by 
			\[
			I_{m,n}^\infty([\mathbf{s}]_{m,n})= [\wt{\mathbf{s}}]_\infty, \mbox{where $\wt{\mathbf{s}} \in S_0(T_{-\infty,\infty}), \restr{\wt{\mathbf{s}}}{\llbracket m ,n \rrbracket}=\mathbf{s}$ and $\supp(\wt{\mathbf{s}}) \subset \llbracket m ,n \rrbracket$}
			\]
			is an isometry. 
			\item For any finitely supported function  $\mathbf{u}: \llbracket n+1,\infty \rrbracket \to \mathbb{Z}$ where $n \in \mathbb{Z}$, we have	$\diam(T_\infty^{\mathbf{u}}, \mathsf{d}_{T_\infty}) \le 2^n$.	
			\item Let $\mathbf{u}, \mathbf{v}: \llbracket n+1,\infty \rrbracket \to \mathbb{Z}$ be finitely supported functions such that $\mathbf{u}(k), \mathbf{v}(k) \in \llbracket 0, \branch(k)-1 \rrbracket$ for all $k \in \llbracket n+1,\infty \rrbracket$  with $\mathbf{u} \neq \mathbf{v}$. Then exactly one of the alternatives hold:
			\begin{enumerate}[(i)]
				\item If $T^\mathbf{u}_\infty \cap T^\mathbf{v}_\infty \neq \emptyset$, then $T^\mathbf{u}_\infty \cap T_\infty^\mathbf{v}= \{ [\mathbf{s}]_\infty\}$ where 
				$\mathbf{s} \in S_0(T_{-\infty,\infty})$ is such that $\mathbf{s}(k)=0$ for all $k < n$ and $\mathbf{s}(n) \in \{0,1\}$. In case $\mathbf{s}(n)=1$, then $\mathbf{s}(k)=\mathbf{u}(k)=\mathbf{v}(k)$ for all $k > n+1$ and $\mathbf{u}(n+1) \neq \mathbf{v}(n+1)$. In the case $\mathbf{s}(n)=0$, then there exist $l \in \llbracket n+2, \infty \rrbracket$ such that $\mathbf{s}(k)=\mathbf{u}(k)=\mathbf{v}(k)$ for all $k \in \llbracket n,\infty\rrbracket \setminus \{l\}$,  $\mathbf{s}(l-1)=\mathbf{u}(l-1)=\mathbf{v}(l-1)=1$ and $\mathbf{s}(k)=\mathbf{u}(k)=\mathbf{v}(k)=0$ for all $k \le l-2$.
				\item If $T^\mathbf{u}_\infty \cap T^\mathbf{v}_\infty = \emptyset$, then $\mathsf{d}_{\infty}([\mathbf{s}]_\infty,[\mathbf{t}]_\infty) \ge 2^n$ for all $[\mathbf{s}]_\infty \in T^\mathbf{u}_\infty,[\mathbf{t}]_\infty \in T^\mathbf{v}_\infty$.
			\end{enumerate}
		\end{enumerate}
	\end{lemma}
	
	\begin{proof}
		\begin{enumerate}[(a)]
			\item As a immediate consequence of Lemmas \ref{l:incntree}(b) and \ref{l:decmtree}. 
			\item 
			Using  (a), any pair of points in $(T_\infty^{\mathbf{u}},\mathsf{d}_{T_\infty})$ is isometric to a pair of points in $(T_{m,n},2^m \mathsf{d}_{m,n})$ for some $m < n$ small enough (where $m$ is allowed to depend on the pair of points). This along with \ref{l:incntree}(c) implies $\diam(T_\infty^{\mathbf{u}}, \mathsf{d}_{T_\infty}) \le 2^n$.
			\item 
			\begin{enumerate}[(i)]
				\item If  $T^\mathbf{u}_\infty \cap T^\mathbf{v}_\infty \neq \emptyset$, then there exists $\mathbf{s}, \mathbf{t} \in S_0(T_{-\infty,\infty})$ such that 
				$\restr{\mathbf{s}}{\llbracket n+1,\infty \rrbracket}= \mathbf{u} \neq \mathbf{v}= \restr{\mathbf{t}}{\llbracket n+1,\infty \rrbracket}$ with $[\mathbf{s}]_\infty = [\mathbf{t}]_\infty$. By the definition of the equivalence relation $R_\infty$ it implies that there exists $l \in \mathbb{Z}$ such that $\mathbf{s}(l) \neq \mathbf{t}(l), \mathbf{s}(l-1)=\mathbf{t}(l-1)=1$, $\mathbf{s}(k)=\mathbf{t}(k)=0$ for all $k<l-1$ and $\mathbf{s}(k)=\mathbf{t}(k)$ for all $k >l$. There are two cases $l=n+1$ and $l>n+1$ which leads to the desired dichotomy.
				\item Pick any pair of points $\mathbf{s}, \mathbf{t} \in S_0(T_{-\infty,\infty})$ such that 
				$\restr{\mathbf{s}}{\llbracket n+1,\infty \rrbracket}= \mathbf{u} \neq \mathbf{v}= \restr{\mathbf{t}}{\llbracket n+1,\infty \rrbracket}$. It suffices to show that $\mathsf{d}_{T_\infty}([\mathbf{s}]_\infty,[\mathbf{t}]_\infty) \ge 2^n$. There exist $\wt{n},m \in \mathbb{Z}$ such that $\supp(\mathbf{s}) \cup \supp(\mathbf{t}) \subset \llbracket m+1, \wt{n} \rrbracket$. Then $[\restr{\mathbf{s}}{\llbracket m, \wt{n}\rrbracket}]_{m, \wt{n}}, [\restr{\mathbf{t}}{\llbracket m, \wt{n}\rrbracket}]_{m, \wt{n}} \in V(T_{m,\wt{n}})$ and $2^m \mathsf{d}_{m,\wt{n}}\left([\restr{\mathbf{s}}{\llbracket m, \wt{n}\rrbracket}]_{m, \wt{n}}, [\restr{\mathbf{t}}{\llbracket m, \wt{n}\rrbracket}]_{m, \wt{n}}\right)= \mathsf{d}_{T_\infty}([\mathbf{s}]_\infty, [\mathbf{t}]_\infty)$. Then the sequence of vertices on the shortest path between $[\restr{\mathbf{s}}{\llbracket m, \wt{n}\rrbracket}]_{m, \wt{n}}$ and $[\restr{\mathbf{t}}{\llbracket m, \wt{n}\rrbracket}]_{m, \wt{n}}$ in $V(T_{m,\wt{n}})$ defines a corresponding sequence of points in $T_\infty$. Let $\mathbf{w} : \llbracket n + 1 , \infty \rrbracket \to \mathbb{Z}$ be a finitely supported function such that $\mathbf{w}$ is finitely supported and $\mathbf{w}(k) \in \llbracket 0, \branch(k)-1\rrbracket$ and such $\mathbf{w} \notin \{\mathbf{u},\mathbf{v}\}$ that the corresponding sequence of points in $T_\infty$ from contains at least two points in $T^{\mathbf{w}}_\infty$ (such points exist by (i)). Again by (a), (c)-(i) and Lemma \ref{l:incntree}(c), this path must have $\mathsf{d}_{T_\infty}$-length at least $2^n$. 
			\end{enumerate}
			
		\end{enumerate}
	\end{proof}
	
	The desired $\mathbb{R}$-tree  is obtained by completing $(T_\infty,\mathsf{d}_{T_\infty})$ as defined below.
	\begin{definition}
		Let $(\tree(\branch),\metrictree)$ denote the  completion of the metric space $(T_\infty,\mathsf{d}_{T_\infty})$.
		We denote the open and closed balls centered at $x \in \tree(\branch)$ with radius $r>0$ as $B_{\tree(\branch)}(x,r)$ and $\overline{B}_{\tree(\branch)}(x,r)$.
		% We also denote the space $(\tree,\metrictree)$ as $(\tree(\branch), \metrictree)$ if we would like to emphasize the dependence of the branching function $\branch$.
	\end{definition}
	There is a natural `coding' map $\codetree:(\ultra(\branch), \mathsf{d}_{\ultra(\branch)}) \to (\tree(\branch), \metrictree)$ that we describe  below. 
	
	\begin{prop} \label{p:dsr}
		\begin{enumerate}[(a)]
			\item  For any $\mathbf{s} \in \ultra(\branch), n \in \mathbb{N}$, define $\mathbf{s}_n \in S_0(T_{-\infty,\infty})$ as
			\[
			\mathbf{s}_n(k)= \begin{cases}
				\mathbf{s}(k) & \mbox{if $k > -n$,}\\
				0 & \mbox{if $k \le -n$.}
			\end{cases}
			\]
			Then $([\mathbf{s}_n]_\infty)_{n \in \mathbb{N}}$ is a Cauchy sequence in $(T_\infty,\mathsf{d}_{T_\infty})$ and hence defines a map $\codetree:\ultra(\branch) \to \tree(\branch)$ as 
			\[
			\codetree(\mathbf{s})= \lim_{k \to \infty} [\mathbf{s}_k]_\infty.
			\]
			\item The map $\codetree$ is surjective.
			\item The map $\codetree:(\ultra(\branch), \mathsf{d}_{\ultra(\branch)}) \to (\tree(\branch),\metrictree)$ is David-Semmes regular.

		\end{enumerate}
	\end{prop}
	\begin{proof}
		\begin{enumerate}[(a)]
			\item By Lemma \ref{l:inftree}(b), $\mathsf{d}_{T_\infty}([\mathbf{s}_n]_\infty, [\mathbf{s}_m]_\infty) \le 2^{-(\min(m,n))}$ for all $m,n \in \mathbb{N}$. This implies that  $([\mathbf{s}_n]_\infty)_{n \in \mathbb{N}}$ is a Cauchy sequence.

			\item
			If $[\mathbf{s}]_\infty \in T_\infty$, then $\mathbf{s} \in \ultra(\branch)$ and $\codetree(\mathbf{s})$ can be identified with $[\mathbf{s}]_\infty$. Hence it suffices to consider the set $\tree(\branch)\setminus T_\infty$ (where we identify $T_\infty$ as a subset of $\tree(\branch)$ in the usual way).

			Let $([\mathbf{s}_n]_\infty)_{n \in \mathbb{N}}$ be a Cauchy sequence in $(T_\infty,\mathsf{d}_{T_\infty})$ converging to an element in $\tree(\branch) \setminus T_\infty$. 
			Then by passing to a subsequence we may assume that $\mathbf{s}_n(k)$ converges for each $k \in \mathbb{Z}$. Let $\mathbf{s}(k)= \lim_{n \to \infty}\mathbf{s}_n(k)$ for all $k \in \mathbb{Z}$.
			For $l \in \mathbb{Z}$, let $N_l \in \mathbb{N}$ be such that $\mathsf{d}_{T_\infty}([\mathbf{s}_n]_\infty,[\mathbf{s}_m]_\infty) < 2^{-l-1}$ for all $n,m \ge N_l$. Then by Lemma \ref{l:inftree}(b), we have $T_{\infty}^{\restr{\mathbf{s}_n}{\llbracket - l,\infty \rrbracket}} \cap T_{\infty}^{\restr{\mathbf{s}_m}{\llbracket -l,\infty \rrbracket}} \neq \emptyset$. Then by Lemma \ref{l:inftree}(c), there are  only finitely many
			possibilities $\restr{\mathbf{s}_n}{\llbracket -l,\infty \rrbracket}$ for $n \ge N_l$ and for each $l \in \mathbb{N}$. Therefore by the pointwise convergence of $\mathbf{s_n}$, we have that $\restr{\mathbf{s}_n}{\llbracket- l,\infty \rrbracket}$ is eventually constant for each $l \in \mathbb{N}$. This implies that $\mathbf{s} \in \ultra(\branch)$ and $\codetree(\mathbf{s})= \lim_{k \to \infty} [\mathbf{s}_k]_\infty$.
			\item By Lemma \ref{l:inftree}(b)
			the map $\codetree$ is $1$-Lipschitz. 
			
			Since $\codetree$ is surjective any element of $\tree(\branch)$ is of the form $\codetree(\mathbf{s})$ for some $\mathbf{s} \in \ultra(\branch)$. Let $r>0$ and let $n \in \mathbb{Z}$ be such that $2^{n}< r\le2^{n+1}$ for some $n \in \mathbb{Z}$. The $\codetree^{-1}(B_{{\tree(\branch)}}(\codetree(\mathbf{s},r)))$ is covered by balls of the form $B_{{\ultra(\branch)}}(\codetree(\mathbf{t},2^{n+1}))$ where $\mathbf{t}$ is such that $$T_\infty^{\restr{t}{\llbracket n+1,\infty \rrbracket}} \cap T_\infty^{\restr{s}{\llbracket n+1,\infty \rrbracket}} \neq \emptyset.$$ Therefore by Lemma \ref{l:inftree}(c)-(i), $\codetree^{-1}(B_{{\tree(\branch)}}(\codetree(\mathbf{s},r)))$  can be covered by at most $2M$ balls of radii $2^{n+1}$ in $(\ultra(\branch), \mathsf{d}_{\ultra(\branch)})$, where $M_\branch= \sup_{k \in \mathbb{Z}} \branch(k)$. Since each ball of radius $2^{n+1}$ is an union of at most $M_\branch$ balls of radii $2^n$, we conclude that $\codetree^{-1}(B_{{\tree(\branch)}}(\codetree(\mathbf{s},r)))$  can be covered by at most $2M_\branch^2$ balls of radii $r$ in $(\ultra(\branch), \mathsf{d}_{\ultra(\branch)})$.
		\end{enumerate}
	\end{proof}
	We now define a measure on the tree $(\tree,\metrictree)$ as the push-forward measure
	\begin{equation} \label{e:defmeastree}
		\meastree=   \codetree_*(\mathsf{m}_{\ultra(\branch)}), 
	\end{equation}
	where $\mathsf{m}_{\ultra(\branch)}$ is as defined in \eqref{e:defultrameas} (see also \eqref{e:measultraformula}) and $\codetree$ is the coding map in Proposition \ref{p:dsr}. Recall that a metric space $(X,d)$ is said be be \textbf{proper}, if every closed ball
	\[
	\overline{B}(x,r)=\{y \in X \vert d(x,y) \le r\}
	\]
	is compact for all $x \in X, r>0$.

	The following estimate of $	\meastree$ follows from the David-Semmes regularity of $\codetree$ (recall \eqref{e:defvol}). 
	\begin{cor} \label{c:meastree}
		There exists $C>0$ such that for all $\mathbf{s} \in \ultra(\branch)$ and for all $r>0$, we have
		\begin{equation*}
			C^{-1} V_\branch(r) \le  \meastree(B_{{\tree(\branch)}}(\codetree(\mathbf{s}),r)) \le C  	V_\branch(r).
		\end{equation*}
		In particular, $\meastree$ is a doubling measure on $(\tree(\branch),\metrictree)$ and $(\tree(\branch),\metrictree)$ is a proper, separable, metric space. 
	\end{cor}
	
	\begin{proof}
		The estimate on $\meastree(B_{{\tree(\branch)}}(\codetree(\mathbf{s}),r))$ is a consequence of Proposition \ref{p:dsr}(b,c) along with Lemma \ref{l:volume}. The fact that $\meastree$ is a doubling measure  follows from \eqref{e:measultraformula} and $\sup_{k \in \mathbb{Z}} \branch(k)<\infty$. It is well-known that every complete metric space that carries a doubling measure is proper. This follows by observing that every closed ball in a doubling metric space is totally bounded; cf.~\cite[Exercise 10.17, Definition 10.15, p. 82]{Hei}. This total boundedness of metric balls also implies that the space is separable.	
	\end{proof}
	
	\begin{definition} \label{d:rtree}
		A metric space $(X,d)$ is a $\mathbb{R}$-tree if it satisfies the following conditions.
		\begin{enumerate}
			\item (0-hyperbolic) For any $x,y,z,t \in X$, we have
			\[
			d(x,y)+d(z,t) \le \max \left( d(x,z)+d(y,t), d(x,t) +d(y,z) \right).
			\]
			\item (geodesic) For any pair of points $x,y \in X$, there exists a continuous map (also called a \emph{curve})  $\gamma:[a,b] \to (X,d)$ such that $\gamma(a)=x,\gamma(b)=y$ and its \emph{length} $L(\gamma)=d(x,y)$, where the \textbf{length of a curve} $\gamma:[a,b] \to X$ is defined by
			$$L(\gamma):= \sup \Biggl\{ \sum_{i=0}^k \gamma(t_i): k \in \mathbb{N}, a=t_0< t_1 \ldots < t_k=b \Biggr\}.$$
			Such a curve $\gamma$ is said to be a \emph{geodesic} between $x$ and $y$.
		\end{enumerate}
	\end{definition}	
	We recall the definition of quasiconvexity which is a more flexible version of the geodesic property.
	\begin{definition}\label{d:qgeodesic}
		A metric space $(X,d)$ is said to be \textbf{quasiconvex} if there exists a constant $C_q \in [1,\infty)$ such that every pair of points $x,y \in X$ can be connected by a curve $\gamma$ whose length $L(\gamma)$ satisfies $L(\gamma) \le C_q d(x,y)$. 
		A curve $\gamma:[a,b] \to X$ is said to be a \textbf{geodesic} from $x$ to $y$, if $\gamma(a)=x, \gamma(b)=y$ and $L(\gamma)$ equals in the infimum of the length of any curve from $x$ to $y$. A curve $\gamma:[a,b] \to X$ is said to be \textbf{parameterized by arc length} if for any $s,t \in [a,b]$ such that $s<t$, we have $t-s= L\left( \restr{\gamma}{[s,t]}\right)$.
	\end{definition}

	We show that $(\tree(\branch),\metrictree)$ is a $\mathbb{R}$-tree. 
	
	\begin{prop} \label{p:rtree}
		The space $(\tree(\branch),\metrictree)$ is a $\mathbb{R}$-tree. 
	\end{prop}
	
	\begin{proof}
		Given any four points in $T_\infty$, we can view them as contained in $V(T_{m,n})$ for some $m \le n$ with $m,n \in \mathbb{Z}$ using the isometry in Lemma \ref{l:inftree}. By
		Lemma \ref{l:incntree}(c) and the fact that every tree is a $0$-hyperbolic metric space (cf. \cite[Exemple 5(i), Chapitre 2]{GH}), we have that  $(T_\infty, \mathsf{d}_{T_\infty})$ is $0$-hyperbolic from which it follows that its completion $(\tree(\branch),\metrictree)$  is $0$-hyperbolic.
		
		Each pair of points in $x,y \in (T_\infty,\mathsf{d}_{T_\infty})$, there is a unique midpoint $z\in x,y$ such that $d(z,x)=d(z,y)=d(x,y)/2$. This follows from viewing $x,y$ as points in $V(T_{m,n})$ for some $m \le n$ with $m,n \in \mathbb{Z}$ and then since midpoint of each  edge in $T_{m,n}$ belongs to $(V(T_{m-1,n}),2^{m-1}\mathsf{d}_{m-1,n})$, we have the midpoint property. This implies that every pair of points in  $T_\infty$ has a geodesic connecting them in the completion $(\tree(\branch),\metrictree)$.  For arbitrary pair of points in $(\tree(\branch),\metrictree)$ we can take limit of geodesics with endpoints in $T_\infty$. Such a limit exists as a geodesic due to Arzela-Ascoli theorem (see \cite[Theorem 2.5.14]{BBI}) which can be applied by using the fact that every closed ball is compact (see Corollary \ref{c:meastree}).
	\end{proof}
	\subsection{A Laakso-type space} \label{ss:laak}
	
	Let us fix the base point $p_\tree \in \tree(\branch)$ defined by  \begin{equation} \label{e:bptree}
		p_\tree=\codetree(p_{\ultra(\branch)})
	\end{equation} where $p_{\ultra(\branch)}:\mathbb{Z} \to \mathbb{Z}$ is the function in $\ultra(\branch)$ that is identically zero and   $\codetree$ is as given in Proposition \ref{p:dsr}(a).
	Define `level-$n$ wormholes' $\worm_n \subset \tree(\branch)$ for each $n \in \mathbb{Z}$ as  
	\begin{align} \label{e:defwormhole}
		\worm_n &= \{ \codetree([\mathbf{s}]_\infty) \vert \mathbf{s}  \in S_0(T_{-\infty,\infty}), \mathbf{s}(n)=1, \mathbf{s}(k)=0 \mbox{ for all $k < n$} \} \nonumber \\
		&= \{  \mathbf{t} \in \tree(\branch) \vert 2^n \metrictree(\mathbf{t}, p_\tree )  \in \mathbb{N} \setminus 2 \mathbb{N} \},
	\end{align}
	where $\mathbb{N} \setminus 2 \mathbb{N}$ is the set of positive odd integers. The equivalence between the two definitions in \eqref{e:defwormhole} follows from Lemma \ref{l:incntree}(a).
	
	For a reader interested in Gaussian bounds (that is $\Psi(r)=r^2$), it suffices to consider the following tree that corresponds to $\branch \equiv 2$.
	\begin{example} \label{x:tb}
		If the branching function is identically two, then the metric measure space $(\tree(\branch),\metrictree,\meastree)$ can be identified with $[0,\infty)$ equipped with Euclidean metric and the   Lebesgue measure. Under this identification, the base point $p_\tree$ is $0$ and the level-$n$ wormholes $\worm_n$ is the set $\{(2k-1) 2^n: k \in \mathbb{N}\}$.
	\end{example}
	
	We record some basic properties of the family of sets $(\worm_n)_{n \in \mathbb{Z}}$ in the following lemma.
	\begin{lemma} \label{l:worm}
		\begin{enumerate}[(i)]
			\item For any $\mathbf{t} \in \tree(\branch), n \in \mathbb{Z}$, we have $\dist(\mathbf{t},\worm_n)\le 2^n$.
			\item For any $n,m \in \mathbb{Z}, \mathbf{s} \in \worm_m, \mathbf{t} \in \worm_n$ with $\mathbf{s}\neq \mathbf{t}$, we have $\metrictree(\mathbf{s},\mathbf{t}) \ge 2^{\min(m,n)}$.
		\end{enumerate}
	\end{lemma}
	\begin{proof}
		\begin{enumerate}[(i)]
			\item If  $\mathbf{t} = [\mathbf{s}]_\infty \in( T_\infty,\mathsf{d}_{T_\infty})$, then the conclusion follows from Lemma \ref{l:inftree}(b) by choosing the point $\codetree([\wt{\mathbf{s}}]_\infty) \in \worm_n$, where 
			\[
			\restr{\wt{\mathbf{s}}}{\llbracket n+1,\infty \rrbracket}=		\restr{\mathbf{s}}{\llbracket n+1,\infty \rrbracket}, \quad \wt{\mathbf{s}}(n)=1, \quad \wt{\mathbf{s}}(k)=0 \mbox{ for all $k<n$.}
			\]
			Hence this property extends  to the completion $(\tree(\branch),\metrictree)$.
			
			\item Without loss of generality, let $m \le n$. By Lemma \ref{l:inftree}(i), there exists $k \in \mathbb{Z}$ with $k>n$ such that there is an isometric copy of $\mathbf{s}$ and $\mathbf{t}$ in $V(T_{m,k},2^m\mathsf{d}_{m,k})$. Since the distance between any two distinct vertices with respect to $\mathsf{d}_{m,k}$ is at least one, we obtain the desired conclusion.
		\end{enumerate}
	\end{proof}
	
	Next, we define the Laakso-type space as a quotient space of the product space $\mathcal{P}(\glue,\branch):=\ultra(\glue) \times \tree(\branch)$ with respect to an equivalence relation.
	To this end, we define equivalence relation $R_\laak$ on $\mathcal{P}(\glue,\branch)$ as follows. Observe that $(\mathbf{s},\mathbf{t}),(\tilde{\mathbf{s}},\tilde{\mathbf{t}}) \in \mathcal{P}(\glue,\branch)$ are related in $R_{\laak}$ if and only if either $(\mathbf{s},\mathbf{t})=(\tilde{\mathbf{s}},\tilde{\mathbf{t}})$ or if $\mathbf{t}=\tilde{\mathbf{t}} \in \worm_n$ for some $n \in \mathbb{Z}$ and $\restr{\mathbf{s}}{\mathbb{Z} \setminus \{n\}} = \restr{\tilde{\mathbf{s}}}{\mathbb{Z} \setminus \{n\}}$. 
	It is clear that $R_{\laak}$ is an equivalence relation in $\mathcal{P}(\glue,\branch)$  and that each equivalence class contains at most $M_\glue:= \sup_{k \in \mathbb{Z}} \glue(k)$ elements. 
	For $\mathbf{p} \in \mathcal{P}$, we denote the  equivalence class with respect to $R_{\laak}$ containing $\mathbf{p}$ as $[\mathbf{p}]_{\laak}$.
	The collection of equivalence classes $\mathcal{P}(\glue,\branch)/ R_{\laak}$ is called the \textbf{Laakso-type space} $\mathcal{L}(\glue,\branch)$. Let $\quolaak: \mathcal{P}(\glue,\branch) \to  \laak(\glue,\branch)$ denote the cannonical surjective map
	given by
	\begin{equation} \label{e:defquolaak}
		\quolaak(\mathbf{p})= [\mathbf{p}]_\laak, \quad \mbox{for all $\mathbf{p} \in \mathcal{P}(\glue,\branch)$.}
	\end{equation}
	%In Proposition \ref{p:laak}(b), we will see that $\quolaak$ is the quotient map corresponding to the equivalence relation $R_\laak$. 
	Since $\left((\mathbf{u},\mathbf{s}), (\mathbf{v},\mathbf{t}) \right) \in R_\laak$ implies $\mathbf{s}=\mathbf{t}$, there is a well-defined projection map
	$\projtree: \laak(\glue,\branch) \to \tree(\branch)$ defined as $\projtree([(\mathbf{u},\mathbf{t})]_\laak)=\mathbf{t}$.

	Our immediate goal is construct suitable metrics and measures on the Laakso-type space $\mathcal{L}(\glue,\branch)$ and the product space $\mathcal{P}(\glue,\branch)$. On the space  $\mathcal{P}(\glue,\branch)$, we  define the metric 
	\begin{equation} \label{e:defmetprod}
		\mathsf{d}_{\mathcal{P}(\glue,\branch)}\left( (\mathbf{s},\mathbf{t}),(\tilde{\mathbf{s}},\tilde{\mathbf{t}}) \right)= \max\left(\metricultra(\mathbf{s}, \tilde{\mathbf{s}}),\metrictree(\mathbf{t}, \tilde{\mathbf{t}}) \right),
	\end{equation}
	and the product measure
	\begin{equation} \label{e:defmeasprod}
		\mathsf{m}_{\mathcal{P}(\glue,\branch)}:=	\mathsf{m}_{\ultra(\glue)} \times 	\mathsf{m}_{\mathcal{\tree}(\branch)}.
	\end{equation}
	There is a natural choice of metric on the quotient space $\laak(\glue,\branch)$  that we recall now  \cite[Definition 3.1.12]{BBI}.
	To this end, we define a function $\metlaak: \laak(\glue,\branch) \times \laak(\glue,\branch) \to [0,\infty)$ as 
	\begin{equation} \label{e:defmetlaak}
		\metlaak([\mathbf{p}]_{\laak},[\mathbf{q}]_{\laak})= \inf  \sum_{i=1}^k \mathsf{d}_{\mathcal{P}(\glue,\branch)}(\mathbf{p}_i, \mathbf{q}_i), 
	\end{equation}
	where the infimum is over all $k \in \mathbb{N}, \mathbf{p}_i, \mathbf{q}_i \in \mathcal{P}(\glue,\branch)$ for all $i=1,\ldots,k$ such that  $\mathbf{p}_1=\mathbf{p}, \mathbf{q}_k=\mathbf{q}, (\mathbf{p}_{i+1}, \mathbf{q}_{i}) \in R_\laak \mbox{ for all $i=1,\ldots,k-1$}$. By triangle inequality and the estimate $\mathsf{d}_{\mathcal{P}(\glue,\branch)}( (\mathbf{u},\mathbf{s}) , (\mathbf{v}, \mathbf{t})  ) \ge \mathsf{d}_{\ultra(\branch)}(\mathbf{s},\mathbf{t})$, we have 
	\begin{equation} \label{e:projmet}
		\metlaak(\mathbf{p},\mathbf{q}) \ge \metrictree \left(\projtree(\mathbf{p}),\projtree(\mathbf{q})\right), \quad \mbox{for all $\mathbf{p},\mathbf{q} \in \laak(\glue,\branch)$.}
	\end{equation}
	We denote open and closed balls with respect to $\metlaak$ by $B_{\laak(\glue,\branch)}(\cdot,\cdot), \overline{B}_{\laak(\glue,\branch)}(\cdot,\cdot)$ respectively.
	We record an elementary lemma for future use.
	\begin{lemma} \label{l:wormlaak}
		Suppose $[(\mathbf{u},\mathbf{s})]_{\laak}, [(\mathbf{v},\mathbf{t})]_{\laak} \in \laak(\glue,\branch)$, $\metlaak\left([(\mathbf{u},\mathbf{s})]_{\laak}, [(\mathbf{v},\mathbf{t})]_{\laak}\right)<r$. Let $n \in \mathbb{Z}$ be such that $r<2^n$ and $B_{\tree(\branch)}(\mathbf{s},r) \cap \worm_n=\emptyset$. Then, we have $\mathbf{u}(n)=\mathbf{v}(n)$.
	\end{lemma}
	\begin{proof}
		Using the definition \eqref{e:defmetlaak}, pick $k \in \mathbb{N}, (\mathbf{u}_i,\mathbf{s}_i), (\mathbf{v}_i,\mathbf{t}_i)$ for $i=1,\ldots,k$ such that $(\mathbf{u}_1,\mathbf{s}_1)=(\mathbf{u},\mathbf{s})$, $(\mathbf{v}_k,\mathbf{t}_k)=(\mathbf{v},\mathbf{t})$ and $ \left((\mathbf{u}_{j+1},\mathbf{s}_{j+1}), (\mathbf{v}_j,\mathbf{t}_j)\right) \in R_\laak$ for all $j=1,\ldots,k-1$ such that 
		\[
		\sum_{i=1}^k \mathsf{d}_{\mathcal{P}(\glue,\branch)} \left( (\mathbf{u}_i,\mathbf{s}_i), (\mathbf{v}_i,\mathbf{t}_i)\right) < r.
		\]
		Hence by \eqref{e:projmet} and $B_{\tree(\branch)}(\mathbf{s},r) \cap \worm_n=\emptyset$, we have
		$\mathbf{t}_i \notin \worm_n, \mathbf{s}_i \notin \worm_n$ for all $i=1,\ldots,k$ and hence $\mathbf{u}_{i+1}(n)=\mathbf{v}_i(n)$ for all $i=1,\ldots,k-1$. Furthermore since  $\mathsf{d}_{\mathcal{P}(\glue,\branch)} \left( (\mathbf{u}_i,\mathbf{s}_i), (\mathbf{v}_i,\mathbf{t}_i)\right)<r<2^n$ for all $i=1,\ldots,k$, we have $\mathbf{u}_i(n)=\mathbf{v}_i(n)$ for all $i=1,\ldots,k$. Combining the above two statements, we obtain the desired conclusion.
	\end{proof}
	%If $[(\mathbf{u}_1,\mathbf{t}_1)]_\laak
	
	For a general quotient space of a metric space, the quotient `metric' as defined in \eqref{e:defmetlaak} is only a semi-metric and need not be a metric \cite[Example 3.1.10 and Exercise 3.1.11]{BBI}. Even if it were a metric, it need not induce the quotient topology as there are examples of quotient spaces  that are not first countable (and hence not metrizable) \cite[Example 3.1.17]{BBI}.  
	We rule out such a pathological behavior and collect  some important properties of 	$\metlaak$ in the proposition below.
	\begin{prop} \label{p:laak}
		\begin{enumerate}[(a)]
			\item The function 	$\metlaak$ defined in \eqref{e:defmetlaak} is a metric on $\laak(\glue,\branch)$ and 
			the map $\quolaak:(\mathcal{P}(\glue,\branch), \mathsf{d}_{\mathcal{P}(\glue,\branch) } )\to (\laak(\glue,\branch),	\metlaak)$ is a continuous, surjective, David-Semmes regular map. In particular, $ (\laak(\glue,\branch),	\metlaak)$ is a complete, locally compact metric space.
			\item The metric $\metlaak$ induces the quotient topology on $\laak(\glue,\branch)$ corresponding to the equivalence relation $R_\laak$.
			\item The metric $\metlaak$ is quasiconvex. In particular, $(\laak(\glue,\branch),\metlaak)$ satisfies the chain condition.
		\end{enumerate}
	\end{prop}
	
	\begin{proof}
		\begin{enumerate}[(a)]
			\item 
			First, we show that $\metlaak$ is a metric. The non-negativity, triangle inequality and symmetry   immediately follows from the definition. 
			If $[\mathbf{p}]_\laak=[\mathbf{q}]_\laak$, then $\metlaak([\mathbf{p}]_\laak,[\mathbf{q}]_\laak)=0$ by choosing $\mathbf{p}_1=\mathbf{q}_1=\mathbf{p}$ and $\mathbf{p}_2=\mathbf{q}_2=\mathbf{q}$ in \eqref{e:defmetlaak}. 
			
			Next, we show the non-degeneracy property; that is, $\metlaak ([(\mathbf{u},\mathbf{s})]_\laak,[(\mathbf{v},\mathbf{t})]_\laak )=0$ implies that $[(\mathbf{u},\mathbf{s})]_\laak=[(\mathbf{v},\mathbf{t})]_\laak$. By \eqref{e:projmet}, we have $\metrictree(\mathbf{s},\mathbf{t})=0$ and hence $\mathbf{s}=\mathbf{t}$. It therefore suffices to consider the case $\mathbf{u} \neq \mathbf{v}$.
			If $\mathbf{u}(k) \neq \mathbf{v} (k)$ for some $k \in \mathbb{Z}$, then by Lemma \ref{l:wormlaak}, we have that $\mathbf{s} \in \worm_k$. Since $\worm_n$ for different values of $n$ are pairwise disjoint, by the previous line, we have  $\restr{\mathbf{u}}{\mathbb{Z} \setminus \{k\}}=\restr{\mathbf{v}}{\mathbb{Z} \setminus \{k\}}$. This implies the non-degeneracy property of the metric. This completes the proof that $\metlaak$ is a metric.
			
			By taking $k=1$ in \eqref{e:defmetlaak}, we obtain that $\quolaak$ is $1$-Lipschitz. Therefore $\quolaak$ is continuous. It is evident that $\quolaak$ is surjective.
			
			It remains to show that $\quolaak$ is David-Semmes regular. 
			First note by \eqref{e:projmet} and Lemma \ref{l:wormlaak} that for all $[(\mathbf{v},\mathbf{t})]_{R_{\laak}} \in \laak(\glue,\branch),r>0$ we have
			\begin{align}\label{e:ballbox}
				\lefteqn{\quolaak^{-1}\left(B_{{\laak(\glue,\branch)}}([(\mathbf{v},\mathbf{t})]_{R_{\laak}},r) \right)}  \nonumber \\
				&\subset \Biggl\{  (\mathbf{u},\mathbf{s}) \in \mathcal{P}(\glue,\branch)\Biggm|
				\begin{minipage}{250pt}
					$\metrictree(\mathbf{s},\mathbf{t})<r$, $\mathbf{u}\in \ultra(\glue)$  satisfies $\mathbf{u}(k)=\mathbf{v}(k)$  for all $k \in \mathbb{Z}$ such that $\worm_k \cap B_{\tree(\branch)}(\mathbf{t},r)=\emptyset$
				\end{minipage}
				\Biggr\}.
			\end{align}
			Let $[(\mathbf{v},\mathbf{t})]_{R_{\laak}} \in \laak(\glue,\branch),r>0$ be arbitrary. Let $n \in \mathbb{Z}$ be the unique integer such that $2^{n} < r \le 2^{n+1}$. Since $(\tree(\branch),\metrictree)$ is a doubling metric space, by Lemma \ref{l:worm}(ii) and \cite[Exercise 10.17]{Hei}, there exists $L \in \mathbb{N}$ (depending only on the metric doubling constant) such that 
			\[
			\# \left( \bigcup_{\stackrel{k \ge n,} {k \in \mathbb{Z}}} \worm_k \cap B_{\tree(\branch)}(\mathbf{t},r) \right) \le L.
			\]

			Since $(\mathcal{P}(\glue,\branch), \mathsf{d}_{\mathcal{P}(\glue,\branch) } )$ is a complete, doubling metric space, all closed and bounded subsets of $\mathcal{P}(\glue,\branch)$ are compact. Since $\quolaak$ is a David-Semmes regular map, for every closed and bounded set $C \subset \laak(\glue,\branch)$, the set $\quolaak^{-1}(C)$ is compact in $\mathcal{P}(\glue,\branch)$. Hence by the continuity of $\quolaak$, every closed and bounded subset of $\laak(\glue,\branch)$ is compact. This proves the local compactness and completeness of $(\laak(\glue,\branch),\metlaak)$.

			\item 
			In order to show that $\metlaak$ induces the quotient topology, it suffices to show that a subset $U$ of $(\laak(\glue,\branch),	\metlaak)$ is open if and only if $\quolaak^{-1}(U)$ is open in $(\mathcal{P}(\glue,\branch), \mathsf{d}_{\mathcal{P}(\glue,\branch) } )$. By (a), since $\quolaak$ is continuous, it suffices to show the if part.
			
			To this end, let us assume that $U \subset \laak(\glue,\branch)$ is such that $\quolaak^{-1}(U)$ is open in $(\mathcal{P}(\glue,\branch), \mathsf{d}_{\mathcal{P}(\glue,\branch) } )$ is open. Let $[(\mathbf{u},\mathbf{s})]_{R_{\laak}} \in U$. 
			Since the equivalence class $[(\mathbf{u},\mathbf{s})]_{R_{\laak}}$ has at most $\sup_{k \in \mathbb{Z}} \glue(k)$ elements and $\quolaak^{-1}(U)$ is open, there exists $r_1>0$ such that 
			\begin{equation} \label{e:plaak1}
				\bigcup_{ (\mathbf{v},\mathbf{s}) \in [(\mathbf{u},\mathbf{s})]_{R_{\laak}}} B_{\mathcal{P}(\glue,\branch)}((\mathbf{v},\mathbf{s}),r_1) \subset \quolaak^{-1}(U).
			\end{equation}
			Pick $m \le l$ with $m,l \in \mathbb{Z}$ small enough such that $2^l < r_1$ and such that $2^m>\dist(\cup_{\stackrel{k \ge l,}{ k \in \mathbb{Z}}} \worm_k, \mathbf{s})>0$. Using \eqref{e:projmet} and Lemma \ref{l:wormlaak}, we have
			\begin{align} \label{e:plaak2}
				\quolaak^{-1}\left( B_{\laak(\glue,\branch)}([(\mathbf{u},\mathbf{s})]_{R_{\laak}}, 2^m)\right) &\subset \bigcup_{ (\mathbf{v},\mathbf{s}) \in [(\mathbf{u},\mathbf{s})]_{R_{\laak}}} \left( B_{\ultra(\glue)}(\mathbf{v},2^l)\times B_{\tree(\branch)}(\mathbf{s},2^m) \right) \nonumber\\
				&\subset 	\bigcup_{ (\mathbf{v},\mathbf{s}) \in [(\mathbf{u},\mathbf{s})]_{R_{\laak}}} B_{\mathcal{P}(\glue,\branch)}((\mathbf{v},\mathbf{s}),2^l).
			\end{align}
			By \eqref{e:plaak1}, \eqref{e:plaak2} and $2^l<r_1$, we obtain that $B_{\laak(\glue,\branch)}([(\mathbf{u},\mathbf{s})]_{R_{\laak}}, 2^m) \subset U$. Therefore $[(\mathbf{u},\mathbf{s})]_{R_{\laak}}$ is an interior point of $U$. Since $[(\mathbf{u},\mathbf{s})]_{R_{\laak}}$ is an arbitrary point in $U$, we conclude that $U$ is an open subset of $\laak(\glue,\branch)$.
			\item It suffices to show that there exist $C>0$ such that for all pairs of points $[(\mathbf{u},\mathbf{s})]_{\laak}$ and $[(\mathbf{v},\mathbf{t})]_{\laak}$ in  $\laak(\glue,\branch)$, there exists a curve in $\gamma:[0,1] \to \laak(\branch,\glue)$ between $[(\mathbf{u},\mathbf{s})]_{\laak}$ and  $[(\mathbf{v},\mathbf{t})]_{\laak}$ such that $L(\gamma) \le C \mathsf{d}_{\mathcal{P}(\glue,\branch)}\left( (\mathbf{u},\mathbf{s}), (\mathbf{v},\mathbf{t}) \right)$. This is because, we can choose a finite sequence of points using \eqref{e:defmetlaak} to approximate $\metlaak$ within a factor of $(1+\delta)$ for any $\delta>0$. 
			Then each consecutive points can be connected using a curve as above to obtain the  quasiconvexity condition for  $\metlaak$ with constant $C_q=2C$.
			
			Let $[(\mathbf{u},\mathbf{s})]_{\laak},[(\mathbf{v},\mathbf{t})]_{\laak} \in \laak(\glue,\branch)$ be arbitrary distinct points. Let $n \in \mathbb{Z}$ be defined by the condition 
			\[
			2^n< \mathsf{d}_{\mathcal{P}(\glue,\branch)} \left((\mathbf{u},\mathbf{s}),(\mathbf{v},\mathbf{t})\right)\le 2^{n+1}.
			\] Since the tree $(\tree(\branch),\metrictree)$ is a geodesic space, we can connect $[(\mathbf{u},\mathbf{s})]_{\laak}$ and $[(\mathbf{u},\mathbf{t})]_{\laak}$ by a curve of length at most $\metrictree(\mathbf{s},\mathbf{t}) \le \mathsf{d}_{\mathcal{P}(\glue,\branch)} \left((\mathbf{u},\mathbf{s}),(\mathbf{v},\mathbf{t})\right)$. Therefore by concatenating the geodesic between $[(\mathbf{u},\mathbf{s})]_{\laak}$ and $[(\mathbf{u},\mathbf{t})]_{\laak}$ and then by using a curve with controlled length between $[(\mathbf{u},\mathbf{t})]_{\laak}$  and $[(\mathbf{v},\mathbf{t})]_{\laak}$, it suffices to consider the case $\mathbf{s}=\mathbf{t}$. 
			
			By Lemma \ref{l:worm}, there exist $\mathbf{t}_n \in \worm_n,\mathbf{t}_{n+1} \in \worm_{n+1}$  such that 
			\begin{equation} \label{e:qg1}
				\metrictree(\mathbf{s},\mathbf{t}_n) \le 2^n, \quad \metrictree(\mathbf{t},\mathbf{t}_{n+1}) \le 2^{n+1}, \quad 2^n \le \metrictree(\mathbf{t}_n,\mathbf{t}_{n+1}) \le 5 \cdot 2^{n}
			\end{equation}
			Let $\gamma_\tree:[0,1] \to \tree(\branch)$ be the curve obtained by concatenating the geodesic between $\mathbf{s}$ and $\mathbf{t}_n$, followed by the geodesic from $\mathbf{t}_n$ to $\mathbf{t}_{n+1}$ and finally the geodesic from $\mathbf{t}_{n+1}$ back to $\mathbf{t}$. By \eqref{e:qg1}, the curve $\gamma$ satisfies $2^n \le L(\gamma) \le  2^{n+3}$.
			By \eqref{e:defwormhole}, it is easy to see that the image of the geodesic from $\mathbf{t}_n$ to $\mathbf{t}_{n+1}$ intersects $\worm_k$ for all $k \in \llbracket -\infty, n+1 \rrbracket$. Therefore there exists $s_k \in [0,1]$ such that
			\[
			\gamma(s_k) \in \worm_k, \quad \mbox{ and } s_k=\inf \{s \in [0,1] : \gamma(s) \in \worm_k\} \quad \mbox{for all $k \le n+1$.}
			\]
			Since  $\metricultra(\mathbf{u},\mathbf{v}) \le \mathsf{d}_{\mathcal{P}(\glue,\branch)} \left((\mathbf{u},\mathbf{s}),(\mathbf{v},\mathbf{t})\right) \le 2^{n+1}$, we have $\mathbf{u}(k)=\mathbf{v}(k)$ for all $k \in \llbracket n+2, \infty \rrbracket$. Therefore, we obtain a curve $\gamma:[0,1] \to \laak(\glue,\branch)$ defined by $\gamma(t)= [(\mathbf{u}(t),\gamma_\tree(t))]_{\laak}$, where $\mathbf{u}(t)$ is given by
			\[
			\left(\mathbf{u}(t) \right)(k)= \begin{cases}
				\mathbf{u}(k)=\mathbf{v}(k) & \mbox{if $k \in \llbracket n+2,\infty \rrbracket, t \in [0,1]$,}\\
				\mathbf{u}(k) & \mbox{if $k \in \llbracket -\infty,n+1 \rrbracket, t \in [0,s_k]$,}\\
				\mathbf{v}(k) & \mbox{if $k \in \llbracket -\infty, n+2 \rrbracket, t \in (s_k,1]$.}
			\end{cases}
			\]
			It is evident that $\gamma$ is a continuous function and the length of the curve satisfies
			\[
			L(\gamma)= L(\gamma_\tree)\le 2^{n+3} \le 8  \mathsf{d}_{\mathcal{P}(\glue,\branch)} \left((\mathbf{u},\mathbf{s}),(\mathbf{v},\mathbf{t})\right).
			\]
			This completes the proof of quasiconvexity.
			More precisely, we have that 
			\begin{align} \label{e:8qc}
				\MoveEqLeft{\inf \{L(\gamma) : \mbox{$\gamma$ is a curve joining $[(\mathbf{u},\mathbf{s})]_{R_{\laak}}$ and $[(\mathbf{v}, \mathbf{t})]_{\laak}$}\}} \nonumber \\ 
				&\le 8 \metlaak\left([(\mathbf{u},\mathbf{s})]_{R_{\laak}},[(\mathbf{v}, \mathbf{t})]_{\laak}\right)
			\end{align}
			for all $[(\mathbf{u},\mathbf{s})]_{R_{\laak}},[(\mathbf{v}, \mathbf{t})]_{\laak} \in \laak(\glue,\branch)$. 
			
			By part (a) and \cite[Theorem 2.5.23]{BBI}, we conclude that there is a geodesic joining $\gamma$ joining any pair of points $[(\mathbf{u},\mathbf{s})]_{R_{\laak}},[(\mathbf{v}, \mathbf{t})]_{\laak} \in \laak(\glue,\branch)$ such that  
			$L(\gamma) \le 8 \metlaak\left([(\mathbf{u},\mathbf{s})]_{R_{\laak}},[(\mathbf{v}, \mathbf{t})]_{\laak}\right)$. 
			It is easy to see that by choosing points along such a curve, the space $(\laak(\glue,\branch),\metlaak)$ satisfies the chain condition.
		\end{enumerate}
	\end{proof}
	
	The inclusion \eqref{e:ballbox} admits a partial converse that we record for future use. The following lemma is a useful comparison between balls and `cylinders' in Laakso-type space.
	\begin{lemma} \label{l:ballbox}
		For all $[(\mathbf{v},\mathbf{t})]_{\laak} \in \laak(\glue,\branch)$ and $r>0$, we have 
		\begin{align} \label{e:bbox1}
			\lefteqn{\quolaak^{-1}\left(B_{\laak(\glue,\branch)} ([(\mathbf{v},\mathbf{t})]_{\laak},r)\right)} \nonumber \\
			&\subset \Biggl\{ (\mathbf{u},\mathbf{s}) \in \mathcal{P}(\glue,\branch)\Bigm|\begin{minipage}{260pt}
				$\metrictree(\mathbf{s},\mathbf{t})<r, \mathbf{u} \in \ultra(\glue)$ satisfies $\mathbf{u}(k)=\mathbf{v}(k)$ for all $k \in \mathbb{Z}$ such that $\worm_k \cap B_{\tree(\branch)}(\mathbf{t},r) =\emptyset$.
			\end{minipage} \Biggr\},
		\end{align}
		and 
		\begin{align} \label{e:bbox2}
			\lefteqn{\quolaak^{-1}\left(B_{\laak(\glue,\branch)} ([(\mathbf{v},\mathbf{t})]_{\laak},r)\right)} \nonumber \\
			&\supset \Biggl\{ (\mathbf{u},\mathbf{s}) \in \mathcal{P}(\glue,\branch)\Bigm|\begin{minipage}{260pt}
				$\metrictree(\mathbf{s},\mathbf{t})<r/4, \mathbf{u} \in \ultra(\glue)$ satisfies $\mathbf{u}(k)=\mathbf{v}(k)$ for all $k \in \mathbb{Z}$ such that $\worm_k \cap B_{\tree(\branch)}(\mathbf{t},r/4) =\emptyset$.
			\end{minipage} \Biggr\}.
		\end{align}
	\end{lemma}
	\begin{proof}
		Since the estimate \eqref{e:bbox1} was shown in \eqref{e:ballbox}, it remains to show \eqref{e:bbox2}. 
		
		To this end, let us assume that $(\mathbf{u},\mathbf{s}),   (\mathbf{v},\mathbf{t}) \in \mathcal{P}(\glue,\branch)$ and $r>0$ satisfies $\metrictree(\mathbf{s},\mathbf{t})<r/4$, and   $\mathbf{u}(k)=\mathbf{v}(k)$ for all $k \in \mathbb{Z}$ such that $\worm_k \cap B_{\tree(\branch)}(\mathbf{t},r/4) =\emptyset$. Let $n \in \mathbb{Z}$ be such that $2^{n-1} \le r/4 < 2^n$. Then by Lemma \ref{l:worm}(b), we have 
		\begin{equation} \label{e:bbx1}
			\# \left(\bigcup_{\substack{k \ge n+1,\\ k \in \mathbb{Z}}} \worm_k \cap B_{\tree(\branch)}(\mathbf{t},r/4) \right) \in \{0,1\}.
		\end{equation}
		If the above cardinality is zero, then 
		\begin{equation}
			\label{e:bbx2}
			\metlaak\left( [(\mathbf{u},\mathbf{s})]_{\laak}, [(\mathbf{v},\mathbf{t})]_{\laak}\right) \le 2^n \le \frac{r}{2}.
		\end{equation}
		If the cardinality in \eqref{e:bbx1} is one, then let $l \in \llbracket n+1, \infty \rrbracket, \wt{\mathbf{t}} \in \tree(\branch)$ be such that $\worm_l \cap B_{\tree(\branch)}(\mathbf{t},r/4)  = \{ \wt{\mathbf{t}}\}$. If $\mathbf{u}(l)=\mathbf{v}(l)$, then the estimate \eqref{e:bbx2} holds. If $\mathbf{u}(l)\neq \mathbf{v}(l)$, let us define $\wt{\mathbf{u}} \in \ultra(\glue)$ by $\wt{\mathbf{u}}(l)=\mathbf{v}(l), \restr{\mathbf{u}}{\mathbb{Z} \setminus \{l\}}=\restr{\wt{\mathbf{u}}}{\mathbb{Z} \setminus \{l\}}$. In this case, we have 
		\begin{align} \label{e:bbx3}
			\metlaak\left( [(\mathbf{u},\mathbf{s})]_{\laak}, [(\mathbf{v},\mathbf{t})]_{\laak}\right)  &\le \metprod\left( (\mathbf{u},\mathbf{s}),(\mathbf{u},\wt{\mathbf{t}}) \right)+ \metprod\left( (\wt{\mathbf{u}},\wt{\mathbf{t}}),(\mathbf{v},\mathbf{t}) \right) \nonumber \\
			& < \frac{2r}{4} + 2^n < r.
		\end{align}
		Combining \eqref{e:bbx2} and \eqref{e:bbx3}, we obtain \eqref{e:bbox2}.
	\end{proof}
	Next, we define a measure on the Laakso-type space.
	Let us denote the pushforward of $\measprod$ on $\laak(\glue,\branch)$ under the quotient map $\quolaak$ as 
	\begin{equation} \label{e:measlaak}
		\measlaak:= \quolaak_*(\measprod).
	\end{equation}
	
	Similar to Corollary \ref{c:meastree}, the following two-sided volume estimate on $\measlaak$ is a consequence of Proposition \ref{p:laak}(a) and Lemma \ref{l:volume}. The proof of the following result follows from the same argument as Corollary \ref{c:meastree} (recall \eqref{e:defvol}).
	\begin{cor} \label{c:measlaak}
		There exists $C \in [1,\infty)$ such that for all $[(\mathbf{u},\mathbf{t})]_\laak \in \laak(\glue,\branch)$ and $r>0$, we have 
		\begin{equation*}
			C^{-1} 	V_\glue(r)	V_\branch(r) \le	\measlaak\left(B_{{\laak(\glue,\branch)}}([(\mathbf{u},\mathbf{t})]_\laak,r)\right) \le C 	V_\glue(r)	V_\branch(r).
		\end{equation*}
		In particular, $\mathsf{m}_{\ultra(\branch)}$ is a doubling measure on $(\ultra(\branch),\mathsf{d}_{\ultra(\branch)})$ and $(\ultra(\branch),\mathsf{d}_{\ultra(\branch)})$  is a proper, separable, metric space.
	\end{cor}
	
	The geometry of geodesics in $\laak(\glue,\branch)$ plays an important role in the proof of Poincar\'e inequality for the diffusion we construct on the Laakso-type space. The next two lemmas provide some quantitative control on geodesics.
	The following lemma shows an upper bound on the length of geodesics that do not intersect $(\projtree)^{-1}(\worm_k)$ in $\laak(\glue,\branch)$.
	\begin{lemma} \label{l:wormdense}
		There exists $C >0$ such that for any $k \in \mathbb{Z}$, and any geodesic $\gamma$ in $(\laak(\glue,\branch),\metlaak)$ that satisfies $\projtree(\gamma) \cap \worm_k = \emptyset$ has length $l(\gamma) <  2^{k+6}$.
	\end{lemma}
	\begin{proof}

		Let $\gamma:[0,L(\gamma)] \to \laak(\glue,\branch)$ be a geodesic parameterized by arc length.
		
		We claim that 
		\begin{equation} \label{e:wd1}
			\diam(\projtree(\gamma([0,L(\gamma)] )), \metrictree) <2^{k+1}.
		\end{equation}
		Assume to the contrary that \eqref{e:wd1} fails. Then there exists $\mathbf{t}_1,\mathbf{t}_2 \in \tree(\branch)$ such that $\metrictree(\mathbf{t}_1,\mathbf{t}_2) \ge 2^{k+1}$.  Let $\wt{\gamma}: [0,\metrictree(\mathbf{t}_1,\mathbf{t}_2)] \to \tree(\branch)$ denote the unit speed geodesic in $\tree(\branch)$ from $\mathbf{t}_1$ to $\mathbf{t}_2$.
		Since $\tree(\branch)$ is a metric tree, we have $\wt{\gamma}([0,\metrictree(\mathbf{t}_1,\mathbf{t}_2)]) \subset \projtree(\gamma([0,L(\gamma)] ))$.
		By the $0$-hyperbolicity, there exists $s_0 \in [0,\metrictree(\mathbf{t}_1,\mathbf{t}_2)]$ such that $t \mapsto \metrictree(p_\tree, \wt{\gamma}(t))$ is strictly monotone affine function with slope $\pm 1$ in both $[0,s_0]$ and $[s_0,\metrictree(\mathbf{t}_1,\mathbf{t}_2)]$. Therefore, by \eqref{e:defwormhole} we have $\wt{\gamma}([0,\metrictree(\mathbf{t}_1,\mathbf{t}_2)])  \cap \worm_k \neq \emptyset$, a contradiction as it implies $\projtree(\gamma) \cap \worm_k \neq \emptyset$. This proves \eqref{e:wd1}.
		
		By \eqref{e:bbox2} in Lemma \ref{l:ballbox}, we have $\gamma([0,L(\gamma)]) \subset B_{\laak(\glue,\branch)} (\gamma(0),2^{k+3})$ and hence $\metlaak(\gamma(0),\gamma(L(\gamma)))< 2^{k+3}$. Hence by \eqref{e:8qc}, we have $L(\gamma) < 2^{k+6}$
		
	\end{proof}
	
	We would like to quantify that there is a sufficiently rich family of geodesics connecting any pair of points in $\laak(\glue,\branch)$ that have controlled overlap with one another.   To this end, we need the notion of a \emph{length measure} $\lenlaak$ on the Laakso-type space and on the $\mathbf{R}$-tree $\tree(\branch)$.  For $t_1,t_2 \in \tree(\branch)$, let $(t_1,t_2) \subset  \tree(\branch)$ denote the image of the geodesic between $t_1$ and $t_2$ with the endpoints removed. Similarly, we denote the image of the geodesic connecting $t_1$ and $t_2$ as $[t_1,t_2]$. As explained in \cite[p. 3117]{AEW}, there exists a unique $\sigma$-finite Borel measure $\lentree$ on $\left(\tree(\branch),\metrictree\right)$ such that 
	\begin{equation} \label{e:deflenmeas}
		\lentree((t_1,t_2))= \metrictree(t_1,t_2), \quad \mbox{for all pairs of distinct points $t_1,t_2\in\tree(\branch)$.}
	\end{equation}
	In other words, the length measure behaves like the $1$-dimensional Lebesgue measure on (image of) geodesics.
	The    \emph{length measure} $\lenlaak$ on the Laakso-type space is defined by
	\begin{equation} \label{e:deflenlaak}
		\lenlaak := \quolaak_* \left( \measultra \times \lentree \right).
	\end{equation}
	\begin{definition}
		Let $(X,\mathsf{d})$ be a metric space, $x,y \in X, D>0$ be such that $D$ is the length of a geodesic (and hence every geodesic) between $x$ and $y$.
		A \emph{probability measure on the space of geodesics between $x,y \in X$ parametrized by arc length} is a probability measure $\mathbb{P}^{x,y}$ on $\contfunc([0,D],X)$ such that $\gamma \in \contfunc([0,D],X)$ is a geodesic  between $x,y \in X$ parametrized by arc length for  $\mathbb{P}^{x,y}$-almost every $\gamma$.
	\end{definition}
	We construct a probability measure on the space of geodesics  that     are `spread out' in the lemma below. Similar probability measures  are called \textbf{pencil of curves} in \cite[Definition 2.3]{Laa}. A nice pedagogical introduction to this concept and its role in proving Poincar\'e inequality can be found in \cite[p. 30]{Hei}. The following lemma plays a key role in  obtaining the Poincar\'e inequality and thereby sub-Gaussian heat kernel estimates.
	
	\begin{lemma} \label{l:pencil}
		There exists   $C_1  \in (1,\infty)$ such that for any pair of points $x,y \in \laak(\glue,\branch)$ there is a probability measure $\mathbb{P}^{x,y}$ on the space of geodesics between $x$ and $y$ parameterized by arc length such that  for any measurable function $h: \laak(\glue,\branch) \to [0,\infty)$ we have
		\begin{align} \label{e:pencil}
			\lefteqn{\int	\int_0^{L(\gamma)} h(\gamma(t)) \, dt \,d \mathbb{P}^{x,y}(\gamma)} \nonumber \\
			&\le C_1 \int_{B_{\laak(\glue,\branch)}(x,8 \metlaak(x,y))} \frac{h(z)}{V_{\glue}(\metlaak(x,z) \wedge \metlaak(y,z))} \,d \lambda_{\laak(\glue,\branch)}(z).
		\end{align}
	\end{lemma}
	
	\begin{proof}
		For each $k \in \mathbb{Z}$, we pick independent random variables $U_k$ with law $\mathsf{m}_{\glue(k)}$, where $\mathsf{m}_{\glue(k)}$ denotes the uniform probability measure on the finite set $\llbracket 0, \glue(k)-1\rrbracket$ as defined in \textsection \ref{ss:ultra}. Let us pick a geodesic $\mu:[0,D] \to \laak(\glue,\branch)$ from $x$ to $y$. Note that $D \le 8\metlaak(x,y)$ by \eqref{e:8qc} and $D \ge\metlaak(x,y)$ by triangle inequality. Denote by $\mu^{\tree}= \projtree \circ \mu: [0,D] \to \tree(\branch)$. Then by Lemma \ref{l:wormdense}, we have 
		\begin{equation*}
			\mu^\tree([0,D/2)) \cap \worm_k \neq \emptyset, \quad \mbox{and } 	\mu^\tree((D/2,D]) \cap \worm_k \neq \emptyset,
		\end{equation*} 
		for all $k \in \mathbb{Z}$ such that $2^{k+8} \le \metlaak(x,y)$. For any such 
		$k \in \mathbb{Z}$ such that  $2^{k+8} \le \metlaak(x,y)$, we have by Lemma \ref{l:wormdense}
		\begin{equation*}
			f_k:= \inf \{s \in [0,D] : \mu^\tree (s) \in \worm_k \} < D/2, \quad l_k:= \sup \{s \in [0,D] : \mu^\tree (s) \in \worm_k \} > D/2,
		\end{equation*}
		and 
		\begin{equation} \label{e:penc1}
			\max(f_k,D-l_k ) \le 2^{k+6}.
		\end{equation}
		Let $\mu^\ultra:[0,D] \to \ultra(\glue)$ be such that $\mu(s)= [(\mu^\ultra(s),\mu^\tree(s))]_{\laak}$ for all $s \in [0,D]$. By modifying $\mu^\ultra$ if necessary (by making at most one just through the wormholes $\worm_k$ for each $k \in \mathbb{Z}$), we may and will assume that the set of discontinuities of $s \mapsto \mu^\ultra(s)$ is a subset of 
		\[
		\Biggl\{ \inf \{s \in [0,D]: \mu^\tree(s) \in \worm_k \} \Biggm| k \in \mathbb{Z} \mbox{ is such that } \mu^\tree([0,D]) \cap \worm_k \neq \emptyset \Biggr\}.
		\]
		In particular, the set of discontinuities of $s \mapsto \mu^\ultra(s)$ is at most countable.

		We define $\gamma^\ultra:[0,D] \to \ultra(\glue)$ and $\gamma:[0,D] \to \laak(\glue,\branch)$ as 
		\begin{equation} \label{e:penc2}
			(\gamma^{\ultra}(s))(k)= \begin{cases}
				\mu^{\ultra}(s)(k) & \mbox{if $k \in \mathbb{Z}$ such that $2^k> \metlaak(x,y)$,}\\
				\mu^{\ultra}(s)(k) & \mbox{if $k \in \mathbb{Z}$ such that $2^k\le \metlaak(x,y)$, $s \notin [f_k,l_k]$},\\
				U_k & \mbox{if $k \in \mathbb{Z}$ such that $2^k\le \metlaak(x,y)$, $s \in [f_k,l_k]$},
			\end{cases}
		\end{equation}
		and 
		\begin{equation} \label{e:penc3}
			\gamma(s)= [(\gamma^\ultra(s),\mu^\tree(s))]_{\laak}, \quad \mbox{for all $s \in [0,D]$.}
		\end{equation}
		It is easy to check that $\gamma$ is  a geodesic from $x$ to $y$ parameterized by arc length for all possible values of $(U_k)_{k \in \mathbb{Z}}$ with $U_k \in \llbracket 0 , \glue(k)-1 \rrbracket$ for all $k \in \mathbb{Z}$ and that $\gamma^\ultra$ can have at most countably many discontinuities and has countably many possible values for any fixed realization of $(U_k)_{k \in \mathbb{Z}}$. Hence this random curve $\gamma:[0,D] \to \laak(\glue,\branch)$ defines a probability measure on the space of geodesics between $x$ and $y$ parameterized by arc length. 
		For all $s \in [0,D]$ define $R(s) = \{ k \in \mathbb{Z} : 2^k\le \metlaak(x,y), s \in [f_k,l_k]\}$ as the set of integers $k$ for which $(\gamma^\ultra(s))(k)$ is \emph{random}. By \eqref{e:penc1}, Lemma \ref{l:worm}(b), there exists $K_0 \in \mathbb{N}$   depending only on $\sup_{k \in \mathbb{Z}} \branch(k), \sup_{k \in \mathbb{Z}} \glue(k)$ such that
		\begin{equation} \label{e:penc4}
			\mathbb{Z} \cap (-\infty, \log_2(s \wedge (D-s))-6]	\subset R(s), \quad   \# \left(\left(\mathbb{Z}  \setminus (-\infty, \log_2(s \wedge (D-s))-6] \right) \cap R(s) \right) \le K_0.
		\end{equation}
		The law of $\gamma^\ultra(s)$ is characterized by the following property:$\{(\gamma^\ultra(s))(k): k \in \mathbb{Z}\}$ are independent with the  distribution of $(\gamma^\ultra(s))(k)$ being  $\mathsf{m}_{\glue(k)}$ if $k \in R(s)$ and the  distribution of $(\gamma^\ultra(s))(k)$ given by Dirac mass at $\mu^\ultra(s)$ if $k \in \mathbb{Z} \setminus R(s)$. Let $\mathsf{m}_{\gamma^\ultra(s)}$  denote the law of  $\gamma^\ultra(s)$. By \eqref{e:penc4}, there exists $C_2>1$ depending only on $\sup_{k \in \mathbb{Z}} \branch(k), \sup_{k \in \mathbb{Z}} \glue(k)$ such that $\mathsf{m}_{\gamma^\ultra(s)} \ll \measultra$ and
		\begin{equation} \label{e:penc5}
			\frac{d\mathsf{m}_{\gamma^\ultra(s)}}{d\measultra} \le \frac{C_2}{V_{\glue}(s \wedge (D-s))}.
		\end{equation} 
		Since $\gamma$ is a geodesic, we have using Proposition \ref{p:laak}(c) and \eqref{e:8qc} that
		\begin{equation} \label{e:penc7}
			s/8 \le	\metlaak(x,\gamma(s)) \le s, \quad 	(D-s)/8 \le	\metlaak(y,\gamma(s)) \le D-s.
		\end{equation}
		On each interval $I \subset [0,D]$ where $\gamma^\ultra$ is constant, say $\mathbf{u} \in \ultra(\glue)$ the integral $\int_I h(\gamma(t)) \,dt$ is $\int h\circ I_\mathbf{u}(z) \one_{\mu^\tree(I)}(z) \,d\lentree(z)$.
		This along with \eqref{e:penc5}, \eqref{e:penc7} and Fubini's theorem implies \eqref{e:pencil}.
	\end{proof}

	\subsection{Uniform domains in   Laakso-type space} \label{ss:uniform}
	The results in \textsection \ref{ss:uniform} will be used to construct a growing family of graphs with prescribed sub-Gaussian heat kernel estimates. A reader who is only interested in the proof of Theorem \ref{t:mainsuf} can skip this part.
	
	We recall the notion of uniform domain in a metric space. 
	\begin{definition} \label{d:uniform}
		Let $(X,d)$ be a metric space and $c_U \in (0,1), C_U \in [1,\infty)$.
		A connected, non-empty, proper open set $U \subsetneq X$ is said to be a \textbf{$(c_U,C_U)$-uniform domain} if   for every pair of points $x,y \in U$, there exists a curve $\gamma$ in $U$ from $x$ to $y$ such that its diameter $\diam(\gamma) \le C_U d(x,y)$ and for all $z \in \gamma$, 
		\[
		\dist(z,U^c) \ge c_U \min \left(d(x,z), d(y,z) \right).
		\]
		Such a curve $\gamma$ is called a \emph{$(c_U,C_U)$-uniform curve}.
	\end{definition}
	The importance of uniform domains for us is that sub-Gaussian heat kernel estimates for diffusions on an ambient space are inherited by corresponding reflected diffusions on uniform domains \cite[Theorem 2.8]{Mur24}. This allows us to construct compact spaces with diffusions satisfying sub-Gaussian heat kernel bounds starting from diffusions on unbounded spaces.

	We show that certain balls in $\tree(\branch)$ and $\laak(\glue,\branch)$ are uniform domains (recall \eqref{e:bptree}).

	\begin{lemma} \label{l:uniftree}
		For any $k \in \mathbb{Z}$, the ball 
		$B_{\tree(\branch)}(p_\tree,2^k)$ is a $(1,1/5)$-uniform domain.
	\end{lemma}
	\begin{proof}
		Let us denote the zero function in $\ultra(\branch)$ by $p_{\ultra(\branch)}$. For $k \in \mathbb{Z}$, let $p_{\ultra(\branch)}^k \in \ultra(\branch)$ denote the function that is defined by
		\begin{equation} \label{e:unift1}
			p_{\ultra(\branch)}^k(l)= \begin{cases}
				1 & \mbox{if $l=k$,}\\
				0 & \mbox{if $l \neq k$.}
			\end{cases}
		\end{equation}
		Define $p_\tree^k \in \tree(\branch), k \in \mathbb{Z}$ as 
		\begin{equation} \label{e:unift2}
			p_\tree^k = \codetree(p_{\ultra(\branch)}^k).
		\end{equation}
		By Lemmas \ref{l:incntree} and \ref{l:inftree}, for all $k \in \mathbb{Z}$ we have 
		\begin{equation} \label{e:unift3}
			\partial B_{\tree(\branch)}(p_\tree,2^k)= \{p_\tree^k\},  \quad \dist(x, B_{\tree(\branch)}(p_\tree,2^k)^c)=d_{\tree(\branch)}(x,p_\tree^k),  \mbox{ for all $x \in B_{\tree(\branch)}(p_\tree,2^k)$.}
		\end{equation}

		Geodesics from points in $B_{\tree(\branch)}(p_\tree,2^k)$ to $p_\tree^k$ have to pass through some specific points. In order to describe them, we define for every $m<k, k, m \in \mathbb{Z}$ points $q_{\ultra(\branch)}^{m,k} \in \ultra(\branch)$ and $q_{\tree(\branch)}^{m,k} \in \tree(\branch)$ defined as
		\begin{equation} \label{e:defqtree}
			q_{\ultra(\branch)}^{m,k}(l) = \begin{cases}
				1 & \mbox{if $l \in \{m,k\}$,}\\
				0 & \mbox{if $l \notin \{m,k\}$,} 
			\end{cases}
			\quad q_{\tree(\branch)}^{m,k}= \codetree(q_{\ultra(\branch)}^{m,k}).
		\end{equation}
		Then by Lemmas \ref{l:incntree} and \ref{l:inftree}, $k,m \in \mathbb{Z}$ with $m<k$,   for all $x \in B_{\tree(\branch)}(p_\tree,2^k) \setminus B_{\tree(\branch)}(p_\tree^k,2^m), y \in B_{\tree(\branch)}(p_\tree^k,2^m) \cap   B_{\tree(\branch)}(p_\tree,2^k)$ the geodesic $\gamma$ from $x$ to $y$ satisfies
		\begin{equation} \label{e:unift4}
			q_{\tree(\branch)}^{m,k} \in \gamma.
		\end{equation}
		Since $\codetree$ is a $1$-Lipschitz map and $\overline{ B_{\tree(\branch)}(p_\tree,2^m)}= \codetree(\overline{B}_{\ultra(\branch)}(p_{\ultra(\branch)},2^m))$, we have by Lemma \ref{l:incntree}, we have 
		\begin{equation} \label{e:unift5}
			\diam(B_{\tree(\branch)}(p_\tree,2^m)) = 2^m, \quad \mbox{for all $m\in \mathbb{Z}$.}
		\end{equation}
		
		We claim that the geodesic curve between any pair of points in   $B_{\tree(\branch)}(p_\tree,2^k), k \in \mathbb{Z}$ is a $(1,1/5)$ uniform curve.
		Assume to the contrary that there exists $x,y \in B_{\tree(\branch)}(p_\tree,2^k)$ connected by a geodesic,  $z \in \gamma$ such that (by using \eqref{e:unift3})
		\begin{equation} \label{e:unift6}
			\metrictree(x,z) \wedge \metrictree(y,z) > 5 \metrictree(z,p_\tree^k).
		\end{equation}
		Then by triangle inequality and \eqref{e:unift6}, we have 
		\begin{equation*}
			\metrictree(x,p_\tree^k) \wedge \metrictree(y,p_\tree^k) > 4 \metrictree(z,p_\tree^k).
		\end{equation*}
		This along with \eqref{e:unift5} implies there exists $m \in \mathbb{Z}$ with $m<k$ such that 
		\begin{equation} \label{e:unift7}
			\metrictree(x,p_\tree^k) \wedge \metrictree(y,p_\tree^k) > 2^{m}>2^{m-1}> \metrictree(z,p_\tree^k).
		\end{equation}
		Therefore by \eqref{e:unift7}, \eqref{e:unift4}, the geodesics from $z$ to $x$ and  and from $z$ to $y$ contain $q_{\tree(\branch)}^{m,k}$ and $q_{\tree(\branch)}^{m-1,k}$. This is a contradiction since this would imply that the geodesic from $x$ to $y$ backtracks.
	\end{proof}
	
	We define a similar base point on $\laak(\glue,\branch)$. Let $p_\ultra \in \ultra(\glue)$ denote the identically zero function. Then we set the base point $p_\laak \in \laak(\glue,\branch)$ as 
	\begin{equation} \label{e:bplaak}
		p_\laak:=\quolaak((p_\ultra,p_\tree)),
	\end{equation}
	where $\quolaak$ is the quotient map from \eqref{e:defquolaak} and $p_\ultra \in \ultra(\glue)$ is the function that is identically zero.
	
	\begin{prop} \label{p:uniformlaak}
		There exist $C_U \in (1,\infty), c_U \in (0,1)$ such that any $k \in \mathbb{Z}$, the ball $B_{\laak(\glue.\branch)}(p_{\laak},2^k)$ is a $(C_U,c_U)$-uniform domain.
	\end{prop}
	\begin{proof}
		We start with the observation that 
		\begin{equation} \label{e:unlaak1}
			B_{\laak(\glue.\branch)}(p_{\laak},2^k)= \quolaak\left( B_{\ultra(\glue)}(p_{\ultra(\glue)},2^k) \times B_{\tree(\branch)}(p_\tree,2^k) \right) \quad \mbox{for all $k \in \mathbb{Z}$,}
		\end{equation}
		where $p_{\ultra(\glue)} \in \ultra(\glue)$ denotes the function that is identically zero. 
		We note that \eqref{e:unlaak1} follows from Lemma \ref{l:ballbox} along with the observation that $\worm_m \cap B_{\tree(\branch)}(p_\tree,2^k) \neq \emptyset$ if and only if $m \in \llbracket -\infty, k-1\rrbracket$.
		As a result (similar to \eqref{e:unift3}), we have 
		\begin{equation} \label{e:unlaak2}
			\partial B_{\laak(\glue,\branch)}(p_\laak,2^k)= \quolaak(B_{\ultra(\glue)}(p_{\ultra(\glue)},2^k) \times\{ p_\tree^k\}),
		\end{equation}
		and 
		\begin{equation}\label{e:unlaak3}
			\dist(x, B_{\laak(\glue,\branch)}(p_\laak,2^k)^c)=d_{\tree(\branch)}(\projtree(x),p_\tree^k),  \mbox{ for all $x \in B_{\laak(\glue,\branch)}(p_\laak,2^k)$.}
		\end{equation}
		
		Let $x= \quolaak((\mathbf{u},\mathbf{s})), y = \quolaak((\mathbf{v},\mathbf{t})) \in B_{\laak(\glue,\branch)}(p_\laak,2^k)$.  The proof now splits into two cases.
		
		\noindent \emph{Case 1}: $2 \mathsf{d}_{\ultra(\glue)}(\mathbf{u},\mathbf{v}) \le \metrictree(\mathbf{s},\mathbf{t})$. \\
		Then by \eqref{e:wd1}, the geodesic from $\mathbf{s}$ to $\mathbf{t}$ in $\tree(\branch)$ intersects all $\worm_m$ with $m \in \mathbb{Z}$ such that  $2^m \le \metrictree(\mathbf{u},\mathbf{v})$. This implies that there is a geodesic in $\laak(\glue,\branch)$ between $x$ and $y$ such that its projection under $\projtree$ is the geodesic between $\mathbf{s}$ and $\mathbf{t}$. Therefore by \eqref{e:unlaak3} and Lemma \ref{l:uniftree}, any geodesic from $x$ to $y$ is an $(1,1/5)$-uniform curve in $\laak(\glue,\branch)$. \\

		\noindent \emph{Case 2}: $2 \mathsf{d}_{\ultra(\glue)}(\mathbf{u},\mathbf{v}) > \metrictree(\mathbf{s},\mathbf{t})$. \\
		Let  $m \in \llbracket -\infty, k-1\rrbracket$ be such that $\mathsf{d}_{\ultra(\glue)}(\mathbf{u},\mathbf{v}) =2^m$. 
		%  If $m=k-1$, we consider the concatenation of the geodesic from $\mathbf{s}$ to $p_\tree$, followed by the geodesic from $p_\tree$ to $q_{\tree}^{k-1,k}$ and finally the geodesic from $q_{\tree}^{k-1,k}$ to $\mathbf{t}$, where  $q_{\tree}^{k-1,k}$ is as defined in \eqref{e:defqtree}. This curve intersects $\worm_m$ for all $m \in \llbracket -\infty, k-1\rrbracket$  and hence lifts to a $(4,1/5)$-uniform curve in $\laak(\glue,\branch)$ from $x$ to $y$ by Lemma \ref{l:uniftree} and triangle inequality.
		We consider a path on the tree as follows. Let $t_m \in \worm_{m} \cap B_{\tree(\branch)}(p_\tree,2^k), t_{m-1} \in \worm_{m-1} \cap B_{\tree(\branch)}(p_\tree,2^k)$ such that 
		\[
		\metrictree(\mathbf{s},t_m)= \dist( \worm_{m} \cap B_{\tree(\branch)}(p_\tree,2^k) \le 2^m, \quad \mathbf{s}), \quad  \metrictree(t_{m-1},t_m)= 2^{m-1}.
		\]
		We consider the concatenation of the geodesic from $\mathbf{s}$ to $p_\tree$, followed by the geodesic from $p_\tree$ to $q_{\tree}^{k-1,k}$ and finally the geodesic from $q_{\tree}^{k-1,k}$ to $\mathbf{t}$. This curve intersects $\worm_m$ for all $m \in \llbracket -\infty, k\rrbracket$ 
		and hence lifts to a curve in $\laak(\glue,\branch)$ from $x$ to $y$. 
		The desired conclusion for this case follows from Lemma \ref{l:uniftree}, triangle inequality and the fact that $\dist(\worm_m,p_\tree^k) =2^m$ for all $m \in \llbracket -\infty, k-1\rrbracket$.
	\end{proof}

	\section{Construction of Dirichlet form} \label{s:df}
	
	The goal of this section is to construct a strongly local, regular, Dirichlet form on the metric measure space $(\laak(\glue,\branch), \metlaak,\measlaak)$ or equivalently an $\mathsf{m}_\laak(\glue,\branch)$-symmetric diffusion process on $\laak(\glue,\branch)$. 
	To this end, we first recall the construction of Dirichlet form on  the $\mathbb{R}$-tree $\tree(\branch)$ in \textsection \ref{ss:difftree}. We then use the Dirichlet form on the tree to construct a Dirichlet form on the Laakso-type space in \textsection \ref{ss:laakdiff}. Following the approach of Barlow and Evans \cite{BE}, this step is carried by viewing the Laakso-type space as a projective limit of simpler spaces where the diffusion is easier to define.
	
	Similar to the previous section, we assume that the branching and gluing functions $\mathbf{b},\mathbf{g}:\mathbb{Z}\to \mathbb{Z}$ satisfy \eqref{e:defbranch} and \eqref{e:defglue} throughout the section. 
	\subsection{Diffusion on tree} \label{ss:difftree}
	In this subsection, we define a $\meastree$-symmetric diffusion process on $\tree(\branch)$.
	If the function $\glue$ is identically one, then the Laakso-type space $\laak(\glue,\branch)$ is same as the tree $\tree(\branch)$ and hence the diffusion on the tree can be viewed as a special case of the diffusion on the Laakso-type space that we will construct in \textsection \ref{ss:laakdiff}. 
	
	There are   equivalent approaches to define a Dirichlet form on $\tree(\branch)$ due to \cite{Kig95,AEW} (cf. \cite[Remark 3.1]{AEW}). The approach in \cite{Kig95} follows a limit of rescaled discrete energies while \cite{AEW} presents a direct description of this limit. 
	We follow the presentation of \cite{AEW}. 
	
	%	In order to describe the Dirichlet form, we first need to introduce the notion of \textbf{length measure} on the separable $\mathbb{R}$-tree $\left(\tree(\branch),\metrictree\right)$.

	Let $\contfunc(\tree(\branch))$ denote the space of continuous functions  on $\left(\tree(\branch),\metrictree\right)$. Let $\contfunc_0(\tree(\branch))$ denote the space of compactly supported functions in $\contfunc(\tree(\branch))$. Let $\contfunc_\infty(\tree(\branch))$ denote the space of continuous functions vanishing at infinity; that is
	\[
	\contfunc_\infty(\tree(\branch))=  \Biggl\{  f \in \contfunc((\tree(\branch))\Biggm|
	\begin{minipage}{250pt}
		for all $\varepsilon>0$, there exists a compact set $K\subset \tree(\branch)$ such that $\abs{f(x)}< \varepsilon$ for all $x \in \tree(\branch) \setminus K$
	\end{minipage}
	\Biggr\}.
	\]

	We recall the notion of locally absolutely continuous functions. 
	\begin{definition}
		We say that a function $f \in \contfunc(\tree(\branch))$ is \textbf{locally absolutely continuous} if for
		all $\varepsilon>0$ and for all  Borel subsets $S \subset \tree(\branch)$ with $\lentree(S)<\infty$, there exists a $\delta=\delta(\varepsilon,S)>0$ 
		such that if there are disjoint arcs $[x_i,y_i], i=1,\ldots,n$ with $\sum_{i=1}^n \metrictree(x_i,y_i)<\delta$,  then $\sum_{i=1}^n \abs{f(x_i)-f(y_i)}<\varepsilon$. We set $\mathcal{A}(\tree(\branch))$ be the set of locally absolutely continuous functions on $(\tree(\branch),\metrictree)$.
	\end{definition}
	
	In order to define a gradient of a locally absolutely continuous function, we need an orientation of $\tree(\branch)$. For that purpose, we use the base point $p_\tree$ to define a partial order with respect to  $p_\tree$  denoted by $\le_{p}$. This is defined by setting $x \le_p y$ for $x,y \in \tree(\branch)$ if $x \in [p_{\tree},y]$. 
	For $x,y \in \tree(\branch)$ we denote by $x \wedge y$, the (unique) maximal element $z$ such that $z \le_p x$ and $z \le_p y$. Equivalently, $x \wedge y$ is the unique element in the intersection $[x,y] \cap [p_\tree,x] \cap [p_\tree,y]$.
	The root provides an orientation-sensitive integration given by
	\[
	\int_{x}^y g(z)\,\lentree(dz) := -\int_{[x \wedge y,x]}  g(z)\,\lentree(dz)+\int_{[x \wedge y,y]} g(z)\,\lentree(dz) \quad \mbox{for all $x,y \in T$.}
	\]
	By the fundamental theorem of calculus (see \cite[Proposition 1.1]{AEW}), for any locally absolutely continuous function $f \in \mathcal{A}(\tree(\branch))$, there exists a function $g:\tree(\branch) \to \mathbb{R}$ unique up to sets of $\lentree$ measure zero such that $g \in L^1([x,y], \restr{\lentree}{[x,y]})$ for all $x,y \in \tree(\branch)$ and
	\begin{equation} \label{e:defgrad}
		f(y)-f(x)=\int_x^y g(z)\,\lentree(dz).
	\end{equation}
	We say that the function $g$ satisfying the above property as the gradient of $f$, denoted by $\grad f$. We define 
	\begin{equation}\label{e:resform}
		\wt{\mathcal{F}}_{\tree}= \Biggl\{f \in \mathcal{A}(\tree(\branch)): \int_{\tree(\branch)} \abs{\grad f}^2\,d\lentree < \infty \Biggr\}, \quad  \mathcal{E}^\tree(f,g) = \int_{\tree(\branch)} \grad f \grad g \, d\lentree
	\end{equation}
	for all $f,g \in 	\wt{\mathcal{F}}_{\tree}$.
	It is easy to see that $(\mathcal{E}^\tree,\wt{\mathcal{F}}_\tree)$ is a resistance form in the sense of Kigami \cite[Definition 3.1]{Kig12} and the resistance metric coincides with the tree metric $\metrictree$. In particular, we have the following.
	\begin{lemma} \label{l:resform}
		\begin{enumerate}[(1)]
			\item $\wt{\mathcal{F}}_{\tree}$ is a linear subspace of $\contfunc(\tree(\branch))$ and $\mathcal{E}^\tree$ is a non-negative symmetric quadratic form on $\wt{\mathcal{F}}_{\tree}$. $\mathcal{E}^\tree(f,f)=0$ if and only if $f$ is a constant.
			\item Let $\sim$ be an equivalence relation on $\wt{\mathcal{F}}_\tree$ defined by $u \sim v$ if and only if $u-v$ is constant on $\tree(\branch)$. Then the quotient space $(\mathcal{F}^\tree/\sim, \mathcal{E}^\tree)$ is a Hilbert space.
			\item If $x \neq y$, then there exists $u \in \mathcal{F}^\tree$ such that $u(x) \neq u(y)$.
			\item For any $p,q \in \tree(\branch)$, we have 
			\begin{equation} \label{e:resist}
				\sup \Biggl\{ \frac{\abs{u(x)-u(y)}^2}{\mathcal{E}^\tree(u,u)} \Biggm| u \in \wt{\mathcal{F}}_\tree, \mathcal{E}^\tree(u,u)>0 \Biggr\} = \metrictree(p,q).
			\end{equation}
			\item If $u \in \wt{\mathcal{F}}_\tree$, then $\wt{u}:=(0\vee u) \wedge 1 \in \wt{\mathcal{F}}_\tree$ and $\mathcal{E}^\tree(\wt{u},\wt{u}) \le \mathcal{E}^\tree(u,u)$.
		\end{enumerate}
	\end{lemma}
	
	\begin{proof}
		The first claim in (1) follows from observing that the gradient $\nabla$ is a linear operator on $\mathcal{A}(\tree(\branch))$. If $\mathcal{E}^\tree(f,f)=0$, then $\grad f=0$ $\lentree$-almost everywhere which by \eqref{e:defgrad} implies that $f$ is constant. The converse follows from observing that constants functions have vanishing gradient.
		
		Let us prove (4) next. By \eqref{e:defgrad} and the Cauchy-Schwarz inequality, we obtain the estimate 
		\[
		\abs{u(x)-u(y)}^2 \le \abs{\int_{[x,y]} \abs{\grad u} \,d\lentree}^2 \le \metrictree(x,y) \int_{[x,y]} \abs{\grad u}^2 \,d\lentree \le \metrictree(x,y) \mathcal{E}(u,u)
		\]
		for all $x,y \in \tree(\branch), u \in \wt{\mathcal{F}}_\tree$. This gives the upper bound on the supremum in \eqref{e:resist}. For the lower bound, we can easily verify that the supremum is attained by the function
		\begin{equation} \label{e:deffxy}
			f_{x,y}(z)= \metrictree(x,z_{x,y}), \quad \mbox{where $\{z_{x,y}\}=[x,y] \cap [x,z] \cap [y,z]$,}
		\end{equation}
		as $\abs{f_{x,y}(x)-f_{x,y}(y)}= \metrictree(x,y)=\mathcal{E}^{\tree}(f_{x,y},f_{x,y})$ in this case. The function $f_{x,y}$ chosen above also shows (3).
		
		Next, we show (2). By (1), $\mathcal{E}^\tree$ is a well-defined inner product on $\mathcal{F}^{\tree}/\sim$. We need to show the completeness of the inner product. To this end,  let $[u_n]_\sim, n \in \mathbb{N}$ denote an $\mathcal{E}^\tree$-Cauchy sequence, where $[u_n]_\sim$ denotes the equivalence class of $u_n$ under the relation $\sim$.
		Let $\grad u_n$ denotes the corresponding sequence of gradients (note that the gradient is independent of the choice of the representatives $u_n$). Since $[u_n]_\sim, n \in \mathbb{N}$ denote an $\mathcal{E}^\tree$-Cauchy, we have $\grad u_n$ is $L^2(\tree(\branch),\lentree)$-Cauchy and hence converges to a limit $g \in L^2(\tree(\branch),\lentree)$. Defining 
		\[
		u(x)=\int_p^x g(z)\,\lentree(dz), \quad \mbox{for all $x \in \tree(\branch)$},
		\]
		we have that $u \in \mathcal{A}(\tree(\branch))$ and $\grad u= g$ $\lentree$-almost everywhere. Since $\grad u_n \xrightarrow{L^2(\tree(\branch),\lentree)} \grad u$, we conclude that  $\lim_{n \to \infty} \mathcal{E}^\tree(u_n-u,u_n-u)=0$.
		%Let us normalize the representatives $u_n, n \in \mathbb{N}$ such that $u_n(p_\tree)=0$ for all $n \in \mathbb{Z}$.
		
		The contraction property (5) follows from observing that $\abs{\grad \wt{u}} \le \abs{\grad u}$ $\lentree$-almost everywhere, since $\abs{\wt{u}(x)-\wt{u}(y)} \le \abs{{u}(x)-{u}(y)}$ for all $x,y \in \tree(\branch)$.
	\end{proof}
	In order to show that $\wt{\mathcal{F}}_\tree$ contains sufficiently many continuous functions, we need to introduce the following notion. Let $n \in \mathbb{Z}$ and for any finitely supported function $\mathbf{v}:\llbracket n+1 , \infty \rrbracket \to \mathbb{Z}$ such that $\mathbf{v}(k) \in \llbracket 0, \branch(k)-1 \rrbracket$ for all $k \in \llbracket n+1 , \infty \rrbracket$, we set
	\begin{equation} \label{e:treeblock}
		\tree^{\mathbf{v}}(\branch)= \overline{\codetree(T_\infty^{\mathbf{v}})}= \{\codetree(\mathbf{u}): \mathbf{u} \in \ultra(\branch), \restr{\mathbf{u}}{\llbracket n+1, \infty \rrbracket}=\mathbf{v}\},
	\end{equation}
	where $T_\infty^{\mathbf{v}}$ is as defined in \eqref{e:defTinfu}
	By the continuity of the map $\codetree$ and Lemma \ref{l:inftree}(b), we have 
	\begin{equation} \label{e:rr2}
		\diam(	\tree^{\mathbf{v}}(\branch),\metrictree) \le 2^{n}.
	\end{equation}
	We define two special points in 	$p^{\mathbf{v}}_0, p^\mathbf{v}_1 \in \tree^{\mathbf{v}}(\branch)$ coming from Lemma \ref{l:inftree}(c)-(i) as 
	\begin{equation} \label{e:defpvi}
		p^{\mathbf{v}}_i=\codetree(\mathbf{v}_i), 
	\end{equation}
	where $i\in \{0,1\}$, $\mathbf{v}_i \in \ultra(\branch)$ is defined by $\restr{\mathbf{v}_i }{\llbracket n+1,\infty \rrbracket}=\mathbf{v}$ and $\mathbf{v}_i(k)=0$ for all $k \in \llbracket -\infty, n-1 \rrbracket$ and $\mathbf{v}_i(n)=i$ (note that $\metrictree(p^{\mathbf{v}}_0,p^{\mathbf{v}}_1)=2^n$). By Lemma \ref{l:inftree}(c), if t $\mathbf{u}, \mathbf{v}: \llbracket n+1,\infty \rrbracket \to \mathbb{Z}$ be \emph{distinct} finitely supported functions such that $\mathbf{u}(k), \mathbf{v}(k) \in \llbracket 0, \branch(k)-1 \rrbracket$ for all $k \in \llbracket n+1,\infty \rrbracket$, then exactly one of the three alternatives hold:
	\begin{equation} 
		\tree^{\mathbf{v}}(\branch)\cap \tree^{\mathbf{u}}(\branch)= \emptyset, \{p^{\mathbf{v}}_0\}, \mbox{ or } \{p^{\mathbf{v}}_1\}.
	\end{equation}
	For a finitely supported function $\mathbf{v}:\llbracket n+1 , \infty \rrbracket \to \mathbb{Z}$ such that $\mathbf{v}(k) \in \llbracket 0, \branch(k)-1 \rrbracket$ for all $k \in \llbracket n+1 , \infty \rrbracket$ and $i \in \{0,1\}$, we set
	\begin{equation}
		B_n^{\mathbf{v}}(\branch,i)= \Biggl\{\mathbf{u}: \llbracket n+1,\infty \rrbracket \Biggm| \begin{minipage}{250pt} $\mathbf{u}(k) \in \llbracket 0, \branch(k)-1 \rrbracket$ for all $k \in \llbracket n+1,\infty \rrbracket$ such that $\tree^{\mathbf{v}}(\branch)\cap \tree^{\mathbf{u}}(\branch)= \{p^{\mathbf{v}}_i\}$ \end{minipage} \Biggr\}.
	\end{equation} 
	Note that $\# B_n^{\mathbf{v}}(\branch,i) \le \sup_{k \in \mathbb{Z}}\branch(k) -1$ for all $n, \mathbf{v}:\llbracket n+1,\infty \rrbracket$ as above. For such $n \in \mathbb{Z}$,  $\mathbf{v}:\llbracket n+1 , \infty \rrbracket \to \mathbb{Z}$, we define $f_\mathbf{v}: \tree(\branch) \to [0,1]$ as
	\begin{equation} \label{e:deffv}
		f_\mathbf{v}(z)= \begin{cases}
			2^{-n}f_{p^{\mathbf{u}}_{1-i},p^{\mathbf{u}}_i}(z)&\mbox{if $z \in \tree^{\mathbf{u}}(\branch), \tree^{\mathbf{u}}(\branch) \cap \tree^{\mathbf{v}}(\branch)= \{p^{\mathbf{v}}_i\}, i \in \{0,1\}$,  } \\
			1 & \mbox{if $z \in \tree^{\mathbf{v}}(\branch)$, }\\
			0 & \mbox{otherwise},
		\end{cases}
	\end{equation}
	where $f_{p^{\mathbf{v}}_{1-i},p^{\mathbf{v}}_i}$ is as defined in \eqref{e:deffxy} and $\mathbf{u}:\llbracket n+1 , \infty \rrbracket \to \mathbb{Z}$ is a finitely supported such that $\mathbf{u}(k) \in \llbracket 0, \branch(k)-1 \rrbracket$ for all $k \in \llbracket n+1 , \infty \rrbracket$.
	The following lemma is elementary.
	\begin{lemma} \label{l:fvenergy}
		For a finitely supported function $\mathbf{v}:\llbracket n+1 , \infty \rrbracket \to \mathbb{Z}$ such that $\mathbf{v}(k) \in \llbracket 0, \branch(k)-1 \rrbracket$ for all $k \in \llbracket n+1 , \infty \rrbracket$ where $n \in \mathbb{Z}$, the function $f_\mathbf{v}: \tree(\branch) \to [0,1]$ defined in \eqref{e:deffv} satisfies 
		\[
		f_\mathbf{v} \in \contfunc_0(\tree(\branch)) \cap \wt{\mathcal{F}}_\tree, \quad \mathcal{E}^\tree(f_\mathbf{v},f_\mathbf{v}) \le 2^{-n+1} \sup_{k \in \mathbb{Z}} \branch(k).
		\]
	\end{lemma}
	\begin{proof}
		We note that $\abs{\nabla f_{\mathbf{v}}}$ is $\{0,2^{-n}\}$-valued $\lentree$-almost everywhere and it vanishes identically $\lentree$-almost everywhere except on  geodesics of the form $[p_0^{\mathbf{u}},p_1^{\mathbf{u}}]$, where  $\mathbf{u}:\llbracket n+1 , \infty \rrbracket \to \mathbb{Z}$ is a finitely supported such that $\mathbf{u}(k) \in \llbracket 0, \branch(k)-1 \rrbracket$ for all $k \in \llbracket n+1 , \infty \rrbracket$ and satisfies $\# \left( 	\tree^{\mathbf{v}}(\branch)\cap \tree^{\mathbf{u}}(\branch)\right)=1$. The set of such $\mathbf{u}:\llbracket n+1 , \infty \rrbracket \to \mathbb{Z}$ has cardinality at most $2 \sup_{k \in \mathbb{Z}} \branch(k)$ by Lemma \ref{l:inftree}(c). The conclusion follows from combining these facts along with the observation that $\lentree([p_0^{\mathbf{u}},p_1^{\mathbf{u}}])=2^n$ where $\mathbf{u}$ is as above.
	\end{proof}
	The following result shows that $(\mathcal{E}^{\tree},\wt{\mathcal{F}}_\tree)$ is regular in the sense of \cite[Definition 6.2]{Kig12}. 
	\begin{lemma} \label{l:regularres}
		The space $\contfunc_0(\tree(\branch)) \cap \wt{\mathcal{F}}_\tree$ is dense in $\contfunc_0(\tree(\branch))$ with respect to the supremum norm.
	\end{lemma}
	\begin{proof}
		By \cite[Theorem 6.3]{Kig12}, Lemma \ref{l:resform}(4) and the local compactness of $(\tree(\branch),\metrictree)$, it suffices to show that for any closed set $B \subset \tree(\branch)$, we have 
		\begin{equation} \label{e:rr1}
			B= \{x \in \tree(\branch) : f(x)=0 \mbox{ for all $f\in \wt{\mathcal{F}}_\tree$ such that $\restr{f}{B}\equiv 0$.}\}
		\end{equation}
		Clearly $B$ is a subset of the right-hand side above. Therefore it suffices to show that for any $x \notin B$, there exists a function $f\in \wt{\mathcal{F}}_\tree$ such that $\restr{f}{B}\equiv 0$ and $f(x) \neq 0$. By Proposition \ref{p:dsr}, there exists $\mathbf{t} \in \ultra(\branch)$ such that $x=\codetree(\mathbf{t})$. Since $B$ is closed and $x \notin B$, there exists $n \in \mathbb{Z}$ be such that $\metrictree(x,B)>2^n$. 
		
		Let $\mathbf{v}=\restr{\mathbf{t}}{\llbracket n  , \infty \rrbracket}$ and let $f_{\mathbf{v}}:\tree(\branch) \to [0,1]$ be as defined in \eqref{e:deffv}. Note that if $f_{\mathbf{v}}(z) \neq 0$, then either $z \in \tree^{\mathbf{v}}(\branch)$ or $z \in \tree^{\mathbf{u}}(\branch)$, where $\mathbf{u}: \llbracket n ,\infty \rrbracket \to \mathbb{Z}$ such that $\tree^{\mathbf{v}}(\branch) \cap \tree^{\mathbf{u}}(\branch) \neq \emptyset$. In either case 
		we have 
		\[
		\metrictree(z,x) \le \diam(\tree^{\mathbf{u}}(\branch), \metrictree)+\diam(\tree^{\mathbf{v}}(\branch), \metrictree) \stackrel{\eqref{e:rr2}}{\le} 2^{n-1}+2^{n-1}=2^n.
		\]
		Since $\metrictree(x,B)>2^n$, this along with Lemma \ref{l:fvenergy} shows that $\restr{f_\mathbf{v}}{B}\equiv 0$ and 	$f_\mathbf{v} \in \wt{\mathcal{F}}_\tree$.
	\end{proof}
	Next, we define the Dirichlet form  $(\dftree,\domtree)$ on $L^2(\tree(\branch),\meastree)$.
	\begin{definition} \label{d:dftree}
		We define the bilinear form $(\dftree,\domtree)$ using \eqref{e:resform}, where $\domtree \subset \wt{\mathcal{F}}_\tree$ is given by
		\[
		\domtree:= \wt{\mathcal{F}}_\tree\cap L^2(\tree(\branch),\meastree).
		\]
		Here with a slight abuse of notation, we denote  $\restr{\mathcal{E}^{\tree}}{\domtree\times \domtree}$ as $\mathcal{E}^{\tree}$.
	\end{definition}
	We record the basic properties of the bilinear form above. 
	\begin{prop} \label{p:treemmd}
		The linear form $(\dftree,\domtree)$ defined in $(\dftree,\domtree)$ is a strongly local, regular, Dirichlet form on $(\tree(\branch),\metrictree)$.
	\end{prop}
	\begin{proof}
		By Lemma \ref{l:resform} $(\mathcal{E}^\tree,\wt{\mathcal{F}}_\tree)$ is a resistance form in the sense of \cite[Definition 3.1]{Kig12} and this resistance form is regular (as given in \cite[Definition 6.2]{Kig12}). Thus by \cite[Theorem 9.4]{Kig12}, we conclude that $(\dftree,\domtree)$ is a regular Dirichlet form. The fact that it is strongly local follows from the expression \eqref{e:resform}.
	\end{proof}
	
	We note the \emph{energy measure} $\emtree(f,g)$ for $f,g \in \domtree$ is given by
	\begin{equation} \label{e:defemtree}
		\emtree(f,g)(A)= \int_A \grad f \grad g \, d\lentree, \quad \mbox{for any Borel set $A\subset \tree(\branch)$.}
	\end{equation}
	Using  the H\"older regularity estimate \eqref{e:resist}, we obtain an elementary upper bound for any function in $\domtree$.
	\begin{lemma} \label{l:ptub}
		For any $f \in \domtree, M > \dftree(f,f)$ and $\mathbf{t} \in \tree(\branch)$, we have 
		\begin{equation} \label{e:ptwise}
			\frac{f(\mathbf{t})^2}{4} V_{\mathbf{b}}(\abs{f(\mathbf{t})}^2/(4M^2)) \le \int_{\tree(\branch)} f^2\,d\meastree,
		\end{equation}
		where $V_\branch$ is as defined in \eqref{e:defvol}.
		In particular, any function $f \in \domtree$ is bounded.
	\end{lemma}
	\begin{proof}
		Let $M>\dftree(f,f), f \in \domtree$ and $\mathbf{t} \in \tree(\branch)$.
		Without loss of generality, we may assume $f(\mathbf{t})\neq 0$
		By \eqref{e:resist}, we have 
		$\abs{f(\mathbf{s})-f(\mathbf{t})}\le \metrictree(s,t)^{1/2} M^{1/2}$ and hence $\abs{f(\mathbf{s})} \ge \abs{f(\mathbf{t})}/2$
		for all $\mathbf{s} \in B_{\tree(\branch)}(t, \abs{f(\mathbf{t})}^2/(4M^2))$. Hence we obtain \eqref{e:ptwise} by integration.
		
		The final claim follows from \eqref{e:ptwise} and the fact that $\lim_{R \to \infty} \frac{R^2}{4} V_\branch(R^2/(4M^2))=\infty$.
	\end{proof}
	
	Next, we show that the diffusion on tree satisfies sub-Gaussian heat kernel estimate. We define an increasing homeomorphism $\sttree:[0,\infty) \to [0,\infty)$ as 
	\begin{equation} \label{e:sttree}
		\sttree(r)= \begin{cases}
			0 & \mbox{if $r=0$,}\\
			r V_\branch(r)& \mbox{if $r=2^n$ for some $n \in \mathbb{Z}$,}\\
			t	2^n V_\branch(2^{n}) + (1-t)	2^{n+1} V_\branch(2^{n+1}) & \mbox{if $r=t2^n+(1-t)2^{n+1},n \in \mathbb{Z}, t \in [0,1]$.} 
		\end{cases}
	\end{equation}
	This function $\sttree$ plays the role of space-time scaling function for the diffusion on $\tree(\branch)$ corresponding to the Dirichlet form $(\dftree,\domtree)$ on $L^2(\tree(\branch),\meastree)$.
	\begin{theorem} \label{t:hketree}
		The MMD space $(\tree(\branch),\metrictree,\meastree, \dftree, \domtree)$ satisfies sub-Gaussian heat kernel estimate \hyperlink{hke}{$\operatorname{HKE_f}(\sttree)$}.
	\end{theorem}
	\begin{proof}
		The result follows from invoking either \cite[Theorem 3.1]{Kum04} or \cite[Theorem 15.10]{Kig12}. The condition (c) in \cite[Theorem 15.10]{Kig12} is an immediate consequence of \eqref{e:resist} and Corollary \ref{c:meastree}. The required chain condition follows from the fact that $(\tree(\branch),\metrictree)$ is a geodesic space.
	\end{proof}
	Although the regularity of $(\dftree,\domtree)$ follows from abstract results in \cite{Kig12}, we provide a more concrete procedure to approximate an arbitrary function in $\domtree$ by multiplying by a sequence of increasing cutoff functions. This will be used to show the regularity of Dirichlet form that we construct on the Laakso-type space.
	\begin{lemma} \label{l:truncate}
		There exists a sequence of $[0,1]$-valued functions $\phi_n \in \domtree, n \in \mathbb{N}$ satisfying the following properties:
		\begin{enumerate}
			\item $\phi_n \equiv 1$ on $B_{\tree(\branch)}(p_\tree,2^{2n})$ and $\supp(\phi_n) \subset B_{\tree(\branch)}(p_\tree,2^{2n+1})$.
			\item There exists $C_1>1$ such that 
			\begin{equation} \label{e:trun1}
				\dftree(f\phi_n,f\phi_n) \le C_1 \dftree(f,f)+ \frac{C_1}{\sttree(2^{n})}\int f^2\, d\meastree.
			\end{equation}
			\item There exists $C_2>1$ such that for any $m,n \in \mathbb{N}$ with $m \le n$, $f \in \domtree$, we have 
			\begin{equation} \label{e:trun2}
				\dftree(f(\phi_n-\phi_m),f(\phi_n-\phi_m)) \le C_2\left(\int_{B(p_\tree,2^{2m})^c} f^2\, d\meastree+ \emtree(f,f)\left(B(p_\tree,2^{2m})^c \right)\right)
			\end{equation}
			In particular, the sequence $(f\phi_n)_{n \in \mathbb{N}}$ is $\dftree_1$-Cauchy and converges to $f$.
		\end{enumerate}
	\end{lemma}
	\begin{proof}
		By Theorem \ref{t:hketree} and \cite[Theorem 1.2]{GHL}, we have \hyperlink{cs}{$\operatorname{CS}(\sttree)$} for the MMD space $(\tree(\branch),\metrictree,\meastree, \dftree, \domtree)$. Hence there exists $C_1,A \in (1,\infty)$ such that for all $n \in \mathbb{N}$, we have cutoff functions $\phi_n$	for $B_{\tree(\branch)}(p_\tree,2^{2n}) \subset B_{\tree(\branch)}(p_\tree, A2^{2n})$ with 
		\begin{align} \label{e:tru1}
			\int {g}^2 \,d\emtree(\phi_n,\phi_n) &\le C_1 \emtree(g,g)\left(B_{\tree(\branch)}(p_\tree,A2^{2n}) \setminus B_{\tree(\branch)}(p_\tree,2^{2n})\right) \nonumber \\
			& \quad+ \frac{C_1}{\sttree(2^{n})}\int_{B_{\tree(\branch)}(p_\tree,A2^{2n}) \setminus B_{\tree(\branch)}(p_\tree,2^{2n})} g^2\,d\meastree
		\end{align}
		for all $g \in \domtree$ (recall that $g$ is continuous). If $m \le n$, then $\restr{\phi_n-\phi_m}{B(p_\tree,2^{2m})} \equiv 0$. Note that every function in $\domtree$ bounded and hence $\domtree$ is closed under multiplication due to \cite[Theorem 1.4.2(ii)]{FOT}. By Leibniz rule and Cauchy-Schwarz inequality, we have 
		\begin{align} \label{e:tru2}
			\lefteqn{\dftree(f(\phi_n-\phi_m),f(\phi_n-\phi_m))}  \nonumber\\
			&\le 2 \int f^2\, d\emtree(\phi_n-\phi_m, \phi_n-\phi_m) + 2\int (\phi_n-\phi_m)^2\, d\emtree(f,f)\nonumber\\
			&\le 4\int f^2\, d\emtree(\phi_n,\phi_n) + 4\int f^2\, d\emtree(\phi_m,\phi_m)+ 2\int_{B(p_\tree,2^{2m})^c} d\emtree(f,f) \nonumber \\
			&\stackrel{\eqref{e:tru1}}{\le} (2+8C_1)\int_{B(p_\tree,2^{2m})^c} d\emtree(\phi_m,\phi_m)  + \frac{8C_1}{\sttree(2^{m})} \int_{B(p_\tree,2^{2m})^c} f^2\,d\meastree.
		\end{align}
		By \eqref{e:tru2}, the estimate
		\[
		\int f^2(\phi_n-\phi_m)^2 \,d\meastree \le \int_{B_{\tree(\branch)(p_\tree,2^{2m})^c)}} f^2\,d\meastree
		\]
		for all $m \le n, f \in \domtree$ and applying the dominated convergence theorem, we conclude that $(f\phi_n)_{n \in \mathbb{N}}$ is $\dftree_1$-Cauchy and hence converges to $f$. The estimate \eqref{e:trun1} follows from an argument similar to \eqref{e:tru2}.
	\end{proof}
	
	\subsection{Diffusion on Laakso-type space} \label{ss:laakdiff}
	
	We first construct the Dirichlet form on the Laakso-type space $\laak(\glue,\branch)$ under the simplifying assumption  
	\begin{equation} \label{e:locreg}
		\lim_{k \to -\infty} \glue(k)=1, \mbox{ or equivalently, there exists $N_\glue \in \mathbb{Z}$ such that $g(k)=1$ for all $k \le N_\glue$.}	
	\end{equation} 
	The general case can be handled by approximating by a sequence of spaces that satisfy \eqref{e:locreg}. Such approximations are also useful to analyze the resulting diffusion process (see the proof of Poincar\'e inequality in \textsection \ref{ss:pi}).

	If \eqref{e:locreg} holds, then it is easy to see that $\ultra(\glue)$ is either finite or a countable space and the measure $\mathsf{m}_{\ultra(\glue)}$ is a multiple of the counting measure. 
	For $\mathbf{u} \in \ultra(\glue)$, define the isometric embedding 
	\begin{equation} \label{e:treelaak}
		I_\mathbf{u}: (\tree(\branch),\metrictree) \to (\laak(\glue,\branch),\metlaak), \quad 	I_\mathbf{u}(\mathbf{t})= [(\mathbf{u},\mathbf{t})]_{\laak}, \quad \mbox{for all $\mathbf{t} \in \tree(\branch)$}.
	\end{equation}
	Since $I_\mathbf{u}$ is continuous, for any $\mathbf{u} \in \ultra(\glue)$ and for any continuous function $f \in \contfunc(\laak(\glue,\branch))$, we have $f \circ I_\mathbf{u} \in \contfunc(\tree(\branch))$. The converse also holds under the condition \eqref{e:locreg}.
	\begin{lemma} \label{l:contlaak}
		Under the assumption \eqref{e:locreg}, 
		a function $f:\laak(\glue,\branch) \to \mathbb{R}$ is continuous if and only if $f \circ I_\mathbf{u}  \in \contfunc(\tree(\branch))$ for all $\mathbf{u} \in \ultra(\glue)$.
	\end{lemma}
	\begin{proof}
		As noted above, it suffices to show the if part. Let us assume that $f \circ I_\mathbf{u}  \in \contfunc(\tree(\branch))$ for all $\mathbf{u} \in \ultra(\glue)$. 
		
		Let $\varepsilon>0$, $[(\mathbf{v},\mathbf{t})]_{\laak} \in \laak(\glue,\branch)$. Let $N:=N_\glue$ be as given in \eqref{e:locreg}.

		If $\mathbf{t} \notin \cup_{n=N}^\infty \worm_n$, then by Lemma \ref{l:worm}(ii), we have $\metrictree\left(\mathbf{t}, \cup_{n=N}^\infty \worm_n\right)>0$. For any $r<\metrictree\left(\mathbf{t}, \cup_{n=N}^\infty \worm_n\right)$, we have $B_{\laak(\glue,\branch)}([(\mathbf{v},\mathbf{t})]_{\laak},r)= I_{\mathbf{v}}\left( B_{\tree(\branch)}(\mathbf{t},r)\right)$ and hence by continuity of $f\circ I_{\mathbf{v}}$, there exists $0<\delta<r$ such that 
		$\metlaak(p,[(\mathbf{v},\mathbf{t})]_{\laak})<\delta$ implies $\abs{f(p)-f([(\mathbf{v},\mathbf{t})]_{\laak})}<\varepsilon$.
		
		If $\mathbf{t} \in \cup_{n=N}^\infty \worm_n$, then by Lemma \ref{l:worm}(ii), we have $B_{\tree(\branch)}(\mathbf{t}, r) \cap \left(\cup_{n=N}^\infty \worm_n\right) = \{\mathbf{t}\}$ for any $r \le 2^N$. In this case the equivalence class $[(\mathbf{v},\mathbf{t})]_{\laak}$ is a finite set, say $\{(\mathbf{v}_i,\mathbf{t}): i =1,\ldots,n \}$, where $n \le \sup_{k \in \mathbb{Z}} \glue(k)$. Therefore for any $r \le 2^N$, we have $B_{\laak(\glue,\branch)}([(\mathbf{v},\mathbf{t})]_{\laak},r)= \cup_{i=1}^n I_{\mathbf{v}_i}\left(B_{\tree(\branch)}(\mathbf{t},r) \right)$.
		For each $i=1,\ldots,n$, by the continuity of $f\circ I_{\mathbf{v}_i}$, there exists $0<\delta_i<2^N$ such that $\metlaak(p,I_{\mathbf{v}_i}(\mathbf{t}))<\delta$ implies $\abs{f(p)-f(I_{\mathbf{v}_i}(\mathbf{t}))}<\varepsilon$. Hence by choosing $\delta=\min_{i=1,\ldots,n} \delta_i$, we obtain the continuity of $f$ at $[(\mathbf{v},\mathbf{t})]_{\laak}$.
	\end{proof}
	The   advantage of Laakso-type spaces satisfying \eqref{e:locreg} is that
	the definition of the Dirichlet form is simpler in this case. 
	Under  \eqref{e:locreg}, we define a bilinear form $\dflaak:\domlaak \times \domlaak \to \mathbb{R}$ on a subspace $\domlaak$ of $\contfunc(\laak(\glue,\branch))$ as  
	\begin{align} \label{e:defdflaak}
		\domlaak &= \Biggl\{f \in \contfunc(\laak(\glue,\branch)) \Biggm| \begin{minipage}{300pt}
			$I_{\mathbf{u}}(f) \in \domtree$ for all $\mathbf{u}\in \ultra(\glue)$, and $\int_{\laak(\glue,\branch)} f^2\,d\measlaak + \sum_{\mathbf{u} \in \ultra(\glue)} \dftree(f\circ I_\mathbf{u},f\circ I_\mathbf{u}) \mathsf{m}_{\ultra(\glue)}(\{\mathbf{u}\}) <\infty $
		\end{minipage} \Biggr\}, \nonumber \\
		\dflaak(f,g)&= \sum_{\mathbf{u} \in \ultra(\glue)} \dftree(f\circ I_\mathbf{u},g\circ I_\mathbf{u}) \mathsf{m}_{\ultra(\glue)}(\{\mathbf{u}\}),\quad \mbox{for all $f,g \in \domlaak$.}
	\end{align}
	Let us express the Dirichlet form similar in terms of the gradient analogous to \eqref{e:resform}. Since $\cup_{n \in \mathbb{Z}} \worm_n$ is countable and $\lentree$ is non-atomic, for any $f \in \domlaak$ there is a $\lenlaak$-almost everywhere well-defined function $\grad^{\laak} f: \laak(\glue,\branch) \to \mathbb{R}$ such that 
	\begin{equation} \label{e:defgradlaak}
		I_{\mathbf{u}} \circ \grad^{\laak} f = \nabla(I_{\mathbf{u}} \circ f) \quad \mbox{$\lentree$-almost everywhere for all $\mathbf{u} \in \ultra(\glue)$.}
	\end{equation}
	By \eqref{e:defdflaak}, \eqref{e:deflenlaak} and \eqref{e:defgradlaak}, we have (assuming \eqref{e:locreg})
	\begin{equation} \label{e:dflaakalt}
		\dflaak(f_1,f_2)= \int_{\laak(\glue,\branch)} \grad^\laak f_1 \grad^\laak f_2 \, d\lenlaak, \quad \mbox{for all $f_1,f_2 \in \domlaak$.}
	\end{equation}
	We check that this defines a Dirichlet form.
	\begin{prop} \label{p:dflaakprelim}
		Under the assumption \eqref{e:locreg}, the bilinear form $(\dflaak,\domlaak)$ is a  strongly local, regular, Dirichlet form on $L^2(\laak(\glue,\branch),\measlaak)$. In this case, the corresponding energy measure is given by 
		\begin{align} \label{e:emlaakapprox}
			\emlaak(f,f) (A)&= \sum_{\mathbf{u} \in \ultra(\glue)} \mathsf{m}_{\ultra(\glue)}(\{\mathbf{u}\}) I_{\mathbf{u}}^* \left( \emtree(f \circ I_{\mathbf{u}},f \circ I_{\mathbf{u}})\right)(A) \nonumber\\
			&=  \int_A \abs{\grad^\laak f}^2 \,  d\lenlaak, \quad \mbox{for all Borel sets $A \subset \laak(\glue,\branch)$.}
		\end{align}
	\end{prop}
	\begin{proof}
		The bilinearity, Markovian and   strong locality properties for $(\dflaak,\domlaak)$ follow easily from the corresponding properties of $(\dftree,\domtree)$.  
		
		Let us check that $(\dflaak,\domlaak)$ is a closed form. To this end, let us choose a $\dflaak_1$-Cauchy sequence $(f_n)_{n \in \mathbb{N}}$. Since $\dftree_1(f \circ I_\mathbf{u},f \circ I_\mathbf{u}) \le \mathsf{m}_{\ultra(\glue)}(\{\mathbf{u}\})^{-1} \dflaak_1(f ,f )$ for all $f \in \domlaak$ (note that $\mathsf{m}_{\ultra(\glue)}$ is a positive multiple of the counting measure due to \eqref{e:locreg}). Therefore $(f_n \circ I_\mathbf{u})$ is $\dftree_1$-Cauchy for each $\mathbf{u} \in \ultra(\glue)$ and hence $\dftree_1$-converges in a function $f_\mathbf{u} \in \domtree$ for each $\mathbf{u} \in \ultra(\glue)$. 
		
		By Lemma \ref{l:ptub},  $f_n \circ I_\mathbf{u}$ is uniformly bounded for all   $\mathbf{u} \in \ultra(\glue), n \in \mathbb{N}$. Similarly, by \eqref{e:resist} $f_n \circ I_\mathbf{u}$ is equicontinuous for all   $\mathbf{u} \in \ultra(\glue), n \in \mathbb{N}$.
		By Arzela-Ascoli theorem and passing to a subsequence using the diagonal argument, we may assume that $(f_n \circ I_\mathbf{u})_{n \in \mathbb{N}}$ converges to $f_\mathbf{u}$ in the supremum norm for each $\mathbf{u} \in \ultra(\glue)$. In particular, by Lemma \ref{l:contlaak} this implies that there is a well-defined continuous function $f:\laak(\glue,\branch) \to \mathbb{R}$ such that $f([(\mathbf{u},\mathbf{t})]_{\laak})=f_\mathbf{u}(\mathbf{t})$ for all $(\mathbf{u},\mathbf{t}) \in \mathcal{P}(\glue,\branch)$. By the completeness of $L^2(\laak(\glue,\branch),\measlaak)$, $f_n \xrightarrow{L^2}f$. 
		
		It remains to show that $\lim_{n\to \infty}\dflaak(f_n-f,f_n-f) = 0$. Define $\grad f_n: \mathcal{P}(\glue,\branch) \to \mathbb{R}$ as 
		$\grad f_n(\mathbf{u},\mathbf{t})= \grad (f_n \circ I_\mathbf{u})(\mathbf{t})$, for all $(\mathbf{u},\mathbf{t}) \in  \mathcal{P}(\glue,\branch)$ for all $n \in \mathbb{N}$. Note that $\grad f_n$ is well-defined   $\mathsf{m}_{\ultra(\glue)} \times \lentree$-almost everywhere for all $n \in \mathbb{N}$. Since $(f_n)$ is $\dflaak_1$-Cauchy, we have that 	$(\grad f_n)_{n \in \mathbb{N}}$ is a Cauchy sequence in $L^2(\mathcal{P}(\glue,\branch), \mathsf{m}_{\ultra(\glue)} \times \lentree)$ and therefore it converges to a limit $g:\mathcal{P}(\glue,\branch)\to \mathbb{R}$. By the $\dftree_1$-convergence of $(f_n \circ I_\mathbf{u})_{n \in \mathbb{N}}$ to $f_\mathbf{u}$, we obtain $g(\mathbf{u},\cdot)=\nabla f_\mathbf{u}$ for $\lentree$-almost every point in $\tree(\branch)$ and for all $\mathbf{u} \in \ultra(\branch)$. Therefore $\dflaak(f_n-f,f_n-f)= \int_{\mathcal{P}(\glue,\branch)}  \abs{\nabla f_n - g}^2 \, d(\mathsf{m}_{\ultra(\glue)} \times \lentree) \xrightarrow{n \to \infty}0$. This completes the proof that $(\dflaak,\domlaak)$ is a closed form.

		To check the regularity of the Dirichlet form we need to show that $\domlaak \cap \contfunc_0(\laak(\glue,\branch))$ is dense in $\contfunc_\infty(\laak(\glue,\branch))$ with respect to the supremum norm. To this end, it  suffices to show that $\domlaak\cap \contfunc_0(\laak(\glue,\branch))$ separates points by Stone–Weierstrass theorem and the fact that $\domlaak\cap \contfunc_0(\laak(\glue,\branch))$  is an algebra due to \cite[Theorem 1.4.2(ii)]{FOT}. To this end pick a pair of distinct points $[(\mathbf{u}_i,\mathbf{t}_i)]_{\laak}\in \laak(\glue,\branch)$ for $i=1,2$. Let $$0<r< 2^{N_\glue} \wedge (\metlaak([(\mathbf{u}_1,\mathbf{t}_1)]_{\laak},[(\mathbf{u}_2,\mathbf{t}_2)]_{\laak})/2),$$
		where $N_\glue$ is as given in \eqref{e:locreg}.
		By Lemma \ref{l:worm}(ii) and \eqref{e:projmet}, we have $$\# \left(\projtree \left(B_{\laak(\glue,\branch)}\left([(\mathbf{u}_1,\mathbf{t}_1)]_{\laak},r\right)\right) \cap \cup_{n \ge N_\glue} \worm_n \right) \le 1 $$
		and hence we have $\# U(\mathbf{u}_1,\mathbf{t}_1,r) \le \sup_{k\in \mathbb{Z}} \branch(k)$, where 
		\[
		U(\mathbf{u}_1,\mathbf{t}_1,r):=\{\mathbf{u} \in \ultra(\branch): I_{\mathbf{u}}(\tree(\branch)) \cap B_{\laak(\glue,\branch)}\left([(\mathbf{u}_1,\mathbf{t}_1)]_{\laak},r\right) \neq \emptyset \}.
		\]
		Let $n \in \mathbb{Z}$ be such that $2^{n}  < r \le 2^{n+1}$,  $\mathbf{t}_1= \codetree(\wt{\mathbf{v}}_1)$ for some $\mathbf{v} \in \ultra(\branch)$ and $\mathbf{v}= \restr{\wt{\mathbf{v}}_1}{\llbracket n, \infty \rrbracket}$. Then the function $f_\mathbf{v}$ defined in \eqref{e:deffv} satisfies $f\in \domtree$ and $\supp(f) \subset B_{\tree(\branch)}(\mathbf{t},2^n) \subset B_{\tree(\branch)}(\mathbf{t},r)$ by Lemma \ref{l:fvenergy}.
		This implies the function $g \in C(\laak(\glue,\branch))$ defined by
		\[
		g \circ I_\mathbf{u} \equiv \begin{cases}
			f_{\mathbf{v}} & \mbox{if $\mathbf{u} \in U(\mathbf{u}_1,\mathbf{t}_1,r)$,}\\
			0  &\mbox{if $\mathbf{u} \in \ultra(\branch)\setminus U(\mathbf{u}_1,\mathbf{t}_1,r)$.}
		\end{cases}
		\]
		This implies that $\supp(g) \subset B_{\laak(\glue,\branch)}([(\mathbf{u}_1,\mathbf{t}_1)]_{\laak},2r)$ and hence $g([(\mathbf{u}_1,\mathbf{t}_1)]_{\laak})=1$ and  $g([(\mathbf{u}_2,\mathbf{t}_2)]_{\laak})=0$.
		
		In order to conclude the proof of regularity, we need to show that 
		$\domlaak \cap \contfunc_0(\laak(\glue,\branch))$ is dense in $\domlaak$ with respect to the $\dflaak_1$-inner product. 
		Let $f$ be an arbitrary function in $\domlaak$. Let $(\phi_n)\in \domtree$ be a sequence of functions as given by Lemma \ref{l:truncate}. Then we define the function $f_n : \laak(\glue,\branch) \to \mathbb{R}$ as 
		\begin{equation}
			(f_n \circ I_\mathbf{u})(\mathbf{t})= \phi_n(\mathbf{t}) (f\circ I_\mathbf{u})(\mathbf{t}), \quad \mbox{for all $\mathbf{u} \in \ultra(\glue), \mathbf{t} \in \tree(\branch)$.}
		\end{equation}
		It is easy to verify that $f_n$ is a well-defined function on $\laak(\glue,\branch)$ and $f_n \in \contfunc_0(\laak(\glue,\branch))$ for all $n \in \mathbb{N}$. By Lemma \ref{l:truncate}, there exists $C_2>1$ such that for all $f \in \domlaak, n,m \in \mathbb{N}$ with $m \le n$
		\begin{align*}
			\lefteqn{\dflaak(f_n-f_m,f_n-f_m)} \\
			&\le	C_1 \sum_{\mathbf{u}\in \ultra(\glue)} \mathsf{m}_{\ultra(\glue)}(\mathbf{u}) \left(\int_{B(p_\tree,2^{2m})^c} (I_\mathbf{u}\circ f)^2\, d\meastree+ \emtree(I_\mathbf{u}\circ f,I_\mathbf{u}\circ f)\left(B(p_\tree,2^{2m})^c \right)\right).
		\end{align*}
		This along with dominated convergence theorem implies that $(f_n)_{n\in\mathbb{N}}$ is $\dflaak_1$-Cauchy and converges to $f$. This completes the proof that the Dirichlet form $(\dflaak,\domlaak)$  on $L^2(\laak(\glue,\branch),\measlaak)$. 
		
		Note that  every function in $\domlaak$ is bounded due to Lemma \ref{l:ptub} and hence forms an algebra by \cite[Lemma 1.4.2(ii)]{FOT}.
		The expression \eqref{e:emlaakapprox} for energy measure follows from \eqref{e:defdflaak}, \eqref{e:dflaakalt}, the definition of energy measure and the Leibniz rule for gradients on $\tree(\branch)$. 
	\end{proof}
	
	%Note that by the universal property of quotient maps (\cite[Proposition 2.4.2]{Eng}) and by noting from Proposition \ref{p:laak}(b) that $\quolaak:(\mathcal{P}(\glue,\branch),\mathsf{d}_{\mathcal{P}(\glue,\branch)}  ) \to (\laak(\glue,\branch),\metlaak)$ is a quotient map,  a function $f: (\laak(\glue,\branch),\metlaak) \to \mathbb{R}$ is continuous if and only if $f \circ \quolaak \in C(\mathcal{P}(\glue,\branch))$.
	
	We would like to remove the assumption \eqref{e:locreg} in Proposition \ref{p:dflaakprelim}. To this end, we would like to approximate an arbitrary Laakso-type space $\laak(\glue,\branch)$ by those that satisfy the additional assumption \eqref{e:locreg}.
	The given Laasko-type space $\laak(\glue,\branch)$ can be viewed as an projective limit of $\laak(\glue_n,\branch)$ as $n \to \infty$ using the projection maps we define below.
	To this end, for any $\mathbf{g}: \mathbb{Z} \to \mathbb{N}$ and $n \in \mathbb{N}$, we set 
	\begin{equation} \label{e:defgn}
		\glue_n(k)= \begin{cases}
			\glue(k) & \mbox{if $k \ge -n $,}\\
			1 & \mbox{if $k < -n$.}
		\end{cases}
	\end{equation}
	For any $n \in \mathbb{N}$, we  define a \emph{projection map} $\projultra_n: \ultra(\glue) \to \ultra(\glue_n)$ as 
	\begin{equation} \label{e:defprojultra}
		[\projultra_n(\mathbf{u})](k) = \begin{cases}
			\mathbf{u}(k)& \mbox{if $k \in \llbracket -n , \infty \rrbracket$,}\\
			0 & \mbox{otherwise,}
		\end{cases}
	\end{equation}
	for all $\mathbf{u} \in \ultra(\glue)$.
	It is easy to verify that $\projultra_n$ is $1$-Lipschitz and induces  another $1$-Lipschitz map $\projlaak_n: \laak(\glue,\branch) \to \laak(\glue_n,\branch)$ for any $n \in \mathbb{N}$ given by
	\begin{equation} \label{e:defprojlaak}
		\projlaak_n([(\mathbf{u},\mathbf{t})]_{\laak})=[(\projultra_n(\mathbf{u}),\mathbf{t})]_{\laak}, \quad \mbox{for all $\mathbf{u}\in \ultra(\glue), \mathbf{t} \in \tree(\branch)$.}
	\end{equation}
	We note that the identity map induces   isometric embeddings of $\laak(\glue_n,\branch)$ in $\laak(\glue,\branch)$ and similarly of $\laak(\glue_m,\branch)$ in $\laak(\glue_n,\branch)$
	for all $m,n \in \mathbb{N}$ with $m \le n$. We similarly have projection maps
	$\projultra_{n,m}:\ultra(\glue_n) \to \ultra(\glue_m), \projlaak_{n,m}:\laak(\glue_n,\branch) \to \laak(\glue_m,\branch)$ for all $n,m \in \mathbb{N}$ with $n \ge m$.

	By Proposition \ref{p:dflaakprelim} for any bounded function $\mathbf{g}:\mathbb{Z} \to \mathbb{N}$ there is a sequence of MMD spaces $(\laak(\glue_n,\branch),\mathsf{d}_{\laak(\glue_n,\branch)},\mathsf{m}_{\laak(\glue_n,\branch)},\mathcal{E}^{\laak_n},\mathcal{F}^{\laak_n})$. Let $\generlaak_n:D(\generlaak_n) \to L^2(\laak(\glue_n,\branch),\mathsf{m}_{\laak(\glue_n,\branch)})$ denote the  (non-negative definite) generators corresponding to associated Dirichlet forms for $n \in \mathbb{N}$. 
	In the following lemma, we show that these MMD spaces are `consistent' with respect to the projection maps $\projlaak_{n,m}:\laak(\glue_n,\branch) \to \laak(\glue_m,\branch)$  where $m \le n$. 
	\begin{lemma} \label{l:consistent}
		Let $m,n \in \mathbb{N}$ such that $m \le n$. 
		\begin{enumerate}[(a)]
			\item $(\projlaak_{n,m})_*\left(\mathsf{m}_{\laak(\glue_n,\branch)} \right) =\mathsf{m}_{\laak(\glue_m,\branch)}$ and $(\projlaak_{n})_*\left(\measlaak\right)  =\mathsf{m}_{\laak(\glue_n,\branch)}$.
			\item For all $f \in \mathcal{F}^{\laak_m}$, we have $f \circ \projlaak_{n,m} \in \mathcal{F}^{\laak_n}$ and $\mathcal{E}^{\laak_n}(f \circ \projlaak_{n,m} ,f \circ \projlaak_{n,m} )=\mathcal{E}^{\laak_m}(f ,f)$.
			\item For $f \in D(\generlaak_m)$, we have $f \circ \projlaak_{n,m} \in D(\generlaak_n)$ and $\generlaak_n(f \circ \projlaak_{n,m})= (\generlaak_m(f)) \circ \projlaak_{n,m}$.
		\end{enumerate}
	\end{lemma}
	
	\begin{proof}
		\begin{enumerate}[(a)]
			\item This follows from observing that $(\projultra_{n,m})_*(\mathsf{m}_{\ultra(\glue_n)})=\mathsf{m}_{\ultra(\glue_m)}$ and  $(\projultra_{n})_*(\mathsf{m}_{\ultra(\glue)})=\mathsf{m}_{\ultra(\glue_n)}$ for any $n,m \in \mathbb{N}$ with $n \ge m$.
			\item Let $f \in \mathcal{F}^{\laak_m}$, then for any $\mathbf{u} \in \ultra(\glue_n)$, we have $I_{\mathbf{u}}(f \circ \projlaak_{n,m}) =I_{\projultra_{n,m}(\mathbf{u})} f$. Therefore 
			\begin{align*}
				\lefteqn{\sum_{\mathbf{u} \in \ultra(\glue_n)}\dftree(I_{\mathbf{u}}(f \circ \projlaak_{n,m}),I_{\mathbf{u}}(f \circ \projlaak_{n,m})) \mathsf{m}_{\ultra(\glue_n)}(\{\mathbf{u}\})} \\
				&= 	\sum_{\mathbf{u} \in \ultra(\glue_n)} \dftree(I_{\projultra_{n,m}(\mathbf{u})} (f),I_{\projultra_{n,m}(\mathbf{u})} (f)) \mathsf{m}_{\ultra(\glue_n)}(\{\mathbf{u}\})   \\
				&= \sum_{\mathbf{v} \in \ultra(\glue_m)} \dftree( I_{\mathbf{v}} (f),I_{\mathbf{v}} (f)   ) \left((\projultra_{n,m})_*(\mathsf{m}_{\ultra(\glue_n)})\right)(\{\mathbf{v}\})\\
				&\stackrel{(a)}{=}\sum_{\mathbf{v} \in \ultra(\glue_m)}   \dftree( I_{\mathbf{v}} (f),I_{\mathbf{v}} (f)   )   \mathsf{m}_{\ultra(\glue_m)}(\{\mathbf{v}\}) = \mathcal{E}^{\laak_m}(f ,f) <\infty.
			\end{align*}
			By part (a), $f$  and $f \circ  \projlaak_{n,m}$  have the same $L^2$-norms in the respective spaces. %$L^2(\laak(\glue_m,\branch),\mathsf{m}_{\laak(\glue_m,\branch)})$ to $L^2(\laak(\glue_n,\branch),\mathsf{m}_{\laak(\glue_n,\branch)})$.
			Hence we conclude that $f \circ \projlaak_{n,m} \in \mathcal{F}^{\laak_n}$ and $\mathcal{E}^{\laak_n}(f \circ \projlaak_{n,m} ,f \circ \projlaak_{n,m} )=\mathcal{E}^{\laak_m}(f ,f)$.
			\item By a calculation similar to part (b), for any $f \in  \mathcal{F}^{\laak_m}, g \in \mathcal{F}^{\laak_n}$, we have 
			\begin{align} \label{e:cons1}
				\lefteqn{\sum_{\mathbf{u} \in \ultra(\glue_n)}\dftree(I_{\mathbf{u}}(f \circ \projlaak_{n,m}),I_{\mathbf{u}}(g)) \mathsf{m}_{\ultra(\glue_n)}(\{\mathbf{u}\})} \nonumber\\
				&=\sum_{\mathbf{v} \in \ultra(\glue_m)}  \dftree\left(I_{\mathbf{v}}(f),  \sum_{\mathbf{u} \in (\projultra_{n,m})^{-1}\left(\{\mathbf{v}\}\right)}I_{\mathbf{u}}(g) \frac{\mathsf{m}_{\ultra(\glue_n)}(\{\mathbf{u}\})}{\mathsf{m}_{\ultra(\glue_m)}(\{\mathbf{v}\})}\right)  \mathsf{m}_{\ultra(\glue_m)}(\{\mathbf{v}\}).
			\end{align}
			This motivates us to define an \emph{averaging operator} $A_{n,m}: \contfunc(\laak(\glue_n, \branch)) \to \contfunc(\laak(\glue_m, \branch))$ for all $n,m \in \mathbb{N}$ with $n \ge m$ defined as the function such that 
			\begin{equation} \label{e:defapprox}
				I_{\mathbf{v}} \circ A_{n,m}(h)= \sum_{\mathbf{u} \in (\projultra_{n,m})^{-1}\left(\{\mathbf{v}\}\right)}I_{\mathbf{u}}(h) \frac{\mathsf{m}_{\ultra(\glue_n)}(\{\mathbf{u}\})}{\mathsf{m}_{\ultra(\glue_m)}(\{\mathbf{v}\})}= \frac{1}{\#(\projlaak_{n,m})^{-1}\left(\{\mathbf{v}\}\right)} \sum_{\mathbf{u} \in (\projultra_{n,m})^{-1}\left(\{\mathbf{v}\}\right)}I_{\mathbf{u}}(h)
			\end{equation}
			for all $\mathbf{v} \in \ultra(\glue_m), h \in \contfunc(\laak(\glue_n,\branch))$. It is easy to verify that $A_{n,m}$ is well-defined and that $A_{n,m}(h)$ is continuous for all $h \in \contfunc(\laak(\glue_n,\branch))$. By Cauchy-Schwarz inequality, we have
			\begin{align} \label{e:cons2}
				\sum_{\mathbf{v} \in \ultra(\glue_m)} \dftree\left((I_{\mathbf{v}} \circ A_{n,m})(h),(I_{\mathbf{v}} \circ A_{n,m})(h) \right)\mathsf{m}_{\ultra(\glue_m)}(\{\mathbf{v}\}) \le \mathcal{E}^{\laak_n}(h,h), \quad \mbox{for all $h \in \mathcal{F}^{\laak_n}$ }
			\end{align}  
			and 
			\begin{equation} \label{e:cons3}
				\int_{\laak(\glue_m,\branch)} \abs{A_{n,m}(h)}^2 \, d\mathsf{m}_{\laak(\glue_m,\branch)} \le \int_{\laak(\glue_n,\branch)} h^2 \, d\mathsf{m}_{\laak(\glue_n,\branch)}, \quad \mbox{for all $h \in \contfunc(\laak(\glue_n,\branch))$.}
			\end{equation}
			By \eqref{e:cons2} and \eqref{e:cons3},  $h \in \mathcal{F}^{\laak_n}$ implies $A_{n,m}(h) \in \mathcal{F}^{\laak_m}$. Hence by \eqref{e:cons1}, \eqref{e:defapprox}, we have for all $f \in  D(\generlaak_m), g \in \mathcal{F}^{\laak_n}$,
			\begin{align} \label{e:cons4}
				\mathcal{E}^{\laak_n}(f \circ \projlaak_{n,m},g)&=\sum_{\mathbf{u} \in \ultra(\glue_n)}\dftree(I_{\mathbf{u}}(f \circ \projlaak_{n,m}),I_{\mathbf{u}}(g)) \mathsf{m}_{\ultra(\glue_n)}(\{\mathbf{u}\}) \nonumber \\
				&= \mathcal{E}^{\laak_m}(f, A_{n,m}(g))  \stackrel{\eqref{e:generate}}{=} \int \generlaak_m(f)  A_{n,m}(g) \, d\mathsf{m}_{\laak(\glue_m,\branch)}\quad \mbox{(by \eqref{e:cons1},\eqref{e:defapprox})} \nonumber \\
				&= \int (\projlaak_{n,m} \circ \generlaak_m )(f)  g \, d\mathsf{m}_{\laak(\glue_n,\branch)}
			\end{align}
			By \cite[Corollary 1.3.1]{FOT}, \eqref{e:cons4} and \eqref{e:cons3}, we obtain the desired conclusion.
		\end{enumerate}
	\end{proof}
	
	The averaging operator $A_{n,m}$ defined in \eqref{e:defapprox} admits a limit as $n \to \infty$ that we introduce below.
	\begin{definition}
		For $f \in \contfunc_0(\laak(\glue,\branch)), m \in \mathbb{N}$, we define $A_{\infty,m}(f): \laak(\glue_m,\branch) \to \mathbb{R}$ as 
		\begin{equation} \label{e:defaver}
			(A_{\infty,m}(f))([(\mathbf{u},\mathbf{t})]_{\laak})= \frac{1}{\mathsf{m}_{\ultra(\glue)} \left((\projultra_{m})^{-1}(\{\mathbf{u}\})\right)}\int_{(\projultra_{m})^{-1}(\{\mathbf{u}\})} f([(\mathbf{v},\mathbf{t})]_{\laak}) \, \mathsf{m}_{\ultra(\glue)}(d\mathbf{v})
		\end{equation}
		for any $(\mathbf{u},\mathbf{t}) \in \mathcal{P}(\glue_m,\branch)$.
		It is easy to see that $A_{\infty,m}(f)$ is well-defined.
		By the uniform continuity of $f$ restricted to closed and bounded sets, we have that $A_{\infty,m}(f) \in \contfunc_0(\laak(\glue_m,\branch))$ whenever $f \in \contfunc_0(\laak(\glue,\branch))$. Moreover the support of $A_{\infty,m}(f)$  is contained in  $(\projlaak_{m})(\supp(f))$.
		
	\end{definition}

	Let us define an operator $\generlaak_\infty: D(\generlaak_\infty) \to L^2(\laak(\glue,\branch),\measlaak)$ as 
	\begin{align} \label{e:deffred}
		D(\generlaak_\infty)&:= \Biggl\{ f \circ \projlaak_n \Biggm| f \in D(\generlaak_n), n \in \mathbb{N} \Biggr\}, \nonumber \\
		\generlaak_\infty(f \circ \projlaak_n)&:= (\generlaak_n(f))\circ\projlaak_n, \quad \mbox{for all $f \in D(\generlaak_n), n \in \mathbb{N}$.}
	\end{align}
	Let us check that \eqref{e:deffred} gives a well-defined operator. Indeed, if $f \circ \projlaak_n= g \circ \projlaak_m$ for $f \in D(\generlaak_n), g \in D(\generlaak_m)$ where $m < n, m,n \in \mathbb{N}$, then $f=  g \circ \projlaak_{n,m}$. This along with Lemma \ref{l:consistent}(c) implies that 
	\[
	(\generlaak_n(f))\circ\projlaak_n=(\generlaak_n(g \circ \projlaak_{n,m}))\circ\projlaak_n= \generlaak_m(g) \circ \projlaak_{n,m} \circ\projlaak_n= (\generlaak_m(g))\circ\projlaak_m.
	\]
	In other words, $\generlaak_\infty$ is a well-defined. 
	\begin{lemma} \label{l:friedprep}
		The operator  $\generlaak_\infty: D(\generlaak_\infty) \to L^2(\laak(\glue,\branch),\measlaak)$ is densely defined, non-negative definite and  symmetric on   $L^2(\laak(\glue,\branch),\measlaak)$.
	\end{lemma}
	\begin{proof}
		By Lemma \ref{l:consistent}(b), for all $n \in \mathbb{N}, f \in D(\generlaak_n)$, we have 
		\begin{align*}
			\int( f \circ \projlaak_{n}) \generlaak_\infty(f \circ \projlaak_{n}) \, d\measlaak&= \int  f \generlaak_n(f) \, d\mathsf{m}_{\laak(\glue_n,\branch)} \quad \mbox{(by \eqref{e:deffred} and Lemma \ref{l:consistent}(a))} \\
			&= \mathcal{E}^{\laak_n}(f,f) \ge 0 \quad \mbox{(by \eqref{e:generate})}.
		\end{align*}
		Therefore $\laak_\infty$ is non-negative definite. 
		
		For all $n,m \in \mathbb{N}$ with $m \le n$, $f \in D(\generlaak_n), m \in D(\generlaak_m)$
		\begin{align*}
			\int( f \circ \projlaak_{n}) \generlaak_\infty(g \circ \projlaak_{m}) \, d\measlaak&= \int  f \generlaak_n(g \circ \projlaak_{n,m}) \, d\mathsf{m}_{\laak(\glue_n,\branch)} \quad \mbox{(by Lemma \ref{l:consistent}(a,c))} \\
			&\stackrel{\eqref{e:generate}}{=} \mathcal{E}^{\laak_n}(f,g \circ \projlaak_{n,m}) \stackrel{\eqref{e:generate}}{=}   \int  (g \circ \projlaak_{n,m}) \generlaak_n(f) \, d\mathsf{m}_{\laak(\glue_n,\branch)} \\
			&= \int (g \circ \projlaak_{m})  \generlaak_\infty( f \circ \projlaak_{n})\, d\measlaak\quad \mbox{(by Lemma \ref{l:consistent}(a,c))}.
		\end{align*}
		In other words, we conclude that $\generlaak_\infty$ is a symmetric operator on  $L^2(\laak(\glue,\branch),\measlaak)$.
		
		It remains to show that $ D(\generlaak_\infty)$ is dense in $L^2(\measlaak)$. Let us denote the closure of $ D(\generlaak_\infty)$  in $L^2(\measlaak)$ as $\overline{D(\generlaak_\infty)}$. By Lemma \ref{l:consistent}(a,b) and the density of $D(\generlaak_n)$ in $L^2(\mathsf{m}_{\laak(\glue_n,\branch)})$ for all $n \in \mathbb{N}$, we have 
		\begin{equation} \label{e:fred1}
			\bigcup_{n \in \mathbb{N}} \{ f \circ \projlaak_n : f \in \contfunc_0(\laak(\glue_n,\branch))\}  \subset \bigcup_{n \in \mathbb{N}} \{ f \circ \projlaak_n : f \in L^2(\mathsf{m}_{\laak(\glue_n,\branch)})\} \subset	\overline{D(\generlaak_\infty)}.
		\end{equation}
		By \cite[Proposition 3.3.49]{HKST}, it suffices to show that  
		\begin{equation} \label{e:fred2}
			\contfunc_0(\laak(\glue,\branch))	\subset	\overline{D(\generlaak_\infty)}.
		\end{equation}
		For any $f \in \contfunc_0(\laak(\glue,\branch))$, then the sequence $f_n:= (A_{\infty,n}(f))\circ \projlaak_n \in \contfunc_0(\laak(\glue,\branch)) \cap \overline{D(\generlaak_\infty)}$ by \eqref{e:fred1} for all $n \in \mathbb{N}$. Since $\cup_{n \in \mathbb{N}} \supp(f_n)$ is bounded and $\sup_{n \in \mathbb{N}} \sup_{p \in \laak(\glue,\branch)} \abs{f_n(p)} \le \sup_{p \in \laak(\glue,\branch)} \abs{f(p)}$ by the dominated convergence theorem, we have $\lim_{n \to \infty} f_n = f$ in $L^2(\measlaak)$. This completes the proof of \eqref{e:fred2} and hence that $\overline{D(\generlaak_\infty)}= L^2(\measlaak)$.
	\end{proof}
	
	Next, we use Friedrichs extension theorem to construct a generator and the Dirichlet form on $L^2(\laak(\glue,\branch), \measlaak)$.  This completes the construction of diffusion on the limiting Laakso-type space and can be considered as an analytic counterpart of the probabilistic approach of Barlow and Evans \cite{BE}.
	\begin{prop} \label{p:friedrich}
		The symmetric bilinear form $(f,g) \mapsto \int_{\laak(\glue,\branch)} f \generlaak_\infty(g) \, d\measlaak$ on $D(\generlaak_\infty) \times D(\generlaak_\infty)$ is closable. Its closure $\dflaak: \domlaak \times \domlaak \to \mathbb{R}$ is a regular, strongly local, Dirichlet form on $L^2(\laak(\glue,\branch), \measlaak)$ such that    $	\bigcup_{n \in \mathbb{N}} \{f \circ \projlaak_{n} : f\in \mathcal{F}^{\laak_n}\cap \contfunc_0(\laak(\glue_n,\branch)) \}$ is a core.
	\end{prop}
	\begin{proof}
		Note that by Lemma \ref{l:friedprep}, $ \generlaak_\infty$ is a non-negative, symmetric operator. Therefore by the Friedrichs extension theorem \cite[Theorem X.23]{RS}, the corresponding quadratic form $(f,g) \mapsto \int_{\laak(\glue,\branch)} f \generlaak_\infty(g) \, d\measlaak$ on $D(\generlaak_\infty) \times D(\generlaak_\infty)$ is closable. By \cite[Theorem 3.1.1]{FOT}, its closure $(\dflaak,\domlaak)$ is a Dirichlet form on $L^2(\measlaak)$.
		
		Since $D(\generlaak_n)$ is $\mathcal{E}^{\laak_n}_1$-dense in $\mathcal{F}^{\laak_n}$, by Lemma \ref{l:consistent}(a,b), we have 
		\begin{equation} \label{e:fried1}
			\bigcup_{n \in \mathbb{N}} \{f \circ \projlaak_{n} : f\in D(\generlaak_n) \} \subset \domlaak, \quad \dflaak(f \circ \projlaak_{k},g \circ \projlaak_{k}) = \mathcal{E}^{\laak_k}(f,g)
		\end{equation}
		for all $k \in \mathbb{N}, f,g \in \mathcal{F}^{\laak_k}$. The strong locality of $(\dflaak,\domlaak)$ follows from that of $(\mathcal{E}^{\laak_n},\mathcal{F}^{\laak_n})$ for all $n \in \mathbb{N}$, \cite[Theorem 3.1.2, Exercise 3.1.1]{FOT} along with observing that $	\bigcup_{n \in \mathbb{N}} \{f \circ \projlaak_{n} : f\in \mathcal{F}^{\laak_n} \cap \contfunc_0(\laak(\glue_n,\branch)) \}$ is an $\dflaak_1$-dense subset of $\domlaak$ due to the closability of the bilinear form $(f,g) \mapsto \int_{\laak(\glue,\branch)} f \generlaak_\infty(g) \, d\measlaak$ on $D(\generlaak_\infty)$.
		
		Since $	\bigcup_{n \in \mathbb{N}} \{f \circ \projlaak_{n} : f\in \mathcal{F}^{\laak_n} \cap \contfunc_0(\laak(\glue_n,\branch)) \}$ is an algebra that separates points, by the Stone-Weierstrass theorem $	\bigcup_{n \in \mathbb{N}} \{f \circ \projlaak_{n} : f\in \mathcal{F}^{\laak_n} \cap \contfunc_0(\laak(\glue_n,\branch)) \}$ is dense in $\contfunc_\infty(\laak(\glue,\branch))$ with respect to the supremum norm. This completes the proof of the regularity of $(\dflaak,\domlaak)$ and that $	\bigcup_{n \in \mathbb{N}} \{f \circ \projlaak_{n} : f\in \mathcal{F}^{\laak_n} \cap \contfunc_0(\laak(\glue_n,\branch)) \}$ is a core for $(\dflaak,\domlaak)$.  
	\end{proof}
	
	We describe the energy measure of functions in the core described in Proposition \ref{p:friedrich} in terms of the energy measure $\emtree$ for the diffusion on the tree $\tree(\branch)$. Let $\emlaak$ denote the \emph{energy measure corresponding to the MMD space $(\laak(\glue, \branch), \mathsf{d}_{\laak(\glue, \branch)}, \mathsf{m}_{\laak(\glue, \branch)}, \dflaak,\domlaak)$}. We have the following relation between $\emlaak$ and $\emtree$. 
	\begin{lemma} \label{l:emeaslaak}
		For any $n \in \mathbb{N}, f \in \mathcal{F}^{\laak_n} $, we have 
		\begin{equation} \label{e:emlaak}
			\emlaak(f \circ \projlaak_{n}, f \circ \projlaak_{n})= \quolaak_*\left( \sum_{ \mathbf{u} \in \ultra(\glue_n)} \restr{\mathsf{m}_{\ultra(\glue)}}{B_{\ultra(\glue)}(\mathbf{u},2^{-n})} \times  \emtree(f \circ I_{\mathbf{u}},f \circ I_{\mathbf{u}}) \right).
		\end{equation}
	\end{lemma}
	\begin{proof}
		For $m \ge n, m, n\in \mathbb{Z}, f \in \mathcal{F}^{\laak_n}, g \in \mathcal{F}^{\laak_m}$ we have
		\begin{align*}
			\MoveEqLeft{\int (g\circ \projlaak_{m}) \, d\emlaak(f \circ \projlaak_n,f \circ \projlaak_n) } \\
			&= \dflaak(f \circ \projlaak_n,(f \circ \projlaak_n) (g \circ \projlaak_m))-\frac{1}{2}\dflaak((f \circ \projlaak_n)^{2},g  \circ \projlaak_{m}) \\
			&= \mathcal{E}^{\laak_m}(f \circ \projlaak_{m,n},(f \circ \projlaak_{m,n}) g)-\frac{1}{2}\mathcal{E}^{\laak_m}((f \circ \projlaak_{m,n})^{2},g   ) \\
			&= \int g \, d\Gamma^{\laak_m}(f \circ \projlaak_{m,n},f \circ \projlaak_{m,n})  \\
			&= \sum_{\mathbf{u} \in \ultra(\glue_m)} \mathsf{m}_{\ultra(\glue_m)}(\{\mathbf{u}\}) \int_{\tree(\branch)} (g \circ I_{\mathbf{u}})  \, d \emtree(f \circ \pi_{m,n} \circ I_{\mathbf{u}},f \circ \pi_{m,n} \circ I_{\mathbf{u}}) \quad \mbox{(by \eqref{e:emlaakapprox})} \\
			&= \sum_{\mathbf{u} \in \ultra(\glue_m)} \mathsf{m}_{\ultra(\glue)}\left(B_{\ultra(\glue)}(\mathbf{u},2^{-m})\right) \int_{\tree(\branch)} (g \circ I_{\mathbf{u}}) \, d \emtree(f \circ  I_{\projultra_{n}(\mathbf{u})},f \circ  I_{\projultra_{n}(\mathbf{u})}) \\
			&= \int (g \circ \projlaak_{m}) \, d\quolaak_*\left( \sum_{ \mathbf{u} \in \ultra(\glue_m)}  \restr{\mathsf{m}_{\ultra(\glue)}}{B_{\ultra(\glue)}(\mathbf{u},2^{-m})} \times \emtree(f \circ I_{\projultra_{n}(\mathbf{u})},f \circ I_{\projultra_{n}(\mathbf{u})}) \right)\\
			&= \int (g \circ \projlaak_{m})\,  d\quolaak_*\left( \sum_{ \mathbf{u} \in \ultra(\glue_n)}  \restr{\mathsf{m}_{\ultra(\glue)}}{B_{\ultra(\glue)}(\mathbf{u},2^{-n})} \times \emtree(f \circ I_{\mathbf{u}},f \circ I_{\mathbf{u}}) \right).
		\end{align*}
		Since $\cup_{m \ge n} \{g \circ \projlaak_{m}: g \in \mathcal{F}^{\laak_m}\}$ is a core for $(\dflaak,\domlaak)$ we obtain \eqref{e:emlaak}.
	\end{proof}
	
	The expression \eqref{e:emlaak} can be rewritten using the gradient introduced in \eqref{e:defgradlaak}. Let us denote the gradient in \eqref{e:defgradlaak} for $f \in \mathcal{F}^{\laak_n}$ as $\grad^{\laak_n} f: \laak(\glue_n,\branch) \to \mathbb{R}$. In this case, we \emph{define} $\grad^\laak(f \circ \projlaak_n): \laak(\glue,\branch) \to \mathbb{R}$ as 
	\begin{equation} \label{e:defgradlaaklim}
		\grad^\laak(f \circ \projlaak_n):= (\grad^{\laak_n}f)  \circ \projlaak_n \quad \mbox{for all $n \in \mathbb{N}, f \in \mathcal{F}^{\laak_n}$.}
	\end{equation}
	It is easy to check that $\grad^\laak(f \circ \projlaak_n)$ is well-defined $\lenlaak$-almost everywhere. In other words, if $f_n \circ \projlaak_n = f_m \circ \projlaak_m$ for some $m,n \in \mathbb{N}, f_n \in \mathcal{F}^{\laak_n}, \mathcal{F}^{\laak_m}$, then $\grad^\laak(f_m \circ \projlaak_m)= \grad^\laak(f_n \circ \projlaak_n)$ $\lenlaak$-almost everywhere. Furthermore, by \eqref{e:emlaakapprox} and Lemma \ref{l:emeaslaak}, we have 
	\begin{equation} \label{e:emlaakint}
		\emlaak(f \circ \projlaak_n,f \circ \projlaak_n)(A)= \int_A \abs{\grad^\laak(f \circ \projlaak_n)}^2 \, d\lenlaak, \quad \mbox{for all Borel sets $A \subset \laak(\glue,\branch)$.}
	\end{equation}
	
	\section{Heat kernel estimates for the Laakso-type space} \label{s:hke}

	The main result of this section is that the diffusion on Laakso-type space satisfies sub-Gaussian heat kernel bounds (Theorem \ref{t:hkediff}). This is established by deriving a Poincar\'e inequality in \textsection \ref{ss:pi} and a cutoff energy inequality in  \textsection \ref{ss:cs}. 
	This sub-Gaussian estimate for Laakso-type space is enough to prove Theorem \ref{t:mainsuf} by choosing suitable branching and gluing functions using Lemma \ref{l:elementary}. 
	We discuss extensions of the main results to the discrete time setting in \textsection \ref{ss:rw}. The martingale dimension of the Laakso-type MMD space is computed in \textsection \ref{ss:martingale}. Finally, we state some questions and conjectures related to this work in \textsection \ref{ss:conclude}.
	
	\subsection{Poincar\'e inequality} \label{ss:pi}
	Next, we obtain Poincar\'e inequality for the diffusion on Laakso-type space.
	As mentioned in the introduction, the proof is an adaptation the approach used in \cite{Laa} using the pencil of curves constructed in Lemma \ref{l:pencil}.
	\begin{prop} \label{p:pi}
		Let $\branch,\glue:\mathbb{Z} \to \mathbb{N}$ satisfy \eqref{e:defbranch} and \eqref{e:defglue}.
		Then the Laakso-type MMD space $(\laak(\glue,\branch),\metlaak,\measlaak,\dflaak,\domlaak)$ satisfies the Poincar\'e inequality \hyperlink{pi}{$\on{PI(\sttree)}$}.
	\end{prop}
	\begin{proof}
		Note by Proposition \ref{p:friedrich} that	$\bigcup_{n \in \mathbb{N}} \{f \circ \projlaak_{n} : f\in \mathcal{F}^{\laak_n}\cap \contfunc_0(\laak(\glue_n,\branch)) \}$ is a core for the Dirichlet form $(\dflaak,\domlaak)$ on $L^2(\measlaak)$. Hence in order to show \hyperlink{pi}{$\on{PI(\sttree)}$} it suffices to obtain the Poincar\'e inequality for functions in this core.
		
		To this end consider   $f = f_n \circ \projlaak_n$, where $f_n \in \mathcal{F}^{\laak_n} \cap \contfunc_0(\laak(\glue_n,\branch) \cap , n \in \mathbb{N}$.
		For any curve parameterized by arc length $\gamma:[0,D] \to \laak(\glue,\branch)$ joining $x,y \in \laak(\glue,\branch)$, by the fundamental theorem of calculus and Cauchy-Schwarz inequality, we have 
		\begin{equation} \label{e:pi1}
			\abs{f(x)-f(y)} \le \int_{0}^D \abs{\grad^\laak f(\gamma(t))}\,dt \le D^{1/2} \left(\int_0^D \abs{\grad^\laak f(\gamma(t))}^2\,dt\right)^{1/2}. 
		\end{equation}
		By choosing a random geodesic between $x$ and $y$ in \eqref{e:pi1} and averaging using  Lemma \ref{l:pencil}, \eqref{e:8qc},  there exists $C_1 \in (1,\infty)$ such that
		\begin{align} \label{e:pi2}
			\lefteqn{\abs{f(x)-f(y)}^2}  \nonumber\\
			&\le 8 C_1 \metlaak(x,y) \int_{B_{\laak(\glue,\branch)}(x,8\metlaak(x,y))} \frac{ \abs{\grad^\laak f(z)}^2}{V_{\glue}(\metlaak(x,z) \wedge \metlaak(y,z))} \, d\lenlaak(z) \nonumber \\
			&\le 8 C_1 \metlaak(x,y) \int_{B_{\laak(\glue,\branch)}(x,8\metlaak(x,y))} \frac{d\emlaak(f,f)}{V_{\glue}(\metlaak(x,z) \wedge \metlaak(y,z))} \quad \mbox{(by \eqref{e:emlaakint})} \nonumber \\
			&\le  8 C_1 \metlaak(x,y) \int_{B_{\laak(\glue,\branch)}(x,8\metlaak(x,y))} \left(V_{\glue}(\metlaak(x,z))^{-1}+V_{\glue}(\metlaak(y,z))^{-1}\right) d\emlaak(f,f)  
		\end{align}
		for all $x,y \in \laak(\glue,\branch), n \in \mathbb{N}, f_n \in \mathcal{F}^{\laak_n}$ and $f = f_n \circ \projlaak_{n}$.
		
		Let $x_0 \in \laak(\glue,\branch), r>0$ be arbitrary and let us denote by $B=B_{\laak(\glue,\branch)}(x_0,r), AB= B_{\laak(\glue,\branch)}(x_0,Ar)$ for any $A >0$. Note that if $x,y \in B$, we have 
		\begin{equation} \label{e:pi3}
			\metlaak(x,y) <2r, \quad B_{\laak(\glue,\branch)}(x,8\metlaak(x,y)) \subset  B_{\laak(\glue,\branch)}(x_0,17r)=17B.
		\end{equation}
		Then   by \eqref{e:pi2}, \eqref{e:pi3} and Fubini's theorem, we obtain
		\begin{align} \label{e:pi4}
			\lefteqn{	\int_{B} \abs{f(x)-f_{B}}^2 \,d \measlaak} \nonumber\\
			&= \frac{1}{2\measlaak(B)} \int_{B}\int_{B} (f(x)-f(y))^2\, d\measlaak(x)\,d\measlaak(y) \nonumber\\
			&\le \frac{2^4C_1r}{\measlaak(B)} \int_B \int_B \int_{17B}    V_{\glue}(\metlaak(x,z))^{-1}  \,\emlaak(f,f)(dz) \,\measlaak(dy)\,\measlaak(dx) \nonumber \\
			&\le 2^4 C_1 r  \int_{17B} \int_B    V_{\glue}(\metlaak(x,z))^{-1}  \,\measlaak(dx)  \,\emlaak(f,f)(dz)  
		\end{align}
		The inner integral in \eqref{e:pi4}  can be estimated for any $x \in B, z \in 17B$ as
		\begin{align} \label{e:pi5}
			\lefteqn{\int_B    V_{\glue}(\metlaak(x,z))^{-1}  \,\measlaak(dx) } \nonumber \\
			&\le \int_{B_{\laak(\glue,\branch)}(z,18r)}    V_{\glue}(\metlaak(x,z))^{-1}  \,\measlaak(dx) \nonumber \\
			&\le \sum_{n=0}^\infty \int_{B_{\laak(\glue,\branch)}(z, 2^{-n}18r) \setminus B_{\laak(\glue,\branch)}(z, 2^{-n-1}18r)}    V_{\glue}(\metlaak(x,z))^{-1}  \,\measlaak(dx) \nonumber\\
			&\lesssim \sum_{n=0}^\infty V_{\glue}(2^{-n}r)^{-1} \measlaak(B_{\laak(\glue,\branch)}(z, 2^{-n}18r))  \nonumber \\
			&\lesssim \sum_{n=0}^\infty  V_{\branch}(2^{-n}r) \quad  \mbox{(by Corollary \ref{c:measlaak}, \eqref{e:defvol} and \eqref{e:measultraformula})}  \nonumber \\
			&\lesssim V_{\branch}(r) \quad \mbox{(by $\inf_{k \in \mathbb{Z}} \branch(k) \ge 2$, \eqref{e:defvol} and \eqref{e:measultraformula}})
		\end{align}
		Combining \eqref{e:pi4}, \eqref{e:pi5} and \eqref{e:sttree}, we obtain the desired Poincar\'e inequality.
	\end{proof}

	\subsection{Cutoff energy inequality} \label{ss:cs}
	Next, we will obtain the cutoff energy inequality \hyperlink{cs}{$\on{CS(\sttree)}$} for the Laakso-type space $(\laak(\glue,\branch),\metlaak,\measlaak,\dflaak,\domlaak)$. 
	To this, end we recall the simpler sufficient condition introduced in \cite[Defintion 6.1]{Mur24}.
	\begin{definition} \label{d:css}
		We say that $(X,d,m,\mathcal{E},\mathcal{F})$ satisfies the \textbf{simplified cutoff energy inequality} 
		\hypertarget{css}{$\operatorname{CSS}(\Psi)$},
		if there exist  $C_S>0,A_1, A_2, C_1>1$ such that the following holds: for all $x \in X$ and $0<R< \diam(X,d)/A_2$,
		there exists a cutoff function $\phi \in \mathcal{F}$ for $B(x,R) \subset B(x,A_1R)$ such that for all $f \in \mathcal{F}$,
		\begin{equation}\tag*{$\operatorname{CSS}(\Psi)$}
			\int_{B(x,A_1R)} \wt{f}^2\, d\Gamma(\phi,\phi) \le C_1 \int_{B(x,A_1R) }  d\Gamma(f,f)
			+ \frac{C_1}{\Psi(R)} \int_{B(x,A_1R) } f^2\,dm;
		\end{equation}
		where $\wt{f}$ is a quasi-continuous version of $f \in \mathcal{F}$.
	\end{definition}
	In the following proposition, we obtain the cutoff energy inequality for the diffusion on Laakso-type space.  
	The main idea behind the proof is to \emph{lift} suitable cutoff functions on the tree $\tree(\branch)$ to construct cutoff functions on the Laakso-type space $\laak(\glue,\branch)$.
	\begin{prop} \label{p:cs}
		Let $\branch,\glue:\mathbb{Z} \to \mathbb{N}$ satisfy \eqref{e:defbranch} and \eqref{e:defglue}.
		Then the Laakso-type MMD space $(\laak(\glue,\branch),\metlaak,\measlaak,\dflaak,\domlaak)$ satisfies the cutoff energy inequality \hyperlink{cs}{$\on{CS(\sttree)}$}.
	\end{prop}
	\begin{proof}
		Let us note from \cite[Theorem 1.2]{GHL} and Theorem \ref{t:hketree} that the MMD space $(\tree(\branch),\metrictree,\meastree, \dftree, \domtree)$  satisfies \hyperlink{cs}{$\on{CS(\sttree)}$}  and hence \hyperlink{css}{$\on{CSS(\sttree)}$}. By \cite[Lemma 6.2]{Mur24}, it suffices to verify \hyperlink{css}{$\on{CSS(\sttree)}$} for $(\laak(\glue,\branch),\metlaak,\measlaak,\dflaak,\domlaak)$
		
		To this end, let $[(\mathbf{v},\mathbf{t})]_{\laak} \in \laak(\glue,\branch),r >0$ be arbitrary.
		Let us define the \emph{cylinder} set
		\begin{equation*}
			C_{\laak(\glue,\branch)}([(\mathbf{v},\mathbf{t})]_{\laak},r) = \quolaak \left(  \Biggl\{  (\mathbf{u},\mathbf{s}) \in \mathcal{P}(\glue,\branch)\Biggm|
			\begin{minipage}{200pt}
				$\metrictree(\mathbf{s},\mathbf{t})<r$, $\mathbf{u}\in \ultra(\glue)$  satisfies $\mathbf{u}(k)=\mathbf{v}(k)$  for all $k \in \mathbb{Z}$ such that $\worm_k \cap B_{\tree(\branch)}(\mathbf{t},r)=\emptyset$
			\end{minipage}
			\Biggr\}\right).
		\end{equation*}
		By Lemma \ref{l:ballbox}, we can compare balls and cylinders as
		\begin{equation}\label{e:cs1}
			{C_{\laak(\glue,\branch)}([(\mathbf{v},\mathbf{t})]_{\laak},r/4) } \subset {B_{{\laak(\glue,\branch)}}([(\mathbf{v},\mathbf{t})]_{R_{\laak}},r)} \subset {C_{\laak(\glue,\branch)}([(\mathbf{v},\mathbf{t})]_{\laak},r) }
		\end{equation}
		for any  $[(\mathbf{v},\mathbf{t})]_{\laak} \in \laak(\glue,\branch),r >0$.
		
		Let $A_1>1, C_S>0$ be the constant associated with \hyperlink{css}{$\on{CSS(\sttree)}$} for the MMD space $(\tree(\branch),\metrictree,\meastree, \dftree, \domtree)$. Let $\phi^\tree_{\mathbf{t},r}$ is the cutoff function for $B_{\tree(\branch)}(\mathbf{t},r) \subset B_{\tree(\branch)}(\mathbf{t},A_1r)$   such that 
		\begin{equation} \label{e:cs2}
			\int_{B_{\tree(\branch)}(\mathbf{t},A_1r)}  {f}^2\, d\emtree(\phi^\tree_{\mathbf{t},r},\phi^\tree_{\mathbf{t},r}) \le C_S \int_{B_{\tree(\branch)}(\mathbf{t},A_1r) }  d\emtree(f,f)
			+ \frac{C_S}{\sttree(r)} \int_{B_{\tree(\branch)}(\mathbf{t},A_1r) } f^2\,d\meastree
		\end{equation}
		for all $f \in \domtree$.
		Define the function $\phi^\laak_{\mathbf{v},\mathbf{t},r}: \laak(\glue,\branch) \to \mathbb{R}$ defined by
		\[
		\phi^\laak_{\mathbf{v},\mathbf{t},r} \circ I_{\mathbf{u}}(\cdot) = \begin{cases}
			\phi^\tree_{\mathbf{t},r} (\cdot) & \mbox{if $ I_{\mathbf{u}}(\tree(\branch)) \cap  C_{\laak(\glue,\branch)}([(\mathbf{v},\mathbf{t})]_{\laak},A_1 r) \neq \emptyset$,} \\
			0 & \mbox{if $ I_{\mathbf{u}}(\tree(\branch)) \cap  C_{\laak(\glue,\branch)}([(\mathbf{v},\mathbf{t})]_{\laak}, A_1 r) = \emptyset$.}
		\end{cases}
		\]
		It is easy to see that $\phi^\laak_{\mathbf{v},\mathbf{t},r} \in \contfunc(\laak(\glue,\branch) ) \cap \domlaak$ and 
		by \eqref{e:cs1} we have that $\phi^\laak_{\mathbf{v},\mathbf{t},r}$ is a cutoff function for $B_{{\laak(\glue,\branch)}}([(\mathbf{v},\mathbf{t})]_{R_{\laak}},r) \subset B_{{\laak(\glue,\branch)}}([(\mathbf{v},\mathbf{t})]_{R_{\laak}},4A_1 r)$.
		Since there exists $m \in \mathbb{N}$ large enough such that $\worm_{k} \cap B_{\tree(\branch)}(\mathbf{t},A_1R) \neq \emptyset$ for all $k \le -m$. 
		Therefore $\phi^\laak_{\mathbf{v},\mathbf{t},r} = \phi_n \circ \projlaak_{n}$ for all $n \ge m$, where $\phi_n \in \mathcal{F}^{\laak_n}$.

		By Proposition \ref{p:friedrich},  it suffices  \hyperlink{css}{$\on{CSS(\sttree)}$} consider functions of the form $f=f_n \circ \projlaak_n$ for $n \ge m, f_n \in \mathcal{F}^{\laak_n}\cap \contfunc_0(\laak(\glue_n,\branch))$.
		In this case, we have the estimate 
		\begin{align} \label{e:cs3}
			\lefteqn{ \int_{B_{\laak(\glue,\branch)}([(\mathbf{v},\mathbf{t})]_{R_{\laak}},4A_1r)} f^2\, d\emlaak (\phi^\laak_{\mathbf{v},\mathbf{t},r},  \phi^\laak_{\mathbf{v},\mathbf{t},r}) }  \nonumber \\
			&= \sum_{\mathbf{u} \in \ultra(\glue_n,\branch)}  \mathsf{m}_{\ultra(\glue_n)}(\{\mathbf{u}\}) \int_{B_{\tree(\branch)}( \mathbf{t}, r)}(f_n \circ I_{\mathbf{u}})^2 \,d\emtree(\phi_n\circ I_{\mathbf{u}},\phi_n\circ I_{\mathbf{u}} )   \quad \mbox{(by \eqref{e:emlaak})}\nonumber \\
			&\le C_S\sum_{\substack{\mathbf{u} \in \ultra(\glue_n,\branch),\\ \phi_n\circ I_{\mathbf{u}} \neq 0 }}  \mathsf{m}_{\ultra(\glue_n)}(\{\mathbf{u}\}) \left(   \int_{B_{\tree(\branch)}(\mathbf{t},A_1r) }  d\emtree(f_n\circ I_{\mathbf{u}}, f_n\circ I_{\mathbf{u}}) \right. \nonumber \\
			&\qquad \qquad
			\left. + \frac{1}{\sttree(r)} \int_{B_{\tree(\branch)}(\mathbf{t},A_1r) } (f_n\circ I_{\mathbf{u}})^2\,d\meastree\right) \nonumber \\
			&\le C_S \int_{B_{\laak(\glue,\branch)}([(\mathbf{v},\mathbf{t})]_{R_{\laak}},4A_1r)}  d\emlaak(f,f)
			+ \frac{C_S}{\sttree(r)} \int_{B_{\laak(\glue,\branch)}([(\mathbf{v},\mathbf{t})]_{R_{\laak}},4A_1r)} f^2\,d\measlaak,
		\end{align} 
		where the last line above follows from  \eqref{e:emlaak}, \eqref{e:cs1}, and \eqref{e:measlaak}. By \eqref{e:cs3} and approximating an arbitrary function using functions in the core, we complete the proof of \hyperlink{css}{$\on{CSS(\sttree)}$}  for $(\laak(\glue,\branch),\metlaak,\measlaak,\dflaak,\domlaak)$.
	\end{proof}

	\subsection{Sub-Gaussian bounds on the heat kernel}
	We complete the proof of sub-Gaussian heat kernel bounds.
	\begin{theorem} \label{t:hkediff}
		Let $\branch,\glue:\mathbb{Z} \to \mathbb{N}$ satisfy \eqref{e:defbranch} and \eqref{e:defglue}.
		Then the Laakso-type MMD space $(\laak(\glue,\branch),\metlaak,\measlaak,\dflaak,\domlaak)$ satisfies the full sub-Gaussian heat kernel estimate 
		\hyperlink{hke}{$\on{HKE_f(\sttree)}$}. %There exists $C_1>1$ such that volume estimate 
		%	\begin{equation} \label{e:vollaak}
			% C_1^{-1} V_{\glue}(r) V_{\branch}(r) \le	\measlaak \left(B_{\metlaak}(x,r)\right) \le C_1 V_{\glue}(r) V_{\branch}(r) \quad \mbox{for all $x \in \laak(\glue,\branch), r>0$.}
			%	\end{equation}
		
	\end{theorem}
	\begin{proof}
		By Propositions \ref{p:pi} and \ref{p:cs} along with \cite[Theorem 1.2]{GHL}, we obtain the   sub-Gaussian heat kernel estimate 
		\hyperlink{hke}{$\on{HKE(\sttree)}$}, where the  doubling and reverse doubling properties required to apply \cite[Theorem 1.2]{GHL} follows from  Corollary \ref{c:measlaak}. We upgrade this to the desired full heat kernel estimate by chain condition which is a consequence of the property that the metric is quasiconvex as shown in Proposition \ref{p:laak}(c).
	\end{proof}
	
	\begin{remark} \label{r:hkelaak}
		\begin{enumerate}[(a)]
			\item A careful examination of the proof above  shows that the constants associated in the conclusion of Theorem \ref{t:hkediff} depend  only on $\sup_{k \in \mathbb{Z}} \branch(k)$ and $\sup_{k \in \mathbb{Z}} \glue(k)$.
			\item The general results \cite[Theorems 7.2.1 and 4.5.3]{FOT}
			from the theory of regular symmetric Dirichlet forms guarantee the existence
			of an associated diffusion which is unique only up to
			a properly exceptional set of starting points. By the conclusion of Theorem \ref{t:hkediff}, the Laakso-type MMD space admits a unique continuous heat kernel $p=p_{t}(x,y): (0,\infty) \times \laak(\glue,\branch) \times \laak(\glue,\branch) \to [0,\infty)$   and gives a Markovian transition function
			with the Feller and strong Feller properties, which allow us to define canonically
			an associated diffusion starting from \emph{every} $x \in \laak(\glue,\branch)$.   In this case, for every starting point  $x \in \laak(\glue,\branch), t>0$, the law of the diffusion at time $t$ started at $x$ is $p_t(x.y) \,m(dy)$. These statements follow from \cite[Theorem 3.1]{BGK}, \cite[Proposition 3.2]{Lie},  \cite[Proposition 2.18]{KM24+}.
		\end{enumerate} 
	\end{remark}

	We can conclude the proof of Theorem \ref{t:mainsuf} using Theorem \ref{t:hkediff} and the following elementary lemma whose proof is omitted. 
	The proof involves first constructing $\branch$ using the bound on $\Psi$ and then using the estimate on $V$ to construct $\glue$.
	\begin{lemma} \label{l:elementary}
		Let $V, \Psi:[0,\infty) \to [0,\infty)$ be doubling functions and let $C_0 \in (1,\infty)$ be such that 
		\[
		C_0^{-1} \frac{R^2}{r^2} \le \frac{\Psi(R)}{\Psi(r)} \le C_0 \frac{RV(R)}{rV(r)}, \quad \mbox{for all $0<r \le R$.}
		\] Then there exists $\branch,\glue: \mathbb{Z} \to \mathbb{Z}$ and $C_1 \in (1,\infty)$ such that 
		\[
		1 \le \inf_{k \in \mathbb{Z}} \glue(k) \le \sup_{k \in \mathbb{Z}} \glue(k) < \infty, \quad	2 \le \inf_{k \in \mathbb{Z}} \branch(k) \le \sup_{k \in \mathbb{Z}} \branch(k) < \infty
		\]
		and 
		\[
		C_1^{-1} V_\glue(r) V_\branch(r) \le V(r) \le 	C_1  V_\glue(r) V_\branch(r), \quad C_1^{-1} \sttree(r) \le \Psi(r) \le C_1 \sttree(r) \quad \mbox{for all $r>0$.}
		\]
	\end{lemma}
	We now conclude the proof of sufficiency of \eqref{e:necdiff}.
	\begin{proof}[Proof of Theorem \ref{t:mainsuf}]
		Let $\branch,\glue:\mathbb{Z} \to \mathbb{Z}$ be functions that satisfy the conclusion of Lemma \ref{l:elementary}. Then the desired volume growth and heat kernel estimates for the Laakso-type MMD space $(\laak(\glue,\branch),\metlaak,\measlaak,\dflaak,\domlaak)$  follows from Corollary \ref{c:measlaak} and Theorem \ref{t:hkediff} respectively.
	\end{proof}
	The following two sided estimates on Green function under the conclusion of Theorem \ref{t:mainsuf} follows easily.
	\begin{prop} \label{p:greendiff}
		Let $V, \Psi:[0,\infty) \to [0,\infty)$ be doubling functions and let $C_0 \in (1,\infty)$ be such that 
		\[
		C_0^{-1} \frac{R^2}{r^2} \le \frac{\Psi(R)}{\Psi(r)} \le C_0 \frac{RV(R)}{rV(r)}, \quad \mbox{for all $0<r \le R$.}
		\]
		Let $(X,d,m,\mathcal{E},\mathcal{F})$ be an  unbounded MMD space that satisfies the full sub-Gaussian kernel estimate 	\hyperlink{hkef}{$\on{HKE_f(\Psi)}$} and suppose there exists $C_1 \in (1,\infty)$ such that the volume of balls satisfy the estimate \eqref{e:volest}. Let $p_t(\cdot,\cdot)$ denote the continuous heat kernel (cf. Remark \ref{r:hkelaak}(b)). Then there exists $C_2>1$ such that 
		\[
		C_2^{-1} \int_{d(x,y)}^\infty \frac{\Psi(s)}{sV(s)} \,ds \le \int_0^\infty p_t(x,y) \,dt \le C_2 \int_{d(x,y)}^\infty \frac{\Psi(s)}{sV(s)} \,ds
		\]
		In particular, the diffusion is transient if and only if 
		\[
		\int_{1}^\infty \frac{\Psi(s)}{sV(s)} \,ds < \infty.
		\]
	\end{prop}
	\begin{proof}
		It follows from integrating the heat kernel bounds. 
		Alternately, this follows from applying \cite[Lemma 5.10]{BCM} to estimate the green function between $x$ and $y$ for the diffusion killed upon exiting a balls of radius $2^n d(x,y)$ in terms of capacity and then using \cite[Lemma 5.12]{BCM}, \hyperlink{cap}{$\operatorname{cap}(\Psi)$} to estimate the capacity. Finally letting $n \to \infty$ implies the desired result. 
	\end{proof}
	In the following definition, we introduce the MMD space corresponding to reflected diffusions on balls of the Laakso-type space.
	\begin{definition}
		For $k \in \mathbb{Z}$, we define 
		\begin{align*}
			\laak(\glue,\branch)^{(k)}&= \overline{B_{\laak(\glue,\branch)}(p_\laak,2^k)}, \\
			\metlaak^{(k)}&= \restr{\metlaak}{	\laak(\glue,\branch)^{(k)} \times 	\laak(\glue,\branch)^{(k)}}, \\
			\measlaak^{(k)}&= \restr{\measlaak}{	\laak(\glue,\branch)^{(k)}},\\
			\mathcal{F}^{\laak,(k)} &= \{ \restr{f}{\laak(\glue,\branch)^{(k)}}: f \in \domlaak\}, \\
			\mathcal{E}^{\laak,(k)}(f,g)&=	\emlaak(\overline{f},\overline{g})(B_{\laak(\glue,\branch)}(p_\laak,2^k)),
		\end{align*}
		for all $f,g \in \mathcal{F}^{\laak,(k)}$, where $\overline{f},\overline{g} \in \domlaak$ are such that $\restr{\overline{f}}{	\laak(\glue,\branch)^{(k)}}= f$, $\restr{\overline{g}}{	\laak(\glue,\branch)^{(k)}}= g$ and $\emlaak$ is the energy measure corresponding to the MMD space. 
	\end{definition}
	
	The following theorem states properties of the bilinear form $(\mathcal{E}^{\laak,(k)},	\mathcal{F}^{\laak,(k)})$   and the metric measure space $(\laak(\glue,\branch)^{(k)}, 	\metlaak^{(k)}, 	\measlaak^{(k)})$ defined above.
	\begin{theorem} \label{t:hkeseq}
		For each $k \in \mathbb{Z}$,  the metric measure space $(\laak(\glue,\branch)^{(k)}, 	\metlaak^{(k)}, 	\measlaak^{(k)})$ and   the bilinear form $(\mathcal{E}^{\laak,(k)},	\mathcal{F}^{\laak,(k)})$  define a strongly local, regular, Dirichlet form on $L^2(\measlaak^{(k)})$.  Furthermore, the family of MMD spaces  $(\laak(\glue,\branch)^{(k)}, 	\metlaak^{(k)}, 	\measlaak^{(k)}, \mathcal{E}^{\laak,(k)},	\mathcal{F}^{\laak,(k)})$ satisfy the full sub-Gaussian heat kernel estimates 	\hyperlink{hke}{$\on{HKE_f(\sttree)}$} uniformly in $k$; that is,  the constants involved in the estimate 	\hyperlink{hke}{$\on{HKE_f(\sttree)}$} are independent of $k \in \mathbb{Z}$.
	\end{theorem}
	\begin{proof}
		The sub-Gaussian estimate \hyperlink{hke}{$\on{HKE(\sttree)}$} is consequence of \cite[Theorems 2.7 and 2.8]{Mur24} along with the fact that $B_{\laak(\glue,\branch)}(p_\laak,2^k)$ is an uniform domain for each $k \in \mathbb{N}$ due to Proposition \ref{p:uniformlaak}.
		The statement about the dependency of constants follows from the similar dependency the conclusion in Proposition \ref{p:uniformlaak} and the fact that the constants in the conclusion of \cite[Theorem 2.8]{Mur24} depend only on the constants involved in the assumptions.
		This dependency of constants in the conclusion of \cite[Theorem 2.8]{Mur24}  is not explicitly stated there but follows from a careful reading of the proof.
		
		The proof in Proposition \ref{p:laak}(c) shows that $(\laak(\glue,\branch)^{(k)}, 	\metlaak^{(k)})$ is quasiconvex with the constant $C_q=8$ in Definition \ref{d:qgeodesic} for all $k \in \mathbb{Z}$. This allows us to upgrade the uniform sub-Gaussian estimate \hyperlink{hke}{$\on{HKE(\sttree)}$} to the uniform \emph{full} sub-Gaussian estimate \hyperlink{hke}{$\on{HKE_f(\sttree)}$}.
	\end{proof}

	\subsection{Sub-Gaussian estimates for random walks} \label{ss:rw}

	Let $\mathbb{G}=(X,E)$ be a connected (undirected) graph with vertex set $X$ and edge set $E$. Let $d_\mathbb{G}: X \times X \to [0,\infty)$ denote the combinatorial (graph) distance.
	Let $m_\mathbb{G}$ denote the measure on $X$ defined by $m_\mathbb{G}(\{x\})=\deg(x)$ for all $x \in X$, where $\deg(x)$ is the number of neighbors of $x$ in $X$. Let $(Y_n)_{n \in \mathbb{N}}$ denote the simple random on $\mathbb{G}$.
	We denote the transition probability function by $P_n(x,y)= \mathbb{P}_x[Y_n=y]$ for all $x,y \in X, n \in \mathbb{N}$ and the (discrete time) heat kernel as 
	\begin{equation}\label{e:dis-HK}
		p_n(x,y):= \frac{P_n(x,y)}{m_\mathbb{G}(\{y\})} \quad \mbox{for all $x,y \in X, n \in \mathbb{N}$.}
	\end{equation}

	We recall the definition of sub-Gaussian heat kernel estimates for simple random walks on graphs. It is a discrete analogue of Definition \ref{d:HKE}. We use subscript $d$ or superscript $(d)$ to denote that we are in a discrete setting.
	\begin{definition} \label{d:dHKE}
		Let $V^{(d)}, \Psi^{(d)}:[1,\infty) \to [1,\infty)$ be homeomorphism such that there are constants $C_1, C_2, \beta_1, \beta_2 \in (1,\infty)$ such that 
		\begin{equation} \label{e:condscale}
			V^{(d)}(2r) \le C_1 V^{(d)}(r) \quad \mbox{for all $r \in [1,\infty)$,} \quad \mbox{and } C_2^{-1} \left(\frac{R}{r}\right)^{\beta_1} \le \frac{\Psi^{(d)}(R)}{\Psi^{(d)}(r)} \le C_2  \left(\frac{R}{r}\right)^{\beta_1}  
		\end{equation}
		for all $1 \le r < R <\diam(X,d_\mathbb{G})$. 
		
		We say that the simple random walk on a graph $\mathbb{G}= (X,E)$ satisfies full sub-Gaussian heat kernel estimates  	 
		\hypertarget{hked}{$\on{HKE_{d,f}(\Psi^{(d)})}$},  
		%	\begin{equation} \label{e:defPhiRt}
			%		\Phi(R,t) := \Phi_{\Psi}(R,t) := \sup_{r>0} \biggl(\frac{R}{r}-\frac{t}{\Psi(r)}\biggr),
			%		\qquad (R,t)\in[0,\infty)\times(0,\infty).
			%	\end{equation}
		if there exist $C_{1},c_{1},c_{2},c_{3},C_4, C_5>0$ such that for any $n \in \mathbb{N}$
		\begin{align}\label{e:duhke}
			p_n(x,y) &\le \frac{C_{1}}{m\bigl(B(x,\Psi^{-1}(t))\bigr)} \exp \left( -c_{1} t \Phi^{(d)}\left( c_{2}\frac{d_\mathbb{G}(x,y)} {t} \right) \right)
			\qquad \mbox{for all $x,y \in X$,}
		\end{align}
		and the lower bound
		\begin{align}
			p_n(x,y) + p_{n+1}(x,y)	p_t(x,y) &\ge \frac{c_{3}}{m\bigl(B(x,\Psi^{-1}(t))\bigr)} \exp \left( -C_{4} t \Phi^{(d)}\left( C_{5}\frac{d_{\mathbb{G}}(x,y)} {t} \right) \right)
			\qquad  
			\label{e:dnlhke}
		\end{align}
		for all $x,y\in X$, where $\Phi^{(d)}$ is given by 
		\begin{equation*} 
			\Phi^{(d)}(s)= \sup_{r \ge 1} \left(\frac{s}{r}-\frac{1}{\Psi(r)}\right), \quad \mbox{for all $s>0$.}
		\end{equation*}
	\end{definition}
	
	Let us first explain how our construction of the Laakso-type space $\laak(\glue,\branch)$ contains graphs as a special case. This is under the restriction that 
	\begin{equation} \label{e:laakdisc}
		\branch(k)=2, \quad \glue(k)=1, \quad \mbox{for all $k \in \mathbb{Z} \cap (-\infty,0]$.}
	\end{equation}
	In this case, the tree $\tree(\branch)$ can be viewed as the cable system of a graph (see \cite[\textsection 2]{BB04} for the notion of cable system)
	whose vertex set is 
	$\{x \in \tree(\branch): \metrictree(x,p_\tree) \in \mathbb{Z}\}$. Similarly, assuming \eqref{e:laakdisc} the Laakso-type space $\laak(\glue,\branch)$ is a cable-system of a graph whose vertex set is 
	\[
	X(\glue, \branch):= \{ x \in \laak(\glue,\branch):  \metrictree \left(\projtree(x),p_\tree\right) \in \mathbb{Z}\}
	\]
	and two distinct points $x,y \in	X(\laak,\glue)$  form an edge (that is $\{x,y\} \in E(\glue,\branch)$) if and only if $\metlaak(x,y)=1$. 
	By Proposition \ref{p:laak}(c), the graph metric $d_\mathbb{G}$ on the graph $\mathbb{G}:= \mathbb{G}(\glue,\branch)= (	X(\glue, \branch),	E(\glue, \branch))$ corresponding to the cable system is comparable to the restriction of $\metlaak$ on $X(\glue,\branch)$. 
	It is well-known that sub-Gaussian heat kernel estimates for simple random walk on graph is equivalent to sub-Gaussian heat kernel estimates for diffusion on the corresponding cable-system by arguments in \cite[\textsection 3,4]{BB04} (the arguments in \cite{BB04} also generalize when the scaling function is not necessarily of the form $r \mapsto r^\beta$). Therefore, we obtain the following theorem as a consequence of Theorem \ref{t:hkediff} and Corollary \ref{c:measlaak}.
	
	\begin{theorem} \label{t:hkegraph}
		Under the additional assumption \eqref{e:laakdisc}, the simple random walk on the graph $\mathbb{G}=\mathbb{G}(\glue,\branch)$ defined above satisfies the sub-Gaussian heat kernel estimate 	\hyperlink{hked}{$\on{HKE_{d,f}(\sttree^{(d)})}$}, where $\sttree^{(d)}=\restr{\sttree}{[1,\infty)}$.  Furthermore, there exists $C \in (1,\infty)$ such that corresponding measure $m_\mathbb{G}$ on $X(\glue,\branch)$ satisfies 
		\[
		C^{-1} V_\glue(r)V_\branch(r) \le m_\mathbb{G} \left(B_{\mathbb{G}}(x,r)\right) \le   V_\glue(r)V_\branch(r) \quad \mbox{for all $x \in X(\glue,\branch), r \ge 1$,}
		\]
		where $B_{\mathbb{G}}$ denote the balls with respect to the metric $d_\mathbb{G}$.
	\end{theorem}
	In order to obtain analogues of Theorem \ref{t:mainsuf} for random walks we need to state a version of Lemma \ref{l:elementary} whose proof is similar.
	\begin{lemma} \label{l:d-elem}
		Let $V^{(d)}, \Psi^{(d)}:[1,\infty) \to [1,\infty)$ be doubling homeomorphisms and let $C_0 \in (1,\infty)$ be such that 
		\[
		C_0^{-1} \frac{R^2}{r^2} \le \frac{\Psi^{(d)}(R)}{\Psi^{(d)}(r)} \le C_0 \frac{RV^{(d)}(R)}{rV^{(d)}(r)}, \quad \mbox{for all $1< r \le R$.}
		\] Then there exists $\branch,\glue: \mathbb{Z} \to \mathbb{Z}$ and $C_1 \in (1,\infty)$ such that 
		\[
		1 \le \inf_{k \in \mathbb{Z}} \glue(k) \le \sup_{k \in \mathbb{Z}} \glue(k) < \infty, \quad	2 \le \inf_{k \in \mathbb{Z}} \branch(k) \le \sup_{k \in \mathbb{Z}} \branch(k) < \infty,		\]
		with $\branch(l)=2, \glue(l)=1$ for all $l \le 0$, and 
		\[
		C_1^{-1} V_\glue(r) V_\branch(r) \le V^{(d)}(r) \le 	C_1  V_\glue(r) V_\branch(r), \quad C_1^{-1} \sttree(r) \le \Psi^{(d)}(r) \le C_1 \sttree(r) \quad \mbox{for all $r \ge 1$.}
		\]
	\end{lemma}
	
	The following theorem is an analogue of Theorem \ref{t:mainsuf} for random walks on infinite graphs and sequence of growing finite graphs. 
	\begin{theorem} \label{t:sufgraph}
		Let $V^{(d)}, \Psi^{(d)}:[1,\infty) \to [1,\infty)$ be doubling homeomorphisms and let $C_0 \in (1,\infty)$ be such that 
		\begin{equation} \label{e:necdisc}
			C_0^{-1} \frac{R^2}{r^2} \le \frac{\Psi^{(d)}(R)}{\Psi^{(d)}(r)} \le C_0 \frac{RV^{(d)}(R)}{rV^{(d)}(r)}, \quad \mbox{for all $1< r \le R$.}
		\end{equation}
		\begin{enumerate}[(a)] 
			\item There exists an infinite graph $\mathbb{G}=(X,E)$ such that the simple random walk on $\mathbb{G}$ satisfies the sub-Gaussian heat kernel estimate  \hypertarget{hked}{$\on{HKE_{d,f}(\Psi^{(d)})}$} and there exists $C_1 \in (1,\infty)$ such that 
			\[
			C_1^{-1} V^{(d)}(r) \le	m_{\mathbb{G}}\left(B_{d_\mathbb{G}}(x,r)\right) \le C_1 V^{(d)}(r), \quad \mbox{for all $x \in X,r \ge 1$.}
			\]
			\item There exists a point $p \in X$ such that if $\mathbb{G}_n=(X_n,E_n)$ denotes the subgraph induced by the closed ball $\overline{B}_{d_\mathbb{G}}(p,2^n)$ for all $n \in \mathbb{N}$, we have that the sequence of graphs $\mathbb{G}_n$ satisfy  the sub-Gaussian heat kernel estimate  \hypertarget{hked}{$\on{HKE_{d,f}(\Psi^{(d)})}$} uniformly and that there exists $C_2 \in (1,\infty)$ such  that 
			\[
			C_2^{-1} V^{(d)}(r) \le	m_{\mathbb{G}_n}\left(B_{d_{\mathbb{G}_n}}(x,r)\right) \le C_1 V^{(d)}(r),
			\]
			for all $x \in X_n, 1 \le r \le \diam(\mathbb{G}_n), n \in \mathbb{N}$
		\end{enumerate}
	\end{theorem}
	\begin{proof}
		The proof of part (a) is similar to that of Theorem \ref{t:hkediff} given Lemma \ref{l:d-elem} and Theorem \ref{t:hkegraph}.
		
		The proof of (b) follows from Theorem \ref{t:hkeseq} and the discussion before Theorem \ref{t:hkegraph} concerning the equivalence between sub-Gaussian heat kernel estimates for graphs and the corresponding cable systems. 
	\end{proof}
	The following remark concerns the necessity of \eqref{e:necdisc}.
	\begin{remark}
		In the setting of random walks on graphs, one can show the necessity of \eqref{e:necdiff} using the same argument as the corresponding result for diffusions in Theorem \ref{t:mainnec}. Alternately, one can deduce the necessity of \eqref{e:necdiff} by considering cable process  corresponding to  a random walk and then appealing to  Theorem \ref{t:mainnec} for the cable process.
	\end{remark}
	Similar to Proposition \ref{p:greendiff}, the random walk in Theorem \ref{t:hkegraph}(a) is transient if and only if 
	\[
	\int_1^\infty  \frac{\Psi^{(d)}(s)}{sV^{(d)}(s)} \,ds < \infty
	\]
	and the corresponding (discrete time) Green function between $x,y \in X$ comparable to 
	$$\int_{1 \vee d(x,y)}^\infty  \frac{\Psi^{(d)}(s)}{sV^{(d)}(s)} \,ds.$$
	\subsection{Martingale dimension} \label{ss:martingale}
	The goal of this subsection is to show that martingale dimension of the diffusion on the Laakso-type space $(\laak(\glue,\branch),\metlaak,\measlaak,\dflaak,\domlaak)$ is one. 
	Hino showed that the martingale dimension is equivalent to an analytic quantity called the index of the Dirichlet form \cite[Theorem 3.4]{Hin10}.
	We show that the index of $(\laak(\glue,\branch),\metlaak,\measlaak,\dflaak,\domlaak)$ is one.

	To define the index, we first recall the concept of a minimal energy-dominant measure.
	\begin{definition}[{\cite[Definition 2.1]{Hin10}}]\label{d:minimal-energy-dominant}
		Let $(X,d,m,\mathcal{E},\mathcal{F})$ be a MMD space and let $\Gamma(\cdot,\cdot)$ denote the corresponding energy measure. A $\sigma$-finite Borel measure
		$\nu$ on $X$ is called a \textbf{minimal energy-dominant measure}
		of $(\mathcal{E},\mathcal{F})$ if the following two conditions are satisfied:
		\begin{enumerate} 
			\item[(i)](Domination) For every $f \in \mathcal{F}$, we have $\Gamma(f,f) \ll \nu$.
			\item[(ii)](Minimality) If another $\sigma$-finite Borel measure $\nu'$
			on $X$ satisfies condition (i) with $\nu$ replaced
			by $\nu'$, then $\nu \ll \nu'$.
		\end{enumerate}
		Note that by \cite[Lemmas 2.2, 2.3 and 2.4]{Hin10}, a minimal energy-dominant measure of
		$(\mathcal{E},\mathcal{F})$ always exists and is precisely a $\sigma$-finite Borel measure
		$\nu$ on $X$ such that for each Borel subset $A$ of $X$, $\nu(A)=0$ if and only if
		$\Gamma(f,f)(A)=0$ for all $f\in\mathcal{F}$.
	\end{definition}
	Next, we recall the definition of index associated to a Dirichlet form. 
	\begin{definition}\cite[Definition 2.9]{Hin10}
		Let $(X,d,m,\mathcal{E},\mathcal{F})$ be a MMD space. Let $\Gamma(\cdot,\cdot)$ denote the corresponding energy measure and let $\nu$ be a minimal energy dominant measure.
		\begin{enumerate}[(i)]
			\item The \emph{pointwise index} is a measurable function $p:X \to \mathbb{N} \cup \{0,\infty\}$ such that the following hold:
			\begin{enumerate}[(a)]
				\item For any $N \in \mathbb{N}$, $f_1,\ldots,f_N \in \mathcal{F}$, we have 
				\[
				\on{rank} \left(\frac{d\Gamma(f_i,f_j)}{d\nu}(x)\right)_{1 \le i,j \le N} \le p(x) \quad \mbox{for $\nu$-almost every $x \in X$.}
				\]
				\item For any other function $p':X \to \mathbb{N} \cup \{0,\infty\}$ that satisfies (a) with $p'$ instead of $p$, then $p(x) \le p'(x)$ for $\nu$-almost every $x \in X$.
			\end{enumerate}
			\item The \textbf{index} of the MMD space $(X,d,m,\mathcal{E},\mathcal{F})$ is defined as $\nu-\esssup_{x \in X} p(x)$, where $p$ is a pointwise index.
		\end{enumerate}
		It is easy to see that the pointwise index is well-defined in the $\nu$-almost everywhere sense and does not depend on the choice of $\nu$. Therefore the index is well-defined and takes values in $\mathbb{N} \cup \{0,\infty\}$.
	\end{definition}
	The following is the  desired result concerning martingale dimension.
	\begin{prop} \label{p:index}
		The index of the   MMD space $(\laak(\glue,\branch),\metlaak,\measlaak,\dflaak,\domlaak)$ is one.
	\end{prop}
	\begin{proof}
		Recall from Proposition \ref{p:friedrich} that $	\bigcup_{n \in \mathbb{N}} \{f \circ \projlaak_{n} : f\in \mathcal{F}^{\laak_n}\cap \contfunc_0(\laak(\glue_n,\branch)) \}$ is a core, where $\glue_n$ is as defined in \eqref{e:defgn} and $\mathcal{F}^{\laak_n}$ is as defined before Lemma \ref{l:consistent}. Therefore there exists a sequence of functions $(f_k)_{k \in \mathbb{N}}$ such that $f_k \in	\bigcup_{n \in \mathbb{N}} \{f \circ \projlaak_{n} : f\in \mathcal{F}^{\laak_n}\cap \contfunc_0(\laak(\glue_n,\branch)) \}$ for all $k \in \mathbb{N}$ and such that the linear span of $(f_k)_{k \in \mathbb{N}}$ is $\dflaak_1$-dense in $\domlaak$.
		
		Let $\nu$ be a minimal energy dominant measure.
		Recalling the definitions of gradient \eqref{e:defgradlaak}, the expression for energy measure $\emlaak$ given in \eqref{e:emlaakint}, by \cite[Lemma 2.2]{Hin10} we have that $\nu \ll \lenlaak$ (recall \eqref{e:deflenlaak}) and hence
		\[
		\on{rank} \left(\frac{d\Gamma(f_i,f_j)}{d\nu}(x)\right)_{1 \le i,j \le N} =  \on{rank} \left(\frac{d\Gamma(f_i,f_j)}{d\lenlaak}(x)\right)_{1 \le i,j \le N} 
		\]
		for $\nu$-almost every $x \in X$ and all $n \in \mathbb{N}$. By \eqref{e:emlaakint}, we have 
		\[
		\on{rank} \left(\frac{d\Gamma(f_i,f_j)}{d\lenlaak}(x)\right)_{1 \le i,j \le N} =  \on{rank} \left( \nabla^\laak f_i(x) \nabla^\laak f_j(x) \right)_{1 \le i,j \le N} \le 1
		\]
		for $\lenlaak$-almost every (and hence $\nu$-almost every) $x \in X$ and $N \in \mathbb{N}$.
		By \cite[Proposition 2.10]{Hin10}, we obtain that the index is atmost one. 
		The matching lower bound on the index follows from \cite[Proposition 2.11]{Hin10}.
	\end{proof}
	
	\subsection{Concluding remarks} \label{ss:conclude}
	The upper bound on martingale dimension due to Hino \cite[Theorem 3.5]{Hin13}
	along with Proposition \ref{p:index} provides some weak evidence for the following conjecture which concerns the joint behavior of volume growth exponent, escape time exponent and martingale dimension. 
	\begin{conjecture} For any $\alpha \in [1,\infty),  \beta \in [2,\infty)$ and $d_m \in \mathbb{N}$ such that $$2 \le \beta \le  \alpha +1, \quad 1 \le d_m \le \frac{2\alpha}{\beta},$$ there exists a symmetric diffusion process on a metric measure space that satisfies full sub-Gaussian heat kernel estimate with volume growth exponent $\alpha$, escape time exponent $\beta$ and martingale dimension (or index) $d_m$.
	\end{conjecture}
	Another evidence  for the above conjecture is the existence of diffusions with $\alpha=\beta \in (2,\infty)$ and $d_m=2$ \cite[p. 4205]{Mur19}.
	Another question is to prove the necessity of the bound $d_m \le 2 \alpha/\beta$, since Hino's result applies only for self-similar spaces \cite{Hin13}. 
	
	Our diffusions provide  interesting new examples to study   the attainment problem for conformal walk dimension introduced in \cite[Problem 1.3(1)]{KM23}. There is some progress in the case of self-similar sets in \cite[\textsection 6]{KM23} but  our Laakso-type spaces are not self-similar in general.  Hence the following question will require developing  new methods to solve the attainment problem. More precisely, we formulate the following question referring the reader to \cite{KM23} for   definitions and background.
	\begin{question} \label{q:dcw}
		When does the Laakso-type MMD space $(\laak(\glue,\branch),\metlaak,\measlaak,\dflaak,\domlaak)$ attain  the conformal walk dimension? 
	\end{question}
	One obvious sufficient condition is $\sttree(r) \asymp r^2$ which is equivalent to $\# \{k : \branch(k) \neq 2\} <\infty$.  It is not clear if  attainment happens in general.
	
	Another direction of research is to consider \emph{random} variants of trees of Laakso-type spaces. For instance, the branching function $\branch:\mathbb{Z} \to \mathbb{Z}$ can be such that the $(\branch(k))_{k \in \mathbb{Z}}$ are independent and identically distributed random variables whose law $\mu_\branch$ is supported in $\llbracket 2, \infty \rrbracket$. Similarly, let us assume that the gluing function $\glue:\mathbb{Z} \to \mathbb{Z}$ is also random such that  $(\glue(k))_{k \in \mathbb{Z}}$ are independent and identically distributed random variables whose law $\mu_\glue$ is supported in $\llbracket 1, \infty \rrbracket$. If $\mu_\branch$ and $\mu_\glue$ have finite support then by Theorem \ref{t:hkediff}, the corresponding Laakso-type MMD space satisfies sub-Gaussian heat kernel bounds. The interesting case is when one or both of these measures do not have finite support as Theorem \ref{t:hkediff} will no longer apply in this case.

\noindent \textbf{Acknowledgments.} 
	This work was initiated after Martin Barlow communicated Question \ref{q:main} \cite{Bar22}. I thank him for helpful advice and for this question. I am grateful to Laurent Saloff-Coste for  asking me about the existence of  families of finite graphs with prescribed sub-Gaussian heat kernel estimates which led to the formulation of
Theorem \ref{t:sufgraph}(b) \cite{Sal23}.
I thank Xinyi Li and Izumi Okada for their interest in Theorem \ref{t:sufgraph}(a) and explaining why such graphs might be useful for the study of favorite sites for simple random walks. Finally, I thank Shashank Sharma and Aobo Chen for   discussions regarding Laakso-type spaces  and martingale dimension  respectively. The author is grateful to the two anonymous referees for a careful reading of this work and helpful suggestions.

\noindent Department of Mathematics, University of British Columbia,
Vancouver, BC V6T 1Z2, Canada. \\
mathav@math.ubc.ca 

\end{document}